\documentclass[12pt]{article}
\title{A topological proof of the Shapiro--Shapiro conjecture}

\makeatletter\def\@fnsymbol#1{\ensuremath{\ifcase#1\or \dagger\or \ddagger\or\mathsection\or \mathparagraph\or \|\or **\or \dagger\dagger \or \ddagger\ddagger \else\@ctrerr\fi}}\makeatother

\author{Jake Levinson and Kevin Purbhoo\thanks{Research supported by NSERC Discovery Grant RGPIN-04741-2018.}}

\usepackage{latexsym, amsmath, amscd, hyperref, bbm}
\usepackage[shortlabels]{enumitem}
\usepackage{amsthm}
\usepackage{graphicx, xspace, ifthen, rotating, pict2e}
\usepackage[svgnames]{xcolor}
\usepackage[charter]{mathdesign}
\usepackage{algpseudocode, algorithmicx}
\usepackage{tikz}

\usepackage{young}
\ysetshade{Blue!30}
\ysetaltshade{Green!50}
\YSetShade{Green!30}
\YSetAltShade{Purple!45}

\allowdisplaybreaks[1]
\newcommand{\bigmid}{\ \big|\ }

\newcommand{\CC}{\mathbb{C}}
\newcommand{\FF}{\mathbb{F}}

\newcommand{\QQ}{\mathbb{Q}}
\newcommand{\RR}{\mathbb{R}}

\newcommand{\calU}{\mathcal{U}}

\newcommand{\boldw}{\mathbf{w}}
\newcommand{\boldx}{\mathbf{x}}
\newcommand{\boldv}{\mathbf{v}}

\newcommand{\rmx}{\mathrm{x}}

\newcommand{\nth}{\ensuremath{^\text{th}}\xspace}
\newcommand{\Plucker}{Pl\"ucker\xspace}

\newcommand{\PGL}{\mathrm{PGL}}
\newcommand{\GL}{\mathrm{GL}}
\newcommand{\Wr}{\mathrm{Wr}}
\newcommand{\Gr}{\mathrm{Gr}}

\newcommand{\OG}{\mathrm{OG}(n,2n{+}1)}

\newcommand{\ddt}{{\textstyle \frac{d}{dt}}}

\newcommand{\imag}{\mathbbm{i}}
\newcommand{\half}{{\mathchoice
{\tfrac{1}{2}}
{\tfrac{1}{2}}
{\frac{1}{2}}
{\frac{1}{2}}
}}

\newcommand{\Spec}{\mathop{\mathrm{Spec}}}

\newcommand{\vdim}{\mathrm{vdim}}
\newcommand{\Sym}{\mathrm{Sym}}
\newcommand{\sgn}{\mathrm{sgn}}

\newcommand{\inv}{\mathrm{inv}}

\newcommand{\discriminant}{\mathrm{Disc}}

\newcommand{\SYT}{\mathsf{SYT}}
\newcommand{\Tab}{\mathsf{Tab}}
\newcommand{\MN}{\mathsf{MN}}

\newcommand{\shape}{\mathrm{shape}}
\newcommand{\rows}{\mathrm{rows}}
\newcommand{\numsyt}[1]{\mathsf{f}^{#1}}

\newcommand{\one}{{\mathchoice%
{\yng[bb][1ex](1)}
{\yng[bb][1ex](1)}
{\yng[bb][.7ex](1)}
{\yng[bb][.4ex](1)}
}}
\newcommand{\onesmall}{{\mathchoice%
{\yng[bb][1ex](1)}
{\yng[bb][1ex](1)}
{\yng[bb][.5ex](1)}
{\yng[bb][.3ex](1)}
}}
\newcommand{\hdomino}{{\mathchoice%
{\yng[bb][1ex](2)}
{\yng[bb][1ex](2)}
{\yng[bb][.5ex](2)}
{\yng[bb][.3ex](2)}
}}
\newcommand{\vdomino}{{\mathchoice%
{\yng[c][1ex](1,1)}
{\yng[c][1ex](1,1)}
{\yng[bb][.5ex](1,1)}
{\yng[bb][.3ex](1,1)}
}}
\newcommand{\twoskew}{{\mathchoice%
{\yng[c][1ex](,1,1)}
{\yng[c][1ex](,1,1)}
{\yng[bb][.5ex](,1,1)}
{\yng[bb][.3ex](,1,1)}
}}
\newcommand{\twoone}{{\mathchoice%
{\yng[c][1ex](2,1)}
{\yng[c][1ex](2,1)}
{\yng[bb][.5ex](2,1)}
{\yng[bb][.3ex](2,1)}
}}

\newcommand{\conjugate}[1]{#1^*}
\newcommand{\dualsgn}{\sgn^*}
\newcommand{\ambsgn}{\mathrm{asgn}}
\newcommand{\borelgroup}{B_+}
\newcommand{\scell}{\mathcal{X}^\lambda}
\newcommand{\scellclosure}{\smash{\overline{\mathcal{X}}}^\lambda}
\newcommand{\scellb}[1]{\mathcal{X}^{#1}}
\newcommand{\conjscell}{\mathcal{X}^{\conjugate{\lambda}}}
\newcommand{\monics}{\mathcal{P}_n}
\newcommand{\monicsb}[1]{\mathcal{P}_{#1}}
\newcommand{\richmonics}{\mathcal{P}_n^\circ}
\newcommand{\monicsclosure}{\overline{\mathcal{P}_n}}
\newcommand{\modsc}[1]{\overline{\mathcal{M}}_{0,#1}}
\newcommand{\openmodsc}[1]{\mathcal{M}_{0,#1}}
\newcommand{\sgchar}{\chi^\lambda}
\newcommand{\skewsgchar}{\chi^{\lambda/\lambda'}}
\newcommand{\conjsgchar}{\chi^{\conjugate{\lambda}}}
\newcommand{\affinespace}{\mathbf{A}}
\newcommand{\projspace}{\mathbf{P}}
\newcommand{\tspace}{\mathbf{T}}
\newcommand{\component}{\mathcal{C}}
\newcommand{\schubert}{X_\lambda}
\newcommand{\schubertopposite}{X_{\lambda^\vee}}
\newcommand{\schubertone}{X_\one}
\newcommand{\hdominoschubert}{X_\hdomino}
\newcommand{\vdominoschubert}{X_\vdomino}
\newcommand{\MNpoint}{\boldw}
\newcommand{\openrich}{\mathcal{X}^{\lambda/\lambda'}}
\newcommand{\Wrich}{\Wr_{\lambda/\lambda'}}
\newcommand{\rich}{X^\lambda_\kappa}

\newcommand{\richfamily}{E^\lambda_\kappa}
\newcommand{\richfamilyclosure}{\overline{E}}
\newcommand{\hdominovar}{Z^\lambda_\hdomino}
\newcommand{\vdominovar}{Z^\lambda_\vdomino}
\newcommand{\dominovar}{Z^\lambda_\kappa}
\newcommand{\hdominofcn}{\Phi^\lambda_\hdomino}
\newcommand{\vdominofcn}{\Phi^\lambda_\vdomino}
\newcommand{\dominofcn}{\Phi^\lambda_\kappa}
\newcommand{\critvar}{R^\lambda}
\newcommand{\discrimvar}{\varDelta_n}
\newcommand{\rd}{R}
\newcommand{\rdsmooth}{R^\mathrm{sm}}
\newcommand{\laurentC}{\CC(\!(u)\!)}
\newcommand{\laurentR}{\RR(\!(u)\!)}

\newcommand{\fpsR}{\RR[\![u]\!]}
\newcommand{\dualityiso}{\delta}
\newcommand{\maxideal}{\mathfrak{m}}
\newcommand{\finitemap}{\psi}
\newcommand{\jacobian}[1]{\partial#1}
\newcommand{\mfamily}{\smash{\overline{\mathcal{Y}}}^\lambda}
\newcommand{\richmfamily}{\smash{\overline{\mathcal{Y}}}^{\lambda/\lambda'}}
\newcommand{\openmfamily}{\mathcal{Y}^\lambda}
\newcommand{\mmap}{\Psi}
\newcommand{\mfpoint}{\textbf{\textit{y}}}
\newcommand{\mtosmap}{\pi}
\newcommand{\pol}{\mathrm{pol}}
\newcommand{\mbase}{\modsc{n+3}}
\newcommand{\openmbase}{\openmodsc{n+3}}
\newcommand{\mGr}{\widetilde{\mathrm{Gr}}(d,d+m)}
\newcommand{\tmax}{t_\mathrm{max}}
\newcommand{\mschubert}{\widetilde{X}_\lambda}
\newcommand{\mschubertlambdaprime}{\widetilde{X}_{\lambda'}}
\newcommand{\mschubertone}{\widetilde{X}_\one}

\newcommand{\mschubertopposite}{\widetilde{X}_{\lambda^\vee}}
\newcommand{\bigfibre}{Q}
\newcommand{\bigfibrefactor}[1]{Q_T^{(#1)}}
\newcommand{\mfibre}{Y}
\newcommand{\mfibrefactor}[1]{Y_T^{(#1)}}
\newcommand{\symgroup}{\mathfrak{S}}

\newenvironment{packeditemize}{
\begin{itemize}[topsep=1ex]
  \setlength{\itemsep}{0pt}
  \setlength{\parskip}{.4ex}
}{\end{itemize}}

\newtheorem{lemma}{Lemma}[section]
\newtheorem{theorem}[lemma]{Theorem}

\newtheorem{proposition}[lemma]{Proposition}
\newtheorem{corollary}[lemma]{Corollary}

\newenvironment{restatetheorem}[1]
   {\begingroup \newtheorem*{theoremx}{#1}\begin{theoremx}\em}
   {\end{theoremx}\endgroup}

\theoremstyle{definition}
\newtheorem{example}[lemma]{Example}
\newtheorem{definition}[lemma]{Definition}
\newtheorem{remark}[lemma]{Remark}

\numberwithin{equation}{section}
\numberwithin{figure}{section}
\numberwithin{table}{section}

\definecolor{DarkBlue}{rgb}{0, 0.1, 0.55}
\definecolor{DarkRed}{rgb}{0.45, 0, 0}

\newcommand{\defn}[1]{\textbf{\textit{#1}}}

\begin{document}
\maketitle

\begin{abstract}
We prove a generalization of the Shapiro--Shapiro conjecture on
Wronskians of polynomials, allowing the Wronskian to have complex
conjugate roots.  
We decompose the real Schubert cell according to the number of
real roots of the Wronski map, and define an orientation of each
connected component. For each part of
this decomposition, we prove that the topological degree of the 
restricted Wronski
map is given as an evaluation of a symmetric
group character.  In the case where all roots are real, this implies
that the restricted Wronski map is a topologically trivial covering map; 
in particular, this gives a new proof of the Shapiro--Shapiro conjecture.
\end{abstract}


\section{Introduction}

\subsection{The Shapiro--Shapiro conjecture}

In the mid-1990s, B. Shapiro and M. Shapiro formulated a striking
conjecture about the real Schubert calculus.  They considered the 
real 1-parameter family of 
flags osculating a rational normal curve,
and postulated the following: for every $0$-dimensional 
intersection of Grassmannian Schubert 
varieties defined relative to these osculating flags, 
all points of the intersection are real.  

This conjecture is now a theorem of Mukhin, Tarasov, and Varchenko
\cite{MTV1, MTV2}, and it is notable for several reasons.
For typical real flags, there is no reason to expect all points of 
a $0$-dimensional Schubert intersection to be real, and in general
they will not be.
For most flag varieties (in type A, and beyond)
it is not known whether it is even possible 
to choose flags such that the solutions to a Schubert intersection problem 
are all real.  In the Grassmannian case, this was resolved
by Vakil \cite{Vak}, 
but the Shapiro--Shapiro conjecture provides a radically different 
way of addressing this question.

It has long been known that solutions to Schubert problems can be counted using families of combinatorial objects, such as Young tableaux. It is, however, much more difficult to establish precise correspondences between these objects and individual solutions. A core obstruction to such a project is the presence of monodromy among the solutions. On one hand, standard geometric arguments involving Schubert varieties rely on Kleiman's transversality theorem \cite{Kl}, which tells us that Schubert varieties defined relative to generic flags intersect transversely. This space of flags has an intractable fundamental group and, in general, Schubert problems with these flags have large monodromy groups \cite{BV, LS, SW}. The Shapiro--Shapiro conjecture, on the other hand, uses flags defined with respect to a tuple of distinct points on $\mathbb{RP}^1$ -- data with a small fundamental group. This affords the possibility of establishing an explicit mapping. Since the resolution of the Shapiro--Shapiro conjecture in 2005 (see below), this project has been carried out successfully \cite{Pur-Gr, Sp}.

There are several equivalent formulations of the Shapiro--Shapiro conjecture, which connect it to other parts of algebraic geometry and representation theory.  It can be stated in terms of limit linear series \cite{EH}, parameterized rational curves \cite{KS}, families (of Schubert problems and of representations) over the moduli space of stable curves $\modsc{n}$ \cite{HKRW, Ry, Sp, Whi2}, and it has applications in combinatorics \cite{Pur-Gr, Pur-ribbon}, K-theory \cite{GL, Lev}, and control theory \cite{EG-deg, EG-pole, RS}. The geometry has also been generalized to other homogeneous spaces, to elliptic curves, and more \cite{CMPB, LMV, MTV3, Os, Pur-OG, Pur-LG}.

The most elementary formulation involves Wronskians of polynomials.
If $f_1, \dots, f_d$ are polynomials with 
coefficients in a field
of characteristic zero, their \defn{Wronskian} is the polynomial
\[
  \Wr(f_1, \dots, f_d) := \begin{vmatrix}
f_1 & f_1' & f_1'' & \dots & f_1^{(d-1)} \\
f_2 & f_2' & f_2'' & \dots & f_2^{(d-1)} \\
\vdots & \vdots & \vdots & & \vdots \\
f_d & f_d' & f_d'' & \dots & f_d^{(d-1)} \\
\end{vmatrix}
\,.
\]
Recall that a linear subspace $\boldv \subset \CC[z]$ is \defn{real} if $\boldv$ is invariant under complex conjugation.  Equivalently, $\boldv$ is real if $\boldv$ has a 
basis in $\RR[z]$.

\begin{theorem}[Shapiro--Shapiro conjecture]
\label{thm:MTV}
If $f_1(z), \dots, f_d(z) \in \CC[z]$ are linearly independent polynomials 
such that $\Wr(f_1, \dots, f_d)$ has only real roots, then the 
subspace of $\CC[z]$ spanned by $(f_1, \dots, f_d)$ is real.
\end{theorem}

Up to a scalar, the Wronskian depends
only on the span of $f_1, \dots, f_d$, and is zero if and only if 
$f_1, \dots, f_d$ are linearly dependent; as such, $\Wr$ defines
a map from a Grassmannian to projective space of the same dimension,
called the Wronski map.  Roughly, the connection to Schubert calculus comes
from the fact that Schubert varieties in this Grassmannian map to 
linear spaces under the Wronski map.

Motivated by the Shapiro--Shapiro conjecture, Eremenko and Gabrielov \cite{EG-deg} computed the topological degree of the Wronski map over $\RR$.  Their computation provides a lower bound for the number of real solutions to certain Schubert intersection problems. Unfortunately this result is not strong enough to deduce Theorem \ref{thm:MTV}: the lower bound on the number of real solutions is always strictly less than the number of complex solutions (except in the trivial cases, when the Grassmannian is a projective space).

Mukhin, Tarasov, and Varchenko \cite{MTV1, MTV2} have since given two proofs of Theorem~\ref{thm:MTV}. The second proof moreover establishes the transversality of Shapiro-type Schubert intersections.
Both proofs are quite complicated: they use machinery from quantum integrable systems, Fuchsian differential equations, and representation 
theory.  Their work establishes deep connections between these different areas and Schubert calculus.
However, their approach also has two major drawbacks.  First, it is heavily algebraic, using the Bethe Ansatz as a black box for solving systems 
of equations.  As such, the proof offers little geometric insight into why the Shapiro--Shapiro conjecture is true.  Second, there are several 
modifications and generalizations of the Shapiro--Shapiro conjecture which are still open
problems (see \cite{Sot-RSEG}). 
It is not known how to adapt the Mukhin--Tarasov--Varchenko 
machinery to handle these related conjectures.

The purpose of this paper is to formulate and prove a generalization of Theorem~\ref{thm:MTV}, using geometric and topological methods. We address not only the case where the roots of $\Wr(f_1, \dots, f_d)$ are all real, but all cases where $\Wr(f_1, \dots, f_d)$ has real coefficients. When the roots are all real, we obtain a new, independent, conceptually simpler proof of the Shapiro--Shapiro conjecture.
Our main theorem is a degree computation, similar in some respects to that of Eremenko and Gabrielov, but with a significant twist. The statement is less obvious, more powerful, and involves a surprising 
connection to characters of the symmetric group.

\subsection{Decomposition of real Schubert cells}

The real Wronski map is a morphism from the real Grassmannian $\Gr(d,\RR_{d+m-1}[z])$, the space of $d$-planes inside the vector space of polynomials of degree at most $d+m-1$, to projective space.  In this paper, we will mainly focus on
the restriction of this map to a Schubert cell.

Let $\monics(\RR)$ denote the set of real monic polynomials of 
degree $n$.  For a partition $\lambda \vdash n$ with $d$ parts, the real Schubert cell
$\scell(\RR)$ consists of all $d$-dimensional linear subspaces 
of $\RR[z]$ that have a basis of polynomials $(f_1, \dots, f_d)$, 
with $\deg f_i = \lambda_i+d-i$.  
The Wronskian $\Wr(f_1, \dots, f_d)$
is a polynomial of degree $n$, which, up to a scalar multiple,
is independent of the choice of basis. 
Rescaling so that $\Wr(f_1, \dots, f_d)$ is monic, we obtain a map
$\Wr : \scell(\RR) \to \monics(\RR)$; this is the Wronski map
restricted to $\scell(\RR)$.

Regarded as algebraic varieties $\monics(\RR)$ and $\scell(\RR)$ are
both isomorphic to affine space $\affinespace^n(\RR)$, and the Wronski map is
a finite morphism from $\affinespace^n(\RR)$ to itself.  The algebraic
degree of this morphism is $\numsyt\lambda = \#\SYT(\lambda)$, 
the number of standard Young
tableaux of shape $\lambda$.  This comes from a standard
calculation in the Schubert calculus: the intersection number of $n$ 
Schubert divisors with an $n$-dimensional Schubert variety.  Notably,
we also have $\numsyt\lambda = \sgchar(1^n)$, 
where $\sgchar$ is the irreducible symmetric group character associated to 
the partition $\lambda$.

Regarded as a manifold,
$\monics(\RR)$ can be further decomposed according to 
the number of real and non-real roots of polynomials.
Let $\mu = 2^{n_2} 1^{n_1}$ be a partition of $n$ with parts of
size at most $2$.   Let $\monics(\mu) \subseteq \monics(\RR)$ be the subset consisting of polynomials with $n_1$ distinct real
roots, and $n_2$ conjugate pairs of
non-real roots.  (Real roots are required to be distinct, non-real
roots are not.)  The closure, $\overline{\monics(\mu)}$, consists
of all real polynomials with at least $n_1$ real roots (not necessarily
distinct), and at least $n_2$ conjugate pairs of roots (which are no longer required to be non-real).  

\begin{example}
The polynomial $z(z^2+1)^2$ is in $\monicsb{5}(2^21)$.
The polynomial $z^3(z^2+1)$ is in 
both $\overline{\monicsb{5}(21^3)}$ and $\overline{\monicsb{5}(2^21)}$, 
but does not lie in any
$\monicsb{5}(\mu)$.  
\end{example}

Note that $\monics(\mu)$ is a contractible
open semi-algebraic subset of $\monics(\RR)$; in particular it is a
connected orientable (but as yet, not oriented) manifold.   
We orient all spaces $\monics(\mu)$ simultaneously, in the obvious way, 
by fixing an orientation of $\monics(\RR)$ and defining the
orientation of $\monics(\mu)$ to be the restriction
the orientation of the ambient space.

Define $\scell(\mu) := \Wr^{-1}(\monics(\mu))$.  
Then $\scell(\mu)$ is an open semi-algebraic 
subset of $\scell(\RR)$; hence $\scell(\mu)$ is an orientable manifold.  
However, unlike $\monics(\mu)$, $\scell(\mu)$ 
typically has many components.

The restricted 
Wronski map $\Wr : \scell(\mu) \to \monics(\mu)$ is a proper map 
of $n$-dimensional manifolds over a connected base.  
Therefore, with the additional data of an orientation of
$\scell(\mu)$, this map has a well-defined topological degree.
However, there many possible choices of orientation 
for $\scell(\mu)$ (two choices for each component), and the degree depends 
on the choice of orientation.

This brings us to our main result.  We will define an orientation of
$\scell(\mu)$, called the \emph{character orientation}.

\begin{theorem}
\label{thm:main}
With the character orientation,
the topological degree of the restricted Wronski map
$\Wr : \scell(\mu) \to \monics(\mu)$ is equal to
$\sgchar(\mu)$.
\end{theorem}

As an immediate consequence, we obtain new proofs of some previously
known results.  The first is a bound on the number of
real points in the fibre of the Wronski map.
The algebraic degree of a finite morphism gives an upper bound for the 
number of real preimages of any real point in the base; the topological
degree gives a lower bound.

\begin{corollary}
\label{cor:bounds}
For $g \in \monics(\RR)$, 
let $N_g$ be the number of real points in the 
fibre $\Wr^{-1}(g)$, counted with algebraic multiplicity.
If $g \in \overline{\monics(\mu)}$, then
\[
   |\sgchar(\mu)| \leq N_g \leq \numsyt\lambda
\,.
\]
\end{corollary}

Corollary~\ref{cor:bounds} is equivalent to a special case of 
a theorem of Mukhin and Tarasov \cite{MT-bound}, which gives
a lower bound for the number of real points in a
Schubert intersection, with respect to (not-necessarily real) flags
osculating a rational normal curve.  Such Schubert problems were 
previously studied in \cite{HHMS, HHS}, where it was observed 
that there are non-trivial restrictions on the number
of real intersection points.   Mukhin and Tarasov's theorem 
was the first to offer an explanation for some of these 
observations, using the machinery developed in \cite{MTV2}.
Our approach can also be used to recover these inequalities. In Section \ref{sec:MT-bound}, we sketch how to extend our argument to the general case using standard techniques.
The two approaches are very different --- Mukhin and Tarasov's argument is predominantly algebraic, whereas ours is topological --- and it is not clear why they produce exactly the same inequalities.

In the case where $\mu=1^n$, the lower and upper bounds in 
Corollary~\ref{cor:bounds} agree, and we
deduce the following statement, which is a formulation of the
Shapiro--Shapiro conjecture (equivalent to Theorem~\ref{thm:MTV}).

\begin{corollary}
\label{cor:cover}
$\Wr : \scell(1^n) \to \monics(1^n)$ is a topologically trivial 
covering map of degree $\numsyt\lambda$.  
\end{corollary}
 
Here, the statement about the degree is immediate from 
Corollary~\ref{cor:bounds}, and the fact that the map is topologically 
trivial follows from an argument
of Eremenko and Gabrielov (see Section~\ref{sec:transversality}
for details).
Theorem~\ref{thm:main} and Corollaries~\ref{cor:bounds} and~\ref{cor:cover}
also have straightforward generalizations in which the Schubert cell 
$\scell(\RR)$ is replaced by an open Richardson variety.
We discuss these in Section~\ref{sec:richgeneralization}.

Although Theorem~\ref{thm:main} does not imply the
strong transversality statement proved by Mukhin, Tarasov, and Varchenko 
in~\cite{MTV2}, it does imply some important special cases of it.
For example, when rephrased in terms of Schubert intersections, 
Corollary~\ref{cor:cover} asserts that a Shapiro-type
intersection of $n$ Schubert divisors with an $n$-dimensional 
Schubert variety is transverse.  There are a few other similar 
corollaries, which we discuss in Section~\ref{sec:transversality}.
A number of known applications of the Shapiro--Shapiro conjecture
depend on these special cases, but do not require the full 
power of the Mukhin--Tarasov--Varchenko transversality theorem.
For example, the geometric proof of the Littlewood-Richardson rule 
in~\cite{Pur-Gr} and
the cyclic sieving results in~\cite{Pur-ribbon} rely only on these
special cases, and are thus consequences of Theorem~\ref{thm:main}.

\subsection{Aspects of the character orientation}

The key to Theorem~\ref{thm:main} and the main novelty of this paper
is the definition of the character orientation.
In Section~\ref{sec:example}, we explain how this definition works
in the special case where $\lambda = (n-1,1)$: this case is particularly
accessible and will serve as a recurring example throughout the paper.
In the general case, the definition requires additional background on the 
Wronski map, and will be given in Section~\ref{sec:orientation}.
For now, we mention three features that are partially inferrable from
the statement of Theorem~\ref{thm:main}.

\begin{enumerate}
\item
Although $\scell(\mu)$ is an
open submanifold of $\scell(\RR)$, in general the character orientation of 
$\scell(\mu)$ is not a restriction of some orientation of the
ambient space $\scell(\RR)$.  If it were,
then the degree of $\Wr : \scell(\mu) \to \monics(\mu)$ 
would be equal to the degree of the total map 
$\Wr : \scell(\RR) \to \monics(\RR)$, and hence would be independent
of $\mu$.  Instead, the definition of the character orientation
begins by choosing an orientation on ambient space $\scell(\RR)$; we then twist the orientation by a certain real regular function. As such, some components of $\scell(\mu)$ will be oriented the same as the ambient orientation, and some will be oriented opposite to it.

\item
The character orientation is defined globally on each
component of $\scell(\mu)$, rather than locally.  This is more or
less mandatory, because \textit{a priori}, we do not know enough about
the topology of $\scell(\mu)$ to make any kind of local to global arguments.

\item
The character orientation exhibits a kind of skew-symmetry with 
respect to Grassmann duality.  
If $\conjugate{\lambda}$ denotes the 
conjugate partition to $\lambda$, there is
an isomorphism
$\dualityiso : \conjscell(\RR) \to \scell(\RR)$ such that 
$\Wr \circ \dualityiso$ is the Wronski map on $\conjscell(\RR)$.  This 
restricts to a diffeomorphism $\dualityiso : \conjscell(\mu) \to \scell(\mu)$.
Algebraically, $\scell(\mu)$ and $\conjscell(\mu)$ are indistinguishable.  
Yet according to Theorem~\ref{thm:main}, the Wronski map 
can have different topological degrees on these two spaces.
The difference is only
a sign, since $\conjsgchar(\mu) = (-1)^{n_2} \sgchar(\mu)$,
but it is not a global sign 
(it depends on $\mu$).
This seems bizarre, as it implies that the orientation cannot depend only on the abstract geometry of the Wronski map. The situation is reconciled by the fact that there is a second orientation 
on $\scell(\mu)$ called the \emph{dual character orientation}, 
which is interchanged with the character orientation under $\dualityiso$.  We show
that the two orientations coincide on $\scell(\mu)$ if $n_2$ is even, and are opposites if $n_2$ is odd, which explains the
signs.

\end{enumerate}

Once we have formulated this definition, our proof of Theorem~\ref{thm:main} proceeds along the following lines. For each $\mu$ we choose a polynomial $h_\mu(z) \in \monics(\mu)$, such that we can identify all points in the fibre $\Wr^{-1}(h_\mu)$.  We label these points by tableaux. The topological degree of the map $\Wr : \scell(\mu) \to \monics(\mu)$ is then a signed count of points in the fibre (the sign is positive if the Wronski map is locally orientation preserving, and negative otherwise). To compute the signs, we connect up all of these points by a network of paths in $\scell(\RR)$, and count the number of sign changes along each path. Since the points are labelled by tableaux, we are left with a problem of counting certain tableaux with signs. We recognize this enumeration problem as a case of the Murnaghan--Nakayama rule, which gives us the answer as a character evaluation.

\subsection{An example}
\label{sec:example}

We illustrate Theorem~\ref{thm:main} with an elementary example:
the case where $\lambda = (n-1,1)$.  Here, to compute 
$\Wr^{-1}(g)$ for $g \in \monics(\RR)$,
we are looking for polynomials $(f_1, f_2)$ such that
$\deg(f_1) = n$, $\deg(f_2) = 1$ and
\begin{equation}
\label{eqn:elementary}
\Wr(f_1,f_2) = f_1f_2' - f_1'f_2 = g
\,.
\end{equation}
Two solutions to this equation represent the same 
point of $\scell(\RR)$, if they are linearly equivalent, i.e.
they are bases for the same subspace of $\RR[z]$.

We claim that $\Wr^{-1}(g) \subset \scell(\RR)$ can be identified
with the critical points of $g$.
To see this, take derivatives of both sides of the equation 
\eqref{eqn:elementary}, which gives $f_1f_2'' - f_1''f_2 = g'$.  
Since $f_2$ is a linear polynomial, $f_2'' = 0$, and we obtain
\[
-f_1'' f_2 = g'
\,.
\] 
It follows that all solutions to \eqref{eqn:elementary}
are of the form where $f_2(z) = z-c$, $g'(c) = 0$.
Furthermore, it is easy to check that if $g'(c) = 0$, then up to linear 
equivalence, there is a unique polynomial $f_1$ such that 
$\Wr(f_1(z), z-c) = g(z)$.  Thus for $\lambda = (n-1,1)$,
we identify $\scell(\RR)$ with
\[
   \{(g,c) \in \monics(\RR) \times \affinespace^1(\RR) \mid g'(c) = 0\} 
\,.
\]
The Wronski map
is identified with the projection onto the first factor.

From this description, we see immediately that 
Corollary~\ref{cor:cover} is true for $\lambda = (n-1,1)$: if 
$g$ is a degree $n$ polynomial with 
$n$ distinct real roots then $g$ has 
$n-1 = \numsyt\lambda$ distinct real critical points.  
There are $n-1$ components of $\scell(1^n)$: the $i$\nth 
component is the set of pairs
$(g,c)$ such that $g \in \monics(1^n)$, and $c$ is the unique critical 
point of $g$
between the $i$\nth and $(i+1)$\nth smallest roots of $g$; 
this clearly maps diffeomorphically to $\monics(1^n)$.

Before we consider Theorem~\ref{thm:main}, let us first compute
the topological degree of the full map 
$\Wr : \scell(\RR) \to \monics(\RR)$.
The Wronski map fails to be locally one-to-one in a neighbourhood of $(g,c)$
when $c$ is a double or higher order critical point of $g$, i.e. 
when $g'(c) = g''(c) = 0$.  Thus we see that the Jacobian of the 
Wronski map is, up to a scalar,  the function 
$\jacobian\Wr :\scell(\RR) \to \RR$, 
$\jacobian\Wr(g,c) = g''(c)$ (here, the partial derivatives in the Jacobian
are computed with respect to affine coordinates on
$\scell(\RR)$ and $\monics(\RR)$).
It follows that we can find orientations of $\scell(\RR)$ and 
$\monics(\RR)$ such that the Wronski map is locally orientation 
preserving at $(g,c) \in \scell(\RR)$ if and only if $(-1)^n g''(c) > 0$.
We call these the \emph{ambient orientations}.
(The global sign $(-1)^n$ is not necessary for this example in isolation, but ensures that
the Wronski map is locally orientation preserving at 
$(g,c_0)$, where $c_0$ is the smallest critical point of $g$,
in accordance with conventions used throughout this paper.)

To compute the topological degree of the Wronski map with respect to
the ambient orientations, we can pick any
$g \in \monics(\RR)$ with distinct real critical points, and
count the critical points of $g$ with signs: $c$ is counted positively
if $(-1)^n g''(c) > 0$, and negatively if $(-1)^n g''(c) < 0$  
(see Figure~\ref{fig:critpointsigns}, left).  These signs must
alternate, so the degree is $1$ if $n$ is even, and $0$ if $n$ is odd.
This agrees with Eremenko and Gabrielov's result 
(see Section~\ref{sec:EG-deg}).
\begin{figure}
\centering
\newcommand{\xmin}{-2.7} \newcommand{\xmax}{3.3}
\newcommand{\ymin}{-2.8} \newcommand{\ymax}{2.8}
\begin{tikzpicture}[x=2em,y=2em]
   \draw[violet] (\xmin,0) -- (\xmax,0);
   \draw[violet] (.3,\ymax) -- (.3,\ymin);
   \clip  (\xmin,\ymin) rectangle (\xmax,\ymax);
   \draw[samples=200,domain=\xmin:\xmax,smooth,variable=\z,blue,thick] 
     plot ({\z},{.05*\z*\z*(((((\z-3.61667)*\z-6.664)*\z+27.965)*\z+8.77333)*\z-45.08)-.5});
	\node at (-2,2.3) {\small $+$};
	\node at (-1,-1.9) {\small $-$};
	\node at (0,-.25) {\small $+$};
	\node at (1,-1.6) {\small $-$};
	\node at (2.3,1.1) {\small $+$};
	\node at (2.8,0.2) {\small $-$};
\end{tikzpicture}
\qquad \qquad
\begin{tikzpicture}[x=2em,y=2em]
   \draw[violet] (\xmin,0) -- (\xmax,0);
   \draw[violet] (.3,\ymax) -- (.3,\ymin);
   \clip  (\xmin,\ymin) rectangle (\xmax,\ymax);
   \draw[samples=100,domain=\xmin:\xmax,smooth,variable=\z,blue,thick] 
     plot ({\z},{.05*\z*\z*(((((\z-3.61667)*\z-6.664)*\z+27.965)*\z+8.77333)*\z-45.08)-.5});
	\node at (-2,2.3) {\small $+$};
	\node at (-1,-1.9) {\small $+$};
	\node at (0,-.25) {\small $-$};
	\node at (1,-1.6) {\small $+$};
	\node at (2.3,1.1) {\small $+$};
	\node at (2.8,0.2) {\small $-$};
\end{tikzpicture}
\caption{The signs of the critical points of a polynomial 
in $\monicsb{7}(2^21^3)$, with respect to the
ambient orientation (left), and the character orientation (right).}
\label{fig:critpointsigns}
\end{figure}

Now, let $\component$ be a component of $\scell(\mu)$.
The \emph{character orientation} of
$\component$ is defined to be consistent with the ambient orientation if
$(-1)^{n-1} g(c) > 0$ for all $(g,c) \in \component$, and 
opposite to the ambient 
orientation if $(-1)^{n-1} g(c)<0$ for all $(g,c) \in \component$.  
One of these two conditions must hold,
because we cannot have $g(c) = 0$ for $(g,c) \in \component$, or
else $c$ would be a double real root of $g$.
Thus, the Wronski map is locally orientation preserving 
at $(g,c)$ with respect to the character orientation if and only if 
$- \frac{g''(c)}{g(c)} > 0$.  

To compute the topological degree with respect to the character orientation,
we pick any $g \in \monics(\mu)$ with distinct
critical points and count the critical points of $g$ with signs.  This
time, $c$ is counted positively if $- \frac{g''(c)}{g(c)} > 0$, and
negatively if $- \frac{g''(c)}{g(c)} < 0$
(see Figure~\ref{fig:critpointsigns}, right).  
For example, if $n_1 > 0$, we can take $g$ to be a polynomial with
$n_1$ distinct real roots, and $n_1-1$ distinct real critical points,
which will all be counted with positive signs.
If $n_1 = 0$, we can take $g$ to have no real roots and $1$ real critical
point, which is counted with a negative sign.  In either case, we
see that the topological degree is $n_1-1 = \sgchar(2^{n_2}1^{n_1})$, 
in agreement with Theorem~\ref{thm:main}.

We will revisit this example several times in
Sections~\ref{sec:preliminaries} and \ref{sec:orientation},
to illustrate other aspects of Theorem~\ref{thm:main}.

\begin{remark}
For every partition $\lambda$ and $g \in \monics(\RR)$, 
there is a function $\Theta^\lambda_g$ called the \emph{master function},
whose critical points are identified with $\Wr^{-1}(g)$.
This fact is the starting 
point for Mukhin, Tarasov and Varchenko's proof of 
Theorem~\ref{thm:MTV} in \cite{MTV1}.
In general, $\Theta^\lambda_g$ is rational function of several 
variables, and the analysis of its critical points
cannot be carried out in an elementary way;
$\lambda = (n-1,1)$ is exceptionally nice because it is the only
case in which $\Theta^\lambda_g$ is a univariate function.
Our proof of Theorem~\ref{thm:main} is not based on the master function,
but instead generalizes this example in a different way.
\end{remark}

\subsection{Outline}

Section~\ref{sec:preliminaries} begins with an overview of the main 
properties of the Wronski map that will be needed throughout the paper.  
These include the connection with Schubert 
varieties, and an explicit formula for the Wronski map in coordinates.  
We then use these properties to compute certain points in the 
fibre $\Wr^{-1}(g)$, in two important special cases.  We recall
the statement of the Murnaghan--Nakayama rule, and
state a lemma which labels the real points
in certain ``special fibres'' of the Wronski map by 
Murnaghan--Nakayama tableaux (Lemma~\ref{lem:specialfibres}).

In Section~\ref{sec:orientation}, we define the ambient orientation
of $\monics(\RR)$ and $\scell(\RR)$, the character orientation 
of $\scell(\mu)$, and the dual character orientation.  For each
of these orientations, we consider how the sign of a point in $\scell(\RR)$ changes when traversing certain paths.  We state two lemmas, which
provide a collection of paths connecting up all of the real points in the 
aforementioned special fibres of the Wronski map 
(Lemmas~\ref{lem:MNsigns} and~\ref{lem:SYTsigns}).  
Since the sign changes along these paths are predictable, we can
compute the signs of all points in each of our special fibres.  
Theorem~\ref{thm:main} follows from this computation.

To complete the argument, we still need to prove
Lemmas~\ref{lem:specialfibres}, \ref{lem:MNsigns} 
and~\ref{lem:SYTsigns}.  This is the goal of 
Section~\ref{sec:stablecurves}.  We work with Speyer's model of
the Shapiro--Shapiro conjecture \cite{Sp}, in which the Wronski map is 
replaced by a related family over the moduli space of genus zero stable 
curves.  The big advantage of this formulation is that it has a rich 
boundary structure, 
on which explicit fibre calculations can be carried out with relative ease.
This allows us to analyze fibres and paths by degeneration arguments.
We begin Section~\ref{sec:stablecurves},
by reviewing Speyer's construction, and explaining how to compute fibres 
of this map over stable curve that is a $\projspace^1$-chain.  We prove
Lemma~\ref{lem:specialfibres} by replacing polynomials with appropriate
curves, and allowing these curves to degenerate to $\projspace^1$-chains.
The assertions of the lemma essentially translate into properties of 
this degeneration.  The paths in Lemmas~\ref{lem:MNsigns} 
and~\ref{lem:SYTsigns} are constructed and analyzed using similar ideas.

In Section~\ref{sec:bounds} we discuss Corollary~\ref{cor:bounds} and
other related results.  We show that
in several cases, the lower bound in Corollary~\ref{cor:bounds} is
tight.  We use the results of Section~\ref{sec:orientation} to give a quick
proof of the Eremenko--Gabrielov lower bound, and compare 
these results.
We also discuss the relationship between Corollary~\ref{cor:bounds} and 
the Mukhin--Tarasov bound in more detail.

We conclude with some generalizations of our main results and open 
questions, in Section~\ref{sec:conclusion}. 

\begin{remark}
Many of the geometric objects we are considering have a dual 
nature as algebraic varieties/schemes over $\RR$ and differentiable 
manifolds, and we will often move back and forth between these points of
view.
For the most part, topological statements (involving orientations,
paths, continuity, etc.) use the analytic topology.  
When we talk about fibres of 
a morphism over $\RR$, we normally mean the scheme theoretic fibre.
It should hopefully be clear from context how to interpret these 
types of statements.
\end{remark}


\section{Preliminaries}
\label{sec:preliminaries}

\subsection{The Wronski map}

We begin this section by recalling some of the fundamental properties 
of the Wronski map.  We omit proofs of results that are
fairly well established, and refer the reader to \cite{Pur-Gr},
\cite{Sot-F} or \cite{Sot-RSEG} for further details or additional 
background.

Let $\FF$ be a field of characteristic zero.  We denote the 
vector space of polynomials of degree at most $\ell$ over $\FF$ by
$\FF_\ell[z]$.  The \defn{Wronski map}
$\Wr : \Gr(d,\FF_{d+m-1}[z]) \to \projspace(\FF_{dm}[z])$, maps
a $d$-plane spanned by polynomials $f_1, \dots, f_d$ to the line
in $\FF_{dm}[z]$ spanned by $\Wr(f_1, \dots, f_d)$.

When the choice of field $\FF$ is not relevant to discussion at hand,
we will sometimes suppress it from our notation.  In this case, we may
also write the Wronski map as
$\Wr : \Gr(d,d+m) \to \projspace^{dm}$.  

The group $\GL_2(\FF)$ acts on $\FF_\ell[z]$ by M\"obius transformations.
If $\phi = 
\left(\begin{smallmatrix} 
\phi_{11} & \phi_{12} \\ \phi_{21} & \phi_{22}
\end{smallmatrix}\right) \in \GL_2(\FF)$ and
$f(z) \in \FF_\ell[z]$,
the action is given by
\[
\phi f(z) := (\phi_{21} z + \phi_{11})^\ell 
f\Big(\frac{\phi_{22} z + \phi_{12}}{\phi_{21} z + \phi_{11}}\Big)
\,.
\]
This induces a $\PGL_2(\FF)$ action on 
$\Gr(d,\FF_{d+m-1}[z])$ and $\projspace(\FF_{dm}[z])$, and the Wronski map
is $\PGL_2(\FF)$-equivariant with respect to these actions.

We define a family of flags over $\projspace^1(\FF) = \FF \cup \{\infty\}$:
\[
  F_\bullet(a) \ :\ 
  F_0(a) \subset F_1(a) \subset \dots \subset F_{d+m}(a)
\,.
\]
For $a \in \FF$,
$F_i(a) := (z+a)^{d+m-i}\FF[z] \cap \FF_{d+m-1}[z]$ is the subspace of
polynomials in $\FF_{d+m}[z]$ divisible by $(z+a)^{d+m-i}$,
$i=0, \dots, d+m$.
We also set
$F_i(\infty) := \FF_{i-1}[z] = \lim_{a \to \infty} F_i(a)$.
We note that $\phi(F_\bullet(a)) = F_\bullet(\phi(a))$ for
$\phi \in \PGL_2(\FF)$.

Let $\lambda = (\lambda_1, \dots, \lambda_d)$ be a partition, with
$m \geq \lambda_1 \geq \dots \geq \lambda_d \geq 0$.
Then $\lambda$ indexes a \defn{Schubert cell} relative to the flag
$F_\bullet(a)$:
\[
  \schubert^\circ(a) := \big\{ \boldx \in \Gr(d,d{+}m) \bigmid 
   \dim (\boldx \cap F_j(a))  - \dim (\boldx \cap F_{j-1}(a))  = 
        \eta_j,\   j = 1, \dots, d{+}m\}
\,,
\]
where $\eta_j = 1$ if $j = m+i - \lambda_i$ for some $i$, 
and $\eta_j = 0$ otherwise.
Its closure 
\[
  \schubert(a) := \overline{\schubert^\circ(a)}
\]
is the \defn{Schubert variety}. These conventions are such that $|\lambda|$ is the codimension of
$\schubert(a)$ in $\Gr(d,d+m)$.  

We will often identify the partition $\lambda$ with its diagram, 
$\lambda = \{(i,j) \mid 1 \leq i \leq d,\, 1 \leq j \leq \lambda_i\}$,
which is represented pictorially as an array of $|\lambda|$ boxes,
with $\lambda_i$ boxes in row $i$.  We will write
$\schubertone(a)$ to denote the Schubert variety associated to the 
partition $\one = (1,0,\dots,0)$, $\hdominoschubert(a)$ for the
partition $\hdomino = (2,0, \dots, 0)$, etc.

The relationship between these
Schubert varieties and the Wronski map is given by the following
lemma.

\begin{lemma}
\label{lem:schubertwronskian}
Let  $\boldx \in \Gr(d,d+m)$, and let $g = \Wr(\boldx)$.
\begin{enumerate}[(i)]
\item
For $a \in \FF$,
$(z+a)^\ell$  divides $g(z)$
if and only if
$\boldx \in \schubert(a)$ for some partition $\lambda \vdash \ell$.
If $(z+a)^\ell$ is
the largest power of $(z+a)$ that divides $g(z)$, then
$\lambda$ is unique and moreover $\boldx \in \schubert^\circ(a)$.

\item
$\deg(g) \leq dm-\ell$ if and only if $\boldx \in \schubert(\infty)$ for
some $\lambda \vdash \ell$.
If $\deg(g) = dm-\ell$, then in fact $\lambda$ is unique and moreover $\boldx \in \schubert^\circ(\infty)$.
\end{enumerate}
\end{lemma}

It follows that Schubert varieties of the form $\schubert(a)$ intersect
properly.

\begin{lemma}
\label{lem:generalschubertintersection}
Let $a_1, \dots, a_k \in \projspace^1$ be distinct, and let
$\alpha^1, \dots, \alpha^k$ be partitions such that 
$\sum_{i=1}^k |\alpha^i| = dm$.  the intersection
\[
   X_{\alpha^1}(a_1) \cap \dots \cap X_{\alpha^k}(a_k)
\]
is proper and hence
a finite scheme of length equal to the Schubert intersection number
$\int_{\Gr(d,d+m)} [X_{\alpha^1}] \dots [X_{\alpha^k}]\,$.
\end{lemma}

The Wronski map 
$\Wr: \Gr(d,d+m) \to \projspace^{dm}$ is a finite morphism \cite{EH}.  
Its degree is the length of the finite scheme
$\Wr^{-1}(g)$, for any $g \in \projspace^{dm}$.  This is independent
of $\FF$. If $\FF$ is 
algebraically closed, then we can assume $g(z) = \prod_{i=1}^{dm} (z+a_i)$,
with distinct roots.
Lemma~\ref{lem:schubertwronskian} shows that the fibre is an intersection
of $dm$ Schubert divisors on $\Gr(d,d+m)$,
\[
   \Wr^{-1}(g) = \schubertone(a_1) \cap \dots \cap \schubertone(a_{dm})
\,,
\]
and hence the degree is $\#\SYT(m^d)$.

The \defn{\Plucker coordinates} on the Grassmannian are homogeneous 
coordinates $[\rmx_\lambda]$
indexed by the same partitions $\lambda$
as the Schubert varieties.  For $\boldx \in \Gr(d,d+m)$, choose a basis
$(f_1, \dots, f_d)$, and let
$M$ be the $d \times (m + d)$ matrix $M_{ij} = f_i^{(j-1)}(0)$.
Then $\rmx_\lambda$ is the maximal minor of $M$, with columns
$1+\lambda_d, 2+\lambda_{d-1}, \dots, d+ \lambda_1$.
\begin{proposition}
\label{prop:wronskiplucker}
In terms of \Plucker
coordinates, the Wronski map is 
\[
   \Wr(\boldx; z) = \sum_{\ell=0}^{dm} 
         \sum_{\lambda \vdash \ell} 
             \numsyt\lambda \rmx_\lambda \frac{z^\ell}{\ell!} 
\,.
\]
\end{proposition}

\subsection{The Schubert cell $\scell$}

The \defn{complementary partition} to $\lambda$ is 
$\lambda^\vee := (m -\lambda_d, \dots, m-\lambda_1)$.  The
Schubert cell $\schubertopposite^\circ(\infty) \subset \Gr(d,d+m)$ 
will play a special
role, and we denote it by $\scell$.  The use of a superscript is to 
indicate that $|\lambda|$ is the dimension of $\scell$, rather than 
the codimension.
Concretely, $\scell(\FF)$ is
the variety of $d$-planes in $\FF[z]$
that have a basis $(f_1, \dots, f_d)$,
with $\deg f_i = \lambda_i+d-i$.   We note that this characterization
is independent of $m$.  

Now, fix $\lambda \vdash n$, and suppose 
$\kappa = (\kappa_1, \dots, \kappa_d)$ is another partition.  We
write $\kappa \subset \lambda$ if $\kappa_i \leq \lambda_i$ for all $i$.
For $a \in \affinespace^1$, the Schubert variety $X_\kappa(a)$
intersects the Schubert cell $\scell$ non-trivially if and only if 
$\kappa \subset \lambda$.  The intersection
\[
  \rich(a) := X_\kappa(a) \cap \scell
\]
is a \defn{half-open Richardson variety}.

When we restrict the Wronski map to $\scell$, the properties of
the previous section translate into the following facts:
\begin{enumerate}[(i)]
\item The algebraic image of the Wronski map restricted to $\scell(\FF)$ 
is the subvariety of $\projspace(\FF_{dm}[z])$
of polynomials of degree exactly $n$.
 Rescaling so that the leading coefficient is $1$, we identify this image
with $\monics(\FF)$, the affine space of monic polynomials of degree $n$ 
in $\FF[z]$.
\item The subgroup $\borelgroup \subset \PGL_2(\FF)$ of upper triangular matrices 
acts on $\scell(\FF)$ and $\monics(\FF)$ by affine transformations, 
and the Wronski map is $\borelgroup$-equivariant.
\item If $\boldx \in \scell(\FF)$, and $g = \Wr(\boldx)$, then $(z+a)^\ell$
divides $g(z)$ if and only if $\boldx \in \rich(a)$ for
some $\kappa \vdash \ell$.
\item The map $\Wr : \scell(\FF) \to \monics(\FF)$ is a finite morphism
of affine varieties of degree $\numsyt\lambda$.
In this case, the degree computation follows from the fact that the fibre 
over $g(z) = \prod_{i=1}^n (z+a_i)$ is
\[
   \Wr^{-1}(g) = \schubertone(a_1) \cap \dots \cap \schubertone(a_n) \cap 
    \schubertopposite(\infty)
\,.
\]
\item For $\boldx \in \scell(\FF)$, the \Plucker coordinates 
satisfy $\rmx_\kappa = 0$ for $\kappa \not \subset \lambda$, and can be 
normalized so that $\rmx_\lambda = 1$.
In terms of normalized \Plucker coordinates, the Wronski map
$\Wr : \scell(\FF) \to \monics(\FF)$ is
\begin{equation}
\label{eqn:wronskimapcoords}
    \Wr(\boldx; z) = z^n + 
     \Big(\frac{\numsyt\lambda}{n!}\Big)^{-1} \sum_{\ell=0}^{n-1} \sum_{\kappa \vdash \ell}
       \numsyt\kappa \rmx_\kappa \frac{z^\ell}{\ell!}
\,.
\end{equation}
\end{enumerate}

The Schubert cell $\scell(\FF)$ is isomorphic to affine space 
$\affinespace^n(\FF)$.  Explicitly, a point
$\boldx \in \scell(\FF)$ has a unique basis of
polynomials $(f_1, \dots, f_d)$, of the form
\begin{equation}
\label{eqn:affinecoords}
   f_i(z) = \frac{z^{\lambda_i+d-i}}{(\lambda_i+d-i)!} + 
\sum_{j=1}^{\lambda_i}
   (-1)^{i + \conjugate{\lambda}_j} 
   \frac{z^{j-\conjugate{\lambda}_j+d-1}}{(j-\conjugate{\lambda}_j+d-1)!}
   \cdot x_{ij}
\,,
\end{equation}
where $\conjugate{\lambda}$ denotes the conjugate partition.
The coefficients $(x_{ij})_{(i,j) \in \lambda}$ 
of these polynomials give the \defn{affine coordinates} 
of the point $\boldx$.
The coordinate ring of $\scell(\FF)$ is
$\FF[\boldx] := \FF[x_{ij} \mid (i,j) \in \lambda]$.

\begin{example} For $\lambda = 532$, here are the polynomials specified in 
\eqref{eqn:affinecoords}.
\[
{\setlength\arraycolsep{1pt}
\begin{matrix}
f_1(z) &~=~
&+\ x_{11} 
&+\ x_{12} z 
& & -\ x_{13}\frac{z^3}{3!} 
& & +\ x_{14}\frac{z^5}{5!} 
& +\ x_{15}\frac{z^6}{6!} 
& +\ \frac{z^7}{7!} 
\\
f_2(z) &~=~
& -\ x_{21} 
&-\ x_{22} z 
& & +\ x_{23}\frac{z^3}{3!} 
& +\ \frac{z^4}{4!} & & &
\\
f_3(z) &~=~
& +\ x_{31} 
& +\ x_{32} z 
& +\ \frac{z^2}{2!} & &  & & &
\\
\end{matrix}}
\]
\end{example}

\begin{remark}
The precise signs and constants in \eqref{eqn:affinecoords} are 
not too important for most practical purposes.  They are
chosen so that our affine coordinates
are a subset of the normalized \Plucker coordininates.
This has the additional benefit that the coordinates are 
well-behaved under Grassmann duality.
The duality isomorphism $\dualityiso : \conjscell
\to \scell$ is simply defined by $x_{ij} \mapsto x_{ji}$ in
affine coordinates.  Using Proposition~\ref{prop:wronskiplucker}, one
can show that if $\Wr : \scell \to \monics$ is the Wronski
map on $\scell$, then 
$\Wr \circ \dualityiso : \conjscell \to \monics$ 
is the Wronski map on $\conjscell$. 
\end{remark}

\begin{proposition}
\label{prop:richaffine}
Suppose $\lambda/\kappa$ has at most one box in each column, or at most
one box in each row.  Then
$\rich(0)$ is the affine subspace of $\scell$ defined in affine coordinates
by $x_{ij} = 0$ for $(i,j) \in \kappa$.
\end{proposition}

\begin{proof}
$\rich(0)$ is defined in \Plucker coordinates by $\rmx_\alpha = 0$
for all $\alpha \subset \lambda$ such that $\kappa \not \subset \alpha$.
In the case where $\lambda/\kappa$ has at most one box in each column
(or each row), the equations  $x_{ij} = 0$, $(i,j) \in \kappa$ are
a subset of these defining equations.  This subset cuts out an affine space
$V$ of dimension $|\lambda/\kappa|$, and so $\rich(0)$ is a closed
subvariety of $V$.
Since $\dim \rich(0) = |\lambda/\kappa|$, we have $\rich(0) = V$.
\end{proof}

By writing the Wronski map in affine coordinates, we can solve some
specific 
instances of the
inverse Wronskian problem by direct calculation. The most important
examples of this are given in the next two lemmas.

\begin{lemma}
\label{lem:calculation1}
Let $\kappa\subset \lambda$ be a partition such that $|\kappa| = n-1$.
If $g(z) = z^{n-1}(z+a) \in \monics(\RR)$, then there is a unique (reduced)
point 
$\boldx \in \Wr^{-1}(g) \cap \rich(0)$ 
and $\boldx$ is real.
\end{lemma}

\begin{proof}
Suppose the unique box of $\lambda/\kappa$ is in row $i_1$, and column $j_1$.
By Proposition~\ref{prop:richaffine}, $\rich(0)$ is defined by 
$x_{ij} = 0$, for $(i,j) \neq (i_1, j_1)$.  
Thus, \eqref{eqn:wronskimapcoords} simplifies to
\[
    \Wr(\boldx; z) = z^n + 
     \Big(\frac{\numsyt\lambda}{n!}\Big)^{-1} \cdot
       \numsyt\kappa x_{i_1j_1} \frac{z^{n-1}}{(n-1)!}
\,.
\]
Thus the unique solution to $\Wr(\boldx) = g$ is given in affine coordinates
by 
$x_{i_1j_1} = \frac{a \numsyt\lambda}{n \numsyt\kappa}$, $x_{ij} = 0$
for $(i,j) \neq (i_1, j_1)$.
\end{proof}

The \defn{distance} between two boxes $(i_1,j_1)$ and $(i_2,j_2)$ in the 
diagram of $\lambda$ is defined to be $|i_1-i_2|+|j_1-j_2|$.

\begin{lemma}
\label{lem:calculation2}
Let $\kappa\subset \lambda$ be a partition such that $|\kappa| = n-2$.
Let $L$ be the distance between the
two boxes of the skew shape $\lambda/\kappa$.
Let $g(z) = z^{n-2}(z+a_1)(z+a_2) \in \monics(\RR)$ (hence, 
$a_1, a_2 \in \RR$ or $a_1 = \overline{a_2} \in \CC$).
\begin{enumerate}[(i)]
\item 
If $L =1$, then there is a unique (reduced) point 
$\boldx \in \Wr^{-1}(g) \cap \rich(0)$
and $\boldx$ is real.
\item
If $L >1$, then $\Wr^{-1}(g) \cap \rich(0)$ is a finite scheme of length
two.  The two points $\boldx, \boldx' \in \Wr^{-1}(g) \cap \rich(0)$ are 
identified with solutions to to a quadratic
equation with discriminant $(a_1+a_2)^2 - 4(1-L^{-2})a_1a_2$. Hence:
\begin{packeditemize}\setlength{\parskip}{1ex}
\item
If $(a_1+a_2)^2 - 4(1-L^{-2})a_1a_2 > 0$, then $\boldx, \boldx'$ are distinct and real.  
\item
If $(a_1+a_2)^2 - 4(1-L^{-2})a_1a_2 = 0$, then
$\boldx=\boldx'$ is a double real point.
\item
If $(a_1+a_2)^2 - 4(1-L^{-2})a_1a_2 < 0$, then
$\boldx, \boldx'$ are not real.
\end{packeditemize}
\end{enumerate}
\end{lemma}

\begin{proof}
The proof of (i) is similar to Lemma~\ref{lem:calculation1}, and we omit it.
For (ii), suppose that the positions of the two boxes of $\lambda/\kappa$ are 
$(i_1,j_1)$ and $(i_2,j_2)$.  Since $L > 1$, these are both corners of
$\lambda$.  Let $\beta^1$ and $\beta^2$ denote the partitions
of size $n-1$, obtained by deleting corners $(i_1, j_1)$ 
and $(i_2,j_2)$ from $\lambda$ respectively.
Proceeding as in the proof Lemma~\ref{lem:calculation1},
$\rich(0)$ is defined by $x_{ij} = 0$ for $(i,j) \in \kappa$,
and \eqref{eqn:wronskimapcoords} simplifies to
\[
    \Wr(\boldx; z) = z^n + 
     \Big(\frac{\numsyt\lambda}{n!}\Big)^{-1} \cdot
    \left[
       \numsyt{\beta^1} x_{i_1j_1} \frac{z^{n-1}}{(n-1)!}
+
       \numsyt{\beta^2} x_{i_2j_2} \frac{z^{n-1}}{(n-1)!}
+
       \numsyt\kappa x_{i_1j_1}x_{i_2j_2} \frac{z^{n-2}}{(n-2)!}
    \right]
\,.
\]
Equating coefficients of $\Wr(x) = g$, and solving for $x_{i_1j_1}$ we obtain
\[
   n\frac{\numsyt{\beta^1}}{\numsyt\lambda} x_{i_1j_1}^2
   + (a_1 + a_2) x_{i_1j_1}
   + (n-1)^{-1}\frac{\numsyt{\beta^2}}{\numsyt\kappa} a_1a_2
   = 0
\,.
\]
The discriminant of this quadratic equation is
\[
   (a_1+a_2)^2 
  - 4 n(n-1)^{-1}\frac{\numsyt{\beta^1}\numsyt{\beta^2}}
     {\numsyt\lambda\numsyt\kappa} a_1a_2
\,.
\]
Using the hook-length formula \cite{FRT} 
for $\numsyt \lambda$, it is easy to check
that $n(n-1)^{-1}\frac{\numsyt{\beta^1}\numsyt{\beta^2}}
{\numsyt\lambda\numsyt\kappa} = (1-L^{-2})$, from which the result follows.
\end{proof}

\subsection{Tableaux}

Suppose $\mu = (\mu_1, \dots, \mu_k)$ is a composition of $n$, i.e.
an ordered list of positive integers summing to $n$.
There is a partition associated to $\mu$, obtained by sorting the parts 
of $\mu$ in decreasing order.  We adopt the convention that 
whenever we use notation of the form ``$\,\cdot\,(\mu)$'' in a context 
where $\mu$ is supposed to be a partition, we will implicitly mean to use 
this associated partition.
For example, $\sgchar(\mu)$ means $\sgchar$ 
evaluated at the partition associated to $\mu$.

\begin{definition}
A \defn{weakly increasing tableau} of shape $\lambda$ and content $\mu$
is a filling 
of the diagram of $\lambda$ with positive integer entries, weakly
increasing along rows and columns, such that $\mu_b$ of the entries are
equal to $b$, for $b =1, \dots, k$.
We denote the set of all such tableaux by $\Tab(\lambda;\mu)$.
\end{definition}

In particular, the set of standard Young tableaux of 
shape $\lambda$ is $\SYT(\lambda) = \Tab(\lambda; 1^n)$.

For $T \in \Tab(\lambda; \mu)$ let $T(i,j)$ denote the entry in row $i$
and column $j$.
Let $\shape(T|_{\leq b})$ be
the partition defined by the entries of $T$ less
than or equal to $b$.  
Let $\shape(T|_b) := \shape(T|_{\leq b})/\shape(T|_{\leq b-1})$ 
be the skew shape associated to the entries equal to $b$.  
We note the following identity.

\begin{proposition}
\label{prop:easyidentity}
\[
\sum_{T \in \Tab(\lambda; \mu)} \prod_{b=1}^k \#\SYT(\shape(T|_b)) 
= \numsyt\lambda
\,.
\]
\end{proposition}

The Murnaghan--Nakayama rule computes characters of symmetric group 
representations in terms of tableaux.  We recall the statement.

\begin{definition}
$T \in \Tab(\lambda;\mu)$ is a \defn{Murnaghan--Nakayama tableau} if
$\shape(T|_b)$ is a connected shape containing no $2 \times 2$ square, 
for all $b = 1,\dots, k$.
We denote the set of Murnaghan--Nakayama tableaux of
shape $\lambda$ and content $\mu$ by $\MN(\lambda; \mu)$.
The \defn{sign} of a Murnaghan--Nakayama tableau is 
\[
\sgn(T) := \prod_{b=1}^k (-1)^{\rows(T|_b)-1}
\,.
\] 
where
$\rows(T|_b)$ the number of non-empty rows in $\shape(T|_b)$.
\end{definition}

\begin{theorem}[Murnaghan--Nakayama rule]
\label{thm:MNrule}
\[
  \sum_{T \in \MN(\lambda;\mu)} \sgn(T) = \sgchar(\mu)
\,.
\]
\end{theorem}

We now specialize to the case where 
$\mu_i \in \{1,2\}$, for all $i=1, \dots, k$.  
With this assumption, several things simplify.
For a tableau $T \in \Tab(\lambda; \mu)$, $\shape(T|_b)$ consists of
either 
\begin{packeditemize}
\item a single box (iff $\mu_b =1$);
\item two boxes forming a horizontal domino (i.e. in the same row);
\item two boxes forming a vertical domino (i.e. in the same column); or
\item two boxes that are non-adjacent.
\end{packeditemize}
Denote the number of $b$ such that $\shape(T|_b)$ falls into each of 
these cases by 
$\#_\one(T)$, $\#_\hdomino(T)$, $\#_\vdomino(T)$, and $\#_\twoskew(T)$ 
respectively.
Proposition \ref{prop:easyidentity} reduces to the statement
\[
   \sum_{T \in \Tab(\lambda;\mu)} 2^{\#_\twoskew(T)} = \numsyt\lambda
\,.
\]
$T$ is a Murnaghan--Nakayama tableau if and only if
$\#_\twoskew(T) = 0$, in which case we have
$\sgn(T) = (-1)^{\#_\vdomino(T)}$.

\begin{example} 
For $\lambda = 543$, $\mu=(1,1,2,2,1,1,1,2,1)$,
consider the tableau
\[
T =
\begin{young}[c]
1 & 2 & \ynobottom4 & 6 & 7 \\
]= 3 & 3 &]=]\ynotop 4 & 9 \\
5 & ]= 8 &=] 8
\end{young}
 \in \Tab(\lambda; \mu)
\,.
\]
We have 
$\#_\one(T) = 6$, $\#_\hdomino(T) = 2$, $\#_\vdomino(T) = 1$, and 
$\#_\twoskew(T) = 0$.  Hence, $T$ is a Murnaghan--Nakayama tableau, 
with $\sgn(T) = -1$.
\end{example}

\subsection{Special fibres of the Wronski map}
\label{sec:specialfibres-statement}

We continue to assume that $\mu = (\mu_1, \mu_2, \dots, \mu_k)$ 
is a composition of $n$, with $\mu_i \in \{1,2\}$ for $i=1, \dots, k$.   
We now assign points in $\scell$ to tableaux in $\Tab(\lambda; \mu)$.

We begin working over the field $\laurentC$ of formal Laurent series.
Note that if $\boldx(u) \in \Gr(d,\laurentC^{d+m})$, we can always take
$\lim_{u \to 0}\boldx(u)$ to obtain a point in $\Gr(d,\CC^{d+m})$.
Define polynomials
\[
   H_\mu(u,z) := 
   \prod_{b=1}^k \left(z^{\mu_b} + (\half u^{\overline\mu_b}+\half u^{\overline\mu_b + \mu_b-1})^{\mu_b}\right)
\,,
\]
where $\overline\mu_b := n+1-\sum_{i = 1}^b \mu_i$.
Note that if we evaluate at any $u \in (0,1)$, 
$H_\mu(u,z) \in \monics(\mu)$.
The roots of $H_\mu$ are either of the form $-u^j$ or
$\pm \tfrac{\imag}{2}(u^j+u^{j+1})$, where $\imag$ denotes the imaginary unit.

\begin{example}
\label{ex:Hpolynomial}
For $\mu=(2,1,2,2,1)$,
\[
 H_\mu(u,z) = 
\big(z^2 + (\half u^7+\half u^8)^2\big)
\big(z+u^6\big)
\big(z^2 + (\half u^4+\half u^5)^2\big)
\big(z^2 + (\half u^2+ \half u^3)^2\big)
\big(z+u^1\big)
\,.
\]
\end{example}

\begin{lemma}
\label{lem:specialfibres}
Consider the Wronski map $\Wr : \scell(\laurentC) \to \monics(\laurentC)$.
The fibre $\Wr^{-1}(H_\mu)$ consists of $\numsyt\lambda$ distinct points
in $\scell(\laurentC)$.  For each point 
$\boldx(u) \in \Wr^{-1}(H_\mu)$, the normalized \Plucker coordinates
$\big(\rmx_\kappa(u)\big)_{\kappa \subset \lambda}$ are power
series with a positive radius of convergence.
For each $T \in \Tab(\lambda;\mu)$, there is a set $W_T$ consisting
of $2^{\#_\twoskew(T)}$ distinct points in $\scell(\laurentC)$, 
with the following properties.
\begin{enumerate}[(a)]
\item $\bigcup_{T \in \Tab(\lambda; \mu)} W_T = \Wr^{-1}(H_\mu)$.
\item  For $\boldx(u) \in W_T$, and $b=1,\dots, k$,
\[
 \lim_{u \to 0} 
\left(\begin{smallmatrix} 1 & 0 \\[.4ex] 0 & u^{\overline\mu_b} \end{smallmatrix}\right)
\boldx(u)
\ \in\ \scellb{\shape(T|_{\leq b})}
\,,
\]
using the action of $\PGL_2$ on the Grassmannian.
\item For $\boldx(u) \in W_T$, $\boldx(u) \in \scell(\laurentR)$ 
if and only if $T$ is a 
Murnaghan--Nakayama tableau.
\end{enumerate}
\end{lemma}

We prove Lemma~\ref{lem:specialfibres}
in Section \ref{sec:specialfibres-proof}. It follows that for fixed $\lambda$, and for all sufficient small $\varepsilon > 0$, the following are true:
\begin{packeditemize}
\item
For every $\mu$ and every point $\boldx \in \Wr^{-1}(H_\mu)$,
the series $\rmx_\kappa(\varepsilon)$ converges for all 
$\kappa \subset \lambda$.
\item
For $\boldx(u) \in W_T$,
the point $\boldx(\varepsilon) \in \scell(\CC)$, defined by \Plucker
coordinates $\big(\rmx_\kappa(\varepsilon)\big)_{\kappa \subset \lambda}$, 
is real if and only if $T$ is a
Murnaghan--Nakayama tableau.
\item
All of the points $\boldx(\varepsilon)$ are distinct.
\end{packeditemize}

Pick such a suitable $\varepsilon$, and put 
$h_\mu(z) := H_\mu(\varepsilon, z) \in \monics(\mu)$.
For $T \in \MN(\lambda; \mu)$ 
put $\MNpoint_T := \boldx(\varepsilon) \in \scell(\mu)$,
where $\boldx(u)$ is the unique point in $W_T$.

\begin{corollary}
\label{cor:realfibrepoints}
The fibre $\Wr^{-1}(h_\mu)$ is reduced, and the set of real points in
$\Wr^{-1}(h_\mu)$ is
\[
\{\MNpoint_T \mid T \in \MN(\lambda;\mu)\}
\,.
\]
\end{corollary}

\begin{remark}
Lemma~\ref{lem:specialfibres} could be stated much more generally:
essentially, the proof uses only
the asymptotic behaviour of the roots of $H_\mu$, as $u \to 0$.
As such, the definition of $H_\mu$ is somewhat arbitrary, in that there 
are other choices that would work equally well.
However, Lemma~\ref{lem:specialfibres} is not the only consideration.
The choices we have made here will be particularly convenient 
later on, for constructions involving paths in Sections~\ref{sec:mbasepaths} 
and~\ref{sec:mfamilypaths}.
\end{remark}

\begin{example}
\label{ex:specialfibres}
Consider $\lambda = \twoone$.
Here, $\varepsilon = \half$ is sufficiently small, and
we have
\begin{align*}
  h_{1^3}(z) &= (z+\tfrac{1}{8})(z+\tfrac{1}{4})(z+\tfrac{1}{2}) \\
  h_{21}(z) &= (z^2+(\tfrac{3}{16})^2)(z+\tfrac{1}{2}) \\
  h_{12}(z) &= (z+\tfrac{1}{8})(z+(\tfrac{3}{8})^2) \,.
\end{align*}
As explained in Section~\ref{sec:example}, when $\lambda$ is of the
form $(n-1,1)$, the points of the fibre $\Wr^{-1}(h_\mu)$
correspond to the critical points of $h_\mu$.  
The polynomials
$h_{1^3}$ and $h_{21}$ each have two real critical points, and $h_{12}$ has
zero real critical points.  
According to Corollary~\ref{cor:realfibrepoints}, the real points
of $\Wr^{-1}(h_\mu)$ (hence the real critical points of $h_\mu$)
are in bijection with tableaux in $\MN(\lambda, \mu)$.  Indeed, we
have $\#\MN(\twoone, 1^3) = \#\MN(\twoone, 21) = 2$, 
and $\#\MN(\twoone,12) = 0$.  

The precise identification between
critical points of $h_\mu$ and Murnaghan--Nakayama tableaux is as shown in 
Figure~\ref{fig:21plots}.  
\begin{figure}[t]
\centering
\newcommand{\xmin}{-.75}
\newcommand{\xmax}{.75}
\newcommand{\ymin}{-1.1}
\newcommand{\ymax}{1.1}
\begin{tikzpicture}[x=6.25em,y=6.25em]
   \draw[violet] (\xmin,0) -- (\xmax,0);
   \draw[violet] (0,\ymax) -- (0,\ymin) 
       node[below,black] {$h_{1^3}$};
   \draw[<-] (-.42,.1) -- (-.5,.6) node[above] {{\scriptsize \begin{young}[2.3ex] _^1 & _^2 \\ _^3 \end{young}}};
   \draw[<-] (-.16,-.05) -- (.2,-.5) node[right] {{\scriptsize \begin{young}[2.3ex] _^1 & _^3 \\ _^2 \end{young}}};
   \clip  (\xmin,\ymin) rectangle (\xmax,\ymax);
   \draw[scale=1,domain=\xmin:\xmax,smooth,variable=\z,blue,thick] 
     plot ({\z},{20*(\z+1/8)*(\z+1/4)*(\z+1/2)});
\end{tikzpicture}
\qquad\quad
\begin{tikzpicture}[x=6.25em,y=6.25em]
   \draw[violet] (\xmin,0) -- (\xmax,0);
   \draw[violet] (0,\ymax) -- (0,\ymin) 
     node[below,black] {$h_{21}$};
   \draw[<-] (-.3,.52) -- (-.4,.75) node[above] {{\scriptsize \begin{young}[2.3ex]]=_^1 & _^1 \\ _^2 \end{young}}};
   \draw[<-] (-.03,.32) -- (.3,-.3) node[below] {{\scriptsize \begin{young}[2.3ex] _^1\ynobottom & _^2 \\ _^1\ynotop \end{young}}};
   \clip  (\xmin,\ymin) rectangle (\xmax,\ymax);
   \draw[scale=1,domain=\xmin:\xmax,smooth,variable=\z,blue,thick] 
     plot ({\z},{20*(\z*\z+9/256)*(\z+1/2)});
\end{tikzpicture}
\qquad\quad
\begin{tikzpicture}[x=6.25em,y=6.25em]
   \draw[violet] (\xmin,0) -- (\xmax,0);
   \draw[violet] (0,\ymax) -- (0,\ymin) 
      node[below,black] {$h_{12}$};
   \clip  (\xmin,\ymin) rectangle (\xmax,\ymax);
   \draw[scale=1,domain=\xmin:\xmax,smooth,variable=\z,blue,thick] 
     plot ({\z},{20*(\z*\z+9/64)*(\z+1/8)});
\end{tikzpicture}
\caption{Plots of the polynomials $h_\mu$, for $\lambda = \twoone$, and the Murnaghan--Nakayama tableaux corresponding to each critical point.}
\label{fig:21plots}
\end{figure}
We can verify this by computing 
of $\Wr^{-1}(H_\mu)$ directly.  For example, for $\mu = 1^3$,
$H_\mu(u,z) = (z+u^3)(z+u^2)(z+u)$, and the two points 
$\boldx_0(u)$, $\boldx_1(u)$ of $\Wr^{-1}(H_\mu)$ are 
spanned by the following polynomials in $\fpsR[z]$:
\begin{align*}
   \boldx_0(u) = \langle z^3 + \tfrac{3}{2}u^2 z^2 + 3u^5z + \dotsb\,,\,  z+\tfrac{2}{3}u + \dotsb\rangle \\
   \boldx_1(u) = \langle z^3 + 2uz^2 + 4u^4z + \dotsb\,,\, z+\tfrac{1}{2}u^2 + \dotsb \rangle \,. 
\end{align*}
Here we have only explicitly written the leading term in $u$ for each
coefficient of $z$.  Recall that the critical point associated to 
$\boldx_i$ is the root of the linear polynomial in the basis 
for $\boldx_i$; as $u \to 0$, these roots are asymptotically
$-\tfrac{2}{3}u < -\tfrac{1}{2}u^2$, so $\boldx_0$ corresponds to the 
smaller of the two critical points.  
To see which point corresponds to which 
tableau, we use part (b) of Lemma~\ref{lem:specialfibres}.
\begin{align*}
\lim_{u \to 0} \big(\begin{smallmatrix} 1 & 0 \\ 0 & u^3 \end{smallmatrix}\big)
\boldx_0(u) &= \langle \tfrac{3}{2}z^2+3z,\tfrac{2}{3} \rangle \in \scellb{\onesmall}
&
\lim_{u \to 0} \big(\begin{smallmatrix} 1 & 0 \\ 0 & u^3 \end{smallmatrix}\big)
\boldx_1(u) &= \langle 2z^2+4z,\tfrac{1}{2} \rangle \in \scellb{\onesmall}
\\
\lim_{u \to 0} \big(\begin{smallmatrix} 1 & 0 \\ 0 & u^2 \end{smallmatrix}\big)
\boldx_0(u) &= \langle z^3+ \tfrac{3}{2}z^2,\tfrac{2}{3}\rangle  \in \scellb{\hdomino}
&
\lim_{u \to 0} \big(\begin{smallmatrix} 1 & 0 \\ 0 & u^2 \end{smallmatrix}\big)
\boldx_1(u) &= \langle 2z^2,z+\tfrac{1}{2} \rangle \in \scellb{\vdomino}
\\
\lim_{u \to 0} \big(\begin{smallmatrix} 1 & 0 \\ 0 & u^1 \end{smallmatrix}\big)
\boldx_0(u) &= \langle z^3,z+\tfrac{2}{3} \rangle \in \scellb{\twoone}
\qquad\qquad&
\lim_{u \to 0} \big(\begin{smallmatrix} 1 & 0 \\ 0 & u^1 \end{smallmatrix}\big)
\boldx_1(u) &= \langle z^3+2z^2,z \rangle \in \scellb{\twoone}
\end{align*}
This shows that $\boldx_0$ corresponds to 
{\scriptsize \begin{young}[c][2ex] _^1&_^2\\_^3 \end{young}}
and $\boldx_1$ corresponds 
to {\scriptsize \begin{young}[c][2ex] _^1&_^3\\_^2 \end{young}}\,.
\end{example}


\section{Orientations}
\label{sec:orientation}

\subsection{Ambient orientations}

We begin by fixing orientations on affine spaces 
$\monics(\RR)$, and $\scell(\RR)$, which we will refer to as the
\defn{ambient orientations}.  These will serve as a point
of reference for defining orientations 
of $\monics(\mu)$ and $\scell(\mu)$.
Each component of these will either be oriented the same, or opposite 
to the ambient space in which it lies.

The degree of a map depends only on the relative orientations
of the spaces: reversing the orientations of both the domain and codomain
leaves the degree unchanged.  
Thus we can begin by making one choice without loss 
of generality.  We select either orientation as the ambient orientation
for $\monics(\RR)$.  Everything else will be defined relative to
this choice.  

Let $T_0 \in \SYT(\lambda)$ be the standard Young tableau with
entries $1, \dots, n$ in order, from left to right and  top to
bottom (i.e. the unique tableau such that $T_0(i,j) < T_0(i',j')$ 
whenever $i < i'$).
We define the ambient orientation of $\scell(\RR)$ to be the orientation
for which the Wronski map is locally orientation preserving in a 
neighbourhood of $\MNpoint_{T_0}$.  

As stated in the introduction, the orientation on 
$\monics(\mu)$ will simply be the restriction of the ambient orientation.
The character orientation of $\scell(\mu)$ is the complicated one, and
will be defined next.

\subsection{The character orientation}

Let $\kappa \subset \lambda$ be a partition.
The half-open Richardson varieties
$\rich(a)$, $a \in \affinespace^1$ define
a flat family of affine subvarieties of $\scell$ over $\affinespace^1$.
Let $\richfamily \subset \scell \times \affinespace^1$
be the total space of this family. 
Let $\pi_1 : \scell \times \affinespace^1 \to \scell$ be the 
projection onto the first factor, and define
$\dominovar := \pi_1(\richfamily)$ to be the
algebraic image of $\richfamily$ under this projection.
Informally, $\dominovar \subset \scell$ 
is the union of all $\rich(a)$,
$a \in \affinespace^1$.

\begin{proposition}
\label{prop:Zclosed}
$\dominovar$ is a closed subvariety of $\scell$.
\end{proposition}

\begin{proof}
Let $\richfamilyclosure$ be the closure of $\richfamily$ in 
$\scellclosure \times \mathbb{P}^1$, where 
$\scellclosure = \schubertopposite(\infty)$ 
is the Schubert variety. 
Note that $\richfamilyclosure \to \scellclosure$ is proper and that
\[
\richfamilyclosure 
\subset \overline{\bigcup_{a \in \affinespace^1} 
   (\schubertone(a) \cap \scellclosure)
   \times \{a\} }
\,,
\]
so in particular, 
the fibre $\richfamilyclosure_\infty = \pi_1(\pi_2^{-1}(\infty))$ is 
contained in the limiting fibre
\[
\lim_{a \to \infty} \schubertone(a) \cap \scellclosure = \partial \scellclosure
\,.
\]
Therefore $\richfamilyclosure_\infty$ does not intersect the 
Schubert cell $\scell$. Let $\richfamilyclosure|_{\scell} = 
\richfamilyclosure \cap (\scell \times \projspace^1)$.
Then $\pi_1 : \richfamilyclosure|_{\scell} \to \scell$ is again proper, 
so $\pi_1(\richfamilyclosure|_{\scell})$ is closed. By the above,
\[
  \pi_1(\richfamilyclosure|_{\scell}) = 
   \pi_1(\richfamilyclosure \cap (\scell \times \affinespace^1)) = 
   \pi_1(\richfamily) = \dominovar
\,.
\qedhere\]
\end{proof}

\begin{proposition}
\label{prop:Zdim}
$\dominovar$ is irreducible, and 
$\dim \dominovar = |\lambda|-|\kappa|+1$.
\end{proposition}

\begin{proof}
The fibres of $\richfamily$ over $\affinespace^1$ are all isomorphic to
the irreducible variety $\rich(0)$, so $\richfamily$ is integral and 
has dimension $|\lambda| - |\kappa| + 1$. 
Therefore $\dominovar = \pi_1(\richfamily)$ is irreducible and has 
dimension at most $|\lambda| - |\kappa| + 1$. It has dimension at 
least $|\lambda| - |\kappa| + 1$ since a point of $\dominovar$ lies on 
only finitely-many varieties $\rich(a)$, and each of these 
has dimension $|\lambda| - |\kappa|$.
\end{proof}

\begin{proposition}
\label{prop:Ztest}
If $\boldx \in \dominovar$, then $\Wr(\boldx)$ has a root of
multiplicity $|\kappa|$.  Conversely, if $\Wr(\boldx)$ has a root of
multiplicity $\ell$, then $\boldx \in \dominovar$ for some partition
$\kappa \vdash \ell$.
\end{proposition}

\begin{proof}
This follows from Lemma~\ref{lem:schubertwronskian}.
\end{proof}

Now suppose that $|\kappa| = 2$.  In this case, 
by Proposition~\ref{prop:Zdim}, $\dominovar$ is a 
closed hypersurface in $\scell \cong \affinespace^n$. Therefore the 
defining ideal of $\dominovar(\FF)$ is
generated by a single polynomial
$\dominofcn(\boldx) \in \FF[\boldx]$.  Since 
$\dominovar$ is defined over $\QQ$, this polynomial has rational 
coefficients; in particular we can think of $\dominofcn$ as a real
valued function on $\scell(\RR)$.

The \defn{discriminant variety} $\discrimvar \subset \monics$ is the 
hypersurface defined the vanishing of the discriminant function
$g \mapsto \discriminant_z(g(z))$.
We have $g \in \discrimvar$ if and only if $g$ has a repeated root.
By Proposition~\ref{prop:Ztest}, 
\[
   \Wr^{-1}(\discrimvar) = \hdominovar \cup \vdominovar
\,.
\]
Since the polynomials $h_\mu$ are not in $\discrimvar$, the points $\MNpoint_T$, $T\in \MN(\lambda;\mu)$
are not in $\dominovar$ for either $\kappa \vdash 2$.  
In particular, $\dominofcn(\MNpoint_{T_0}) \neq 0$,
and we will assume that $\dominofcn(\MNpoint_{T_0}) > 0$.

We call the two functions
$\vdominofcn(\boldx)$ and $\hdominofcn(\boldx)$ the
\defn{character orientation function} and the
\defn{dual character orientation function}, respectively.

\begin{lemma}
Let $\component$ be a component of $\scell(\mu)$, and $\kappa \vdash 2$.  
Then exactly one of the following must be true:
\begin{enumerate}[(a)]
\item
$\dominofcn$ is globally non-negative on $\component$; or
\item 
$\dominofcn$ is globally non-positive on $\component$.
\end{enumerate}
\end{lemma}

\begin{proof}
Since $\dim \dominovar = n-1$, and $\dim \component = n$, 
$\dominofcn$ cannot be identically zero on $\component$.  
Therefore at most one of (a) and (b) holds.

To see that at least one of these holds, we show that 
$\component \setminus \dominovar(\RR)$ is connected.
Since $\dominofcn$ is by definition non-vanishing on
$\component \setminus \dominovar(\RR)$ this implies either (a)
or (b) holds.

Suppose that 
$\boldx \in \dominovar(\RR) \cap \component$.
Since $\boldx \in \dominovar(\RR)$,
$\Wr(\boldx) \in \discrimvar$, i.e. $\Wr(\boldx)$ 
must have a repeated root.  On the other hand,
by definition if $\boldx \in \scell(\mu)$ then $\Wr(\boldx)$ cannot
have a repeated real root.  Therefore, $\Wr(\boldx)$ must have a repeated
non-real root.

Let $K = \{g \in \monics(\RR) \mid \text{$g$ has a repeated non-real
root}\}$.  Then $K$ has real codimension $2$ in $\monics(\RR)$. 
Since $\Wr$
is a proper map with finite fibres, $\Wr^{-1}(K)$ has codimension $2$ in
$\scell(\RR)$.  Since we just showed
$\dominovar(\RR) \cap \component \subset \Wr^{-1}(K)$, it follows
that 
$\dominovar(\RR) \cap \component$ has codimension at least $2$,
and hence $\component \setminus \dominovar(\RR)$ is connected.
\end{proof}

\begin{definition}
The \defn{character orientation} of $\scell(\mu)$ is defined as follows.  For
each component $\component$ of $\scell(\mu)$ we orient according to
the sign of the character orientation function.
\begin{enumerate}[(a)]
\item
If $\vdominofcn$ is globally non-negative on $\component$, then
the character orientation of $\component$ 
is the same as the ambient orientation.
\item
If $\vdominofcn$ is globally non-positive on $\component$, then
the character orientation of $\component$ is opposite to the ambient 
orientation.
\end{enumerate}
The \defn{dual character orientation} of $\scell(\mu)$ is the orientation
obtained similarly, using the dual character orientation function.
\end{definition}

\begin{remark} 
A common way to define an orientation on a manifold $X$ is to specify
a global non-vanishing volume form on $X$.  The non-vanishing condition
can be relaxed slightly to allow the form to vanish on a set of 
codimension $2$.  From this perspective, the character orientation 
on $\scell(\mu)$ is simply the ambient orientation multiplied by the
character orientation function.
\end{remark}

\begin{example}
\label{ex:orientation}
Recall, from Section~\ref{sec:example}, that for $\lambda = (n-1,1)$, 
$\scell$ is identified with the set of pairs $(g,c)$, 
where $g \in \monics$ and $g'(c) = 0$. On $\monics$, the variety defined by the discriminant
is irreducible, but working 
on $\scell$, we have the additional
data of a specified critical point, and we can identify
two components of $\Wr^{-1}(\discrimvar)$.  Specifically
$\vdominovar$ consists pairs $(g,c)$ such that $c$ is a double (or higher
order) root of $g$;
this is cut out by the additional equation $g(c) = 0$.
$\hdominovar$ is the closure of the set of pairs $(g,c)$ 
such $g$ has a double root at some point other than $c$.
In this case, the character orientation function is
$\vdominofcn(g,c) = (-1)^{n-1} g(c)$.  
(The sign is explained by the fact 
that the point $\MNpoint_{T_0}$ corresponds to the pair $(h_{1^n}, c_0)$,
where $c_0$ is the smallest critical point of $h_{1^n}$, and
$(-1)^{n-1}h_{1^n}(c_0) > 0$.)
As noted in Section~\ref{sec:example}, this function is in fact globally 
positive or globally negative on any component of $\scell(\mu)$.
The dual character orientation function is
$\hdominofcn(g,c) = (-1)^{n-1}\discriminant_z(g(z))/g(c)$.
\end{example}

\subsection{Signs of points in $\scell(\RR)$}

Let $\jacobian{\Wr}: \scell(\FF) \to \FF$ 
denote the Jacobian of the Wronski map
$\Wr : \scell(\FF) \to \monics(\FF)$ (with respect to affine coordinates).
The \defn{ramification divisor} of $\Wr : \scell \to \monics$, 
is the hypersurface in $\scell$ on which 
$\jacobian{\Wr}$ vanishes.  We denote it by $\critvar$.

For $\boldx \in \scell(\mu) \setminus \critvar(\RR)$, 
define the \defn{sign}
of $\boldx$ to be $\sgn(\boldx) = 1$ if the Wronski map is orientation
preserving with respect to the character orientation 
in a neighbourhood of $\boldx$, and $\sgn(\boldx) = -1$ otherwise.
Similarly, define the \defn{dual sign} of $\boldx$, denoted
$\dualsgn(\boldx)$, using the dual character orientation.
Define the \defn{ambient sign} of $\boldx$, denoted $\ambsgn(\boldx)$
using the ambient orientation.

Our main goal is to prove the following theorem.

\begin{theorem}
\label{thm:signs}
For every $T \in \MN(\lambda; \mu)$, 
\[
\sgn(\MNpoint_T) = \sgn(T)
\,.
\]
\end{theorem}

Our strategy is to join the points $\MNpoint_T$ together by paths,
and count the number of times the sign changes along a path.
Let $\gamma: [0,1] \to \scell(\RR)$, $t \mapsto \gamma_t$
be a path, and assume the following:
\begin{packeditemize}
\item $\gamma_0, \gamma_1 \notin \critvar(\RR) \cup \hdominovar(\RR) \cup \vdominovar(\RR)$;
\item $\gamma_t \in \critvar(\RR) \cup \hdominovar(\RR) \cup \vdominovar(\RR)$ for only finitely many values of $t$.
\end{packeditemize}
These conditions ensure that $\sgn(\gamma_t)$, $\dualsgn(\gamma_t)$ 
and $\ambsgn(\gamma_t)$ 
are defined at
all but finitely many points, including
$\gamma_0$ and $\gamma_1$.

We first establish the main properties of interest in detecting sign changes along paths in $\scell(\RR)$. In  Section~\ref{sec:nicepaths}, we then discuss how the existence of sufficiently nice paths allows us to prove Theorem \ref{thm:signs}.

\begin{proposition}
If $\sgn(\gamma_t)$ changes at the point $t$, then either
$\gamma_t \in \critvar(\RR)$, or $\gamma_t \in \vdominovar(\RR)$.
If $\dualsgn(\gamma_t)$ changes at the point $t$, then either
$\gamma_t \in \critvar(\RR)$, or $\gamma_t \in \hdominovar(\RR)$.
If $\ambsgn(\gamma_t)$ changes at the point $t$, then 
$\gamma_t \in \critvar(\RR)$.
\end{proposition}

\begin{proof}
If $\sgn(\gamma_t)$ changes, then either the sign of the Jacobian
of $\Wr$ reverses, or the orientation of the space reverses.  The
former can happen only when $\jacobian{\Wr}(\gamma_t) = 0$, i.e. 
the path crosses $\critvar(\RR)$, and the latter
can happen only when only when $\vdominofcn(\gamma_t) = 0$, i.e.
the path crosses $\vdominovar(\RR)$.
The other statements are similar.
\end{proof}

For the converse, we need a slightly stronger condition.

\begin{definition}
Let $V \subset \scell$ be an algebraic hypersurface defined over $\RR$, 
and let 
$\gamma: [0,1] \to \scell(\RR)$ be a path such that $\gamma_t \in V(\RR)$
for only finitely many values of $t$.  Let $t \in (0,1)$ be such
a value.  We say $\gamma$ has a \defn{simple crossing} 
of $V(\RR)$ at $t$, if $\gamma_t$ is an algebraically smooth 
point of $V(\RR)$, and $\gamma$ crosses $V(\RR)$ at $t$.
(Formally, ``crossing'' means the following: 
There exists an open neighbourhood
$U \subset \scell(\RR)$ of $\gamma_t$
such that for every sufficiently small $\epsilon > 0$, 
$\gamma_{t-\epsilon}$ and $\gamma_{t+\epsilon}$ are in different
components of $U \setminus V(\RR)$.)
\end{definition}

\begin{remark}
In the above definition, the path $\gamma$ itself need not be a smooth 
function of $t$ at the crossing point --- only the hypersurface it 
crosses needs to be smooth.  In fact, the paths we consider 
will be constructed piecewise, in such a way that they are non-differentiable 
at precisely the points where they 
cross one of the hypersurfaces of interest.  
\end{remark}

\begin{proposition}
If $\gamma$ has a simple crossing of $\critvar(\RR)$ or
$\vdominovar(\RR)$ at $t$, then $\sgn(\gamma_t)$ changes at $t$.
If $\gamma$ has a simple crossing of $\critvar(\RR)$ or
$\hdominovar(\RR)$ at $t$, then $\dualsgn(\gamma_t)$ changes at $t$.
If $\gamma$ has a simple crossing of $\critvar(\RR)$ at $t$, 
then $\ambsgn(\gamma_t)$ changes at $t$.
\end{proposition}

\begin{proof}
All three statements follow immediately from the following more general
statement.  
Let $V \subset \scell$ be a hypersurface, defined as the zero locus of 
a polynomial $\Phi$. 
If $\gamma$ has a simple crossing of $V(\RR)$ at $t$, then
$\Phi(\gamma_t)$ changes sign at $t$.   

To prove this, note that since $V$ is smooth at $\gamma_t$, we may,
by perturbing $\gamma$, assume that $\gamma$ is smooth 
and transverse to $V(\RR)$. The function $\Phi(\gamma_t)$ vanishes at $t$,
and $\ddt\Phi(\gamma_t)$ cannot vanish at $t$, since 
(by transversality) the tangent vector to $\gamma$ at $t$ is not in
the tangent space of $V$ at $\gamma_t$.  Therefore $\Phi(\gamma_t)$
changes signs at $t$. 
\end{proof}

\begin{remark}
For any point $\boldx \in \scell(\mu) \setminus \critvar(\RR)$,
$\mu = 2^{n_2}1^{n_1}$,
there exists a path $\gamma$ from $\MNpoint_{T_0}$ to $\boldx$ such that all
crossings are simple.  Let $g_t = \Wr(\gamma_t)$.  
We note that whenever the number of real roots of $g$ changes, 
$\gamma_t$ must have a simple crossing of $\vdominovar$ at $t$, 
or a simple crossing of $\hdominovar$ at $t$, but not both.  We therefore
have a relationship between $\sgn(\boldx)$ and $\dualsgn(\boldx)$:
\[
\sgn(\boldx)\cdot \dualsgn(\boldx) = (-1)^{n_2}
\,.
\]
Equivalently this shows that the dual character orientation of 
$\scell(\mu)$ is the same as the character orientation if $n_2$ is even, 
and opposite if $n_2$ is odd.
\end{remark}

The following technical lemma will help to identify simple crossings
of $\critvar(\RR)$.

\begin{lemma}
\label{lem:ramificationdivisor}
Let $\finitemap : X \to Y$ be a quasifinite map of smooth varieties of the same dimension. 
Let $\rd \subset X$ be the ramification divisor 
(defined locally by the vanishing of the Jacobian determinant). 
Then the smooth locus of $\rd$ is 
\[ 
   \rdsmooth 
    = \{x \in \rd : \text{the ramification degree at $x$ is exactly $2$}\}
\,.
\]
\end{lemma}

\begin{proof}
The claim is local on $X$ and $Y$, so we may assume 
$X = \Spec (A,\maxideal_x)$ and $Y = \Spec (B,\maxideal_y)$ are 
spectra of regular local rings, and $\finitemap$ is induced by 
a map of local rings $\finitemap^* : B \to A$, with 
$\finitemap^*(\maxideal_y) \subset \maxideal_x$. 
We have a short exact sequence of modules of differentials
\[
   0 \to \finitemap^*\Omega_Y \to \Omega_X \to \Omega_{X/Y} \to 0
\,,
\]
and the ramification locus $\rd$ is defined by the vanishing of the 
determinant of the map of free $A$-modules 
$\finitemap^*\Omega_Y \to \Omega_X$.

First assume the ramification degree is exactly $2$, that is, 
the scheme-theoretic fibre has length 2: 
\[
  \vdim_\CC (B/\maxideal_y \otimes_B A) 
   = \vdim_\CC (A/\finitemap^*(\maxideal_y)) = 2
\,.
\]
In particular $\finitemap^*(\maxideal_y)$ must contain $\maxideal_x^2$ 
(otherwise the quotient will be too large). 
In fact, $\finitemap^*(\maxideal_y)$ has the form $\maxideal_x^2 + L$, 
where $L \subset \maxideal_x/\maxideal_x^2$ is some codimension-1 
vector subspace.

Consider the map $\maxideal_y/\maxideal_y^2 \to \maxideal_x/\maxideal_x^2$ 
induced by $\finitemap$. By the above, the image of this map is $L$. 
Choosing and lifting bases to obtain minimal generators of the ideals, 
we may assume $\maxideal_x = (x_1, \ldots, x_n)$ and 
$\maxideal_y = (y_1, \ldots, y_n)$ where
\begin{equation} \label{eqn:ramification-reduction}
  \finitemap^*(y_1) = 0 \bmod \maxideal_x 
  \text{ and } 
  \finitemap^*(y_i) = x_i \text{ for } 2 \leq i \leq n
\,. 
\end{equation}
In particular, $\finitemap^*(dy_i) = dx_i$ for $2 \leq i \leq n$.

As for $y_1$, we have $\finitemap^*(y_1) \in \maxideal_x^2$. 
By Cohen-Macaulayness 
$\Sym^2(\maxideal_x/\maxideal_x^2) = \maxideal_x^2/\maxideal_x^3$. 
So, write $\finitemap^*(y_1) \bmod \maxideal_x^3$ 
as a quadratic polynomial $q(x_1, \ldots, x_n)$.
Taking differentials we get 
\[
\finitemap^*(dy_1) = 
\sum_{i=1}^n \tfrac{\partial q}{\partial x_i} dx_i \mod \maxideal_x^2 \langle dx_1, \ldots, dx_n\rangle
\,.
\]
Note that $\tfrac{\partial q}{\partial x_1}$ is a linear form and is 
non-zero, 
since $\tfrac{\partial q}{\partial x_1}=0$  would imply
$x_1^2 \notin \finitemap^*(\maxideal_y)$.
Thus the matrix for the map $\finitemap^*\Omega_Y \to \Omega_X$ has the form
\[\begin{pmatrix}
\tfrac{\partial q}{\partial x_1} + s_1 & 0 & \ldots & 0 \\
\tfrac{\partial q}{\partial x_2} + s_2  & 1 & \ldots & 0 \\
\vdots & 0 & \ddots & \vdots \\
\tfrac{\partial q}{\partial x_n} + s_n& 0 & \cdots & 1
\end{pmatrix}\]
with $s_i \in \maxideal_x^2$ for all $i$, and the determinant is 
$\jacobian{\finitemap} = \tfrac{\partial q}{\partial x_1} + s_1$. 
This cuts out a smooth divisor 
since $\jacobian{\finitemap}$ is part of a minimal generating set 
for $\maxideal_x$.

We have shown that $\rdsmooth$ contains the points of ramification 
index $2$. For the reverse direction, the proof is similar in spirit. 
Since we will not need it, we omit it.
\end{proof}

\subsection{Paths for the proof of Theorem \ref{thm:signs}} 
\label{sec:nicepaths}

To prove Theorem~\ref{thm:signs}, we will establish the following two
Lemmas.

\begin{lemma}
\label{lem:MNsigns}
Let $T \in \MN(\lambda; \mu)$, $\mu = (\mu_1, \dots, \mu_k)$, and 
suppose $\mu_b = 2$.
Let $T'$ be the tableau obtained from $T$ by incrementing entries 
$b+1, \dots, k$ by $1$, and changing the lower-right $b$
to $b+1$.
There exists a path $\gamma : [0,1] \to \scell$
from $\MNpoint_{T'}$ to $\MNpoint_T$,
such that:
\begin{enumerate}[(a)]
\item $\gamma_t \notin \critvar(\RR)$ for all $t \in (0,1)$.
\item If $\shape(T|_b)$ is a horizontal domino, then
      $\gamma_t \notin \vdominovar(\RR)$ for all $t \in (0,1)$, and
     there is exactly one $t \in (0,1)$ such $\gamma_t \in \hdominovar(\RR)$,
     and this is a simple crossing.
\item If $\shape(T|_b)$ is a vertical domino, then
      $\gamma_t \notin \hdominovar(\RR)$ for all $t \in (0,1)$, and
     there is exactly one $t \in (0,1)$ such $\gamma_t \in \vdominovar(\RR)$,
     and this is a simple crossing.
\end{enumerate}
\end{lemma}

\begin{lemma}
\label{lem:SYTsigns}
Let $T' \in \MN(\lambda;\mu')$, $\mu' = (\mu'_1, \dots, \mu'_{k+1})$.
Suppose $\mu'_b = \mu'_{b+1} = 1$, and the entries $b$ and $b+1$
of $T'$ are non-adjacent.
Let $T''$ be the tableau obtained by swapping 
these two entries. There exists a path $\gamma : [0,1] \to \scell$
from $\MNpoint_{T'}$ to $\MNpoint_{T''}$,
such that:
\begin{enumerate}[(a)]
\item There is exactly one $t \in (0,1)$ such that 
$\gamma_t \in \critvar(\RR)$, and this is a simple crossing.
\item There is exactly one $t \in (0,1)$ such that 
$\gamma_t \in \hdominovar(\RR)$, and this is a simple crossing.
\item There is exactly one $t \in (0,1)$ such that 
$\gamma_t \in \vdominovar(\RR)$, and this is a simple crossing.
\end{enumerate}
\end{lemma}

\begin{example}
We now continue Example~\ref{ex:specialfibres}, to illustrate
Lemmas~\ref{lem:MNsigns} and~\ref{lem:SYTsigns} in the case where
$\lambda = \twoone$.  Here, the four
relevant tableaux are
\[
  T_0 = \begin{young}[c] 1& 2 \\ 3 \end{young}
\qquad
  T_1 = \begin{young}[c] 1& 3 \\ 2 \end{young}
\qquad
  T_2 = \begin{young}[c] ]=1 & 1 \\ 2 \end{young}
\qquad
  T_3 = \begin{young}[c] 1\ynobottom & 2 \\ 1\ynotop \end{young}
\,.
\]
It is clear from Figure~\ref{fig:21plots} that we can find a path
$g : [0,1] \to \monicsb3(\RR)$ from $h_{1^3}$ to $h_{21}$ such that
$g_t$ has two distinct real critical 
points, $c_{0,t} < c_{1,t}$, for all $t \in [0,1]$.  
Lifting this path to $\scell(\RR)$, 
we obtain two paths: $\gamma_0,\, \gamma_1 : [0,1] \to \scell(\RR)$,
$\gamma_{i,t} = (g_t, c_{i,t})$ for $i=1,2$;
$\gamma_0$ connects $\MNpoint_{T_0}$ to $\MNpoint_{T_2}$, and
$\gamma_1$ connects $\MNpoint_{T_1}$ to $\MNpoint_{T_3}$.
Note that the larger critical point $c_{1,t}$ must change sign
along the path $g_t$, and so for some $t$, $g_t(c_{1,t}) = 0$.
As explained in Example~\ref{ex:orientation}, this means that
$\gamma_0$ crosses $\hdominovar(\RR)$
and $\gamma_1$ crosses $\vdominovar(\RR)$.
Since $g_t$ has two distinct critical points for all $t$, neither
path crosses $\critvar(\RR)$.
These are the two types of paths described in Lemma~\ref{lem:MNsigns}.

To connect $\MNpoint_{T_0}$ to $\MNpoint_{T_1}$, 
we need a path $\gamma : [0,1] \to \scell(\RR)$, 
from $(h_{1^3}, c_0)$ to $(h_{1^3},c_1)$,
where $c_0 < c_1$ are the two critical points of $h_{1^3}$.
To do this, we begin with a different path 
$g : [0, \half] \to \monics(\RR)$ such that $g_0 = h_{1^3}$,
$g_t$ has a two real critical points, $c_{0,t} < c_{1,t}$, 
for $t \in [0, \half)$,
and $g_{1/2}$ has a double critical point.  Again, we can lift
this path to $\scell(\RR)$ in two ways, but this time the two lifted 
paths meet at the fibre over $g_{1/2}$.
We can then combine these two lifts into a single path:
\[
   \gamma_t = \begin{cases}
    (g_t, c_{0,t}) &\quad \text{for $t \in [0,\half]$} \\
    (g_{1-t}, c_{1,1-t}) &\quad \text{for $t \in [\half, 1]$}.
   \end{cases}
\]
This is illustrated in Figure~\ref{fig:SYTsigns}. There are three special points along this path. Since $g_{1/2}$ has a double critical point, we have $\gamma_{1/2} \in \critvar(\RR)$.
There is also a point $t \in (0,\half)$ such that $g_t$ has a double real root at $c_{0,t}$.  This means that $\gamma_t \in \vdominovar(\RR)$,
and $\gamma_{1-t} \in \hdominovar(\RR)$.  
This is the type of path described in Lemma~\ref{lem:SYTsigns}.
\begin{figure}[t]
\centering
\newcommand{\xmin}{-4} \newcommand{\xmax}{1}
\newcommand{\ymin}{-3} \newcommand{\ymax}{3}
\begin{tikzpicture}[x=1em,y=1em]
   \draw[violet] (\xmin,0) -- (\xmax,0);
   \clip  (\xmin,\ymin) rectangle (\xmax,\ymax);
   \draw[scale=1,domain=\xmin:\xmax,smooth,variable=\z,blue,thick] 
     plot ({\z},{\z*(\z+1)*(\z+3)});
   \draw [teal, fill] (-2.21525,2.1126) circle(0.2);
\end{tikzpicture}
\qquad
\begin{tikzpicture}[x=1em,y=1em]
   \draw[violet] (\xmin,0) -- (\xmax,0);
   \clip  (\xmin,\ymin) rectangle (\xmax,\ymax);
   \draw[scale=1,domain=\xmin:\xmax,smooth,variable=\z,blue,thick] 
     plot ({\z},{\z*(\z+2)*(\z+2)});
   \draw [teal, fill] (-2,0) circle(0.2);
\end{tikzpicture}
\qquad
\begin{tikzpicture}[x=1em,y=1em]
   \draw[violet] (\xmin,0) -- (\xmax,0);
   \clip  (\xmin,\ymin) rectangle (\xmax,\ymax);
   \draw[scale=1,domain=\xmin:\xmax,smooth,variable=\z,blue,thick] 
     plot ({\z},{\z*(\z*\z+3*\z+3});
   \draw [teal, fill] (-1,-1) circle(0.2);
\end{tikzpicture}
\qquad
\begin{tikzpicture}[x=1em,y=1em]
   \draw[violet] (\xmin,0) -- (\xmax,0);
   \clip  (\xmin,\ymin) rectangle (\xmax,\ymax);
   \draw[scale=1,domain=\xmin:\xmax,smooth,variable=\z,blue,thick] 
     plot ({\z},{\z*(\z+2)*(\z+2)});
   \draw [teal, fill] (-.6666666,-1.185185) circle(0.2);
\end{tikzpicture}
\qquad
\begin{tikzpicture}[x=1em,y=1em]
   \draw[violet] (\xmin,0) -- (\xmax,0);
   \clip  (\xmin,\ymin) rectangle (\xmax,\ymax);
   \draw[scale=1,domain=\xmin:\xmax,smooth,variable=\z,blue,thick] 
     plot ({\z},{\z*(\z+1)*(\z+3)});
   \draw [teal, fill] (-.45142,-.63113) circle(0.2);
\end{tikzpicture}
\caption{Connecting $\MNpoint_{T_0}$ to $\MNpoint_{T_1}$, 
for $\lambda = \twoone$.}
\label{fig:SYTsigns}
\end{figure}
 
Figure~\ref{fig:threepaths} shows (a two-dimensional projection of)
the Schubert cell $\scell(\RR)$, along with $\hdominovar(\RR)$,
$\vdominovar(\RR)$ and $\critvar(\RR)$, and the three paths described 
in this example.
\begin{figure}[t]
\centering
\begin{tikzpicture}[x=3.75em,y=3.75em]
   \shadedraw[scale=2,domain=-1:0,smooth,variable=\z,left color=white, right color=red!25,draw=none] plot ({\z},{\z*\z*\z}) -- (-1,0);
   \shadedraw[scale=2,domain=-1:0,smooth,variable=\z,left color=white, right color=red!25,draw=none] plot ({\z},{-\z*\z*\z}) -- (-1,0);
   \shadedraw[scale=2,domain=0:1,smooth,variable=\z,right color=white, left color=red!25,draw=none] plot ({\z},{\z*\z*\z}) -- (1,0);
   \shadedraw[scale=2,domain=0:1,smooth,variable=\z,right color=white, left color=red!25,draw=none] plot ({\z},{-\z*\z*\z}) -- (1,0);
   \draw[blue,very thick] (0,2) -- (0,-2) node[right,black] {$\critvar$};
   \draw[scale=2,domain=-1:1,smooth,variable=\z,red,very thick] 
     plot ({\z},{\z*\z*\z});
   \draw[scale=2,domain=-1:1,smooth,variable=\z,red,dotted,ultra thick] 
     plot ({\z},{-\z*\z*\z});
   \node[right] at (2,-2) {$\hdominovar$};
   \node[left] at (-2,-2) {$\vdominovar$};
   \draw[fill] (-1.4,0.2) circle (0.07) node[left]{$\MNpoint_{T_0}$}
     node[below left=1.4ex] {{\scriptsize \begin{young}[2.3ex] _^1 & _^2 \\ _^3 \end{young}}};
   \draw[fill] (1.4,0.2) circle (0.07) node[right]{$\MNpoint_{T_1}$}
     node[below right=1.4ex] {{\scriptsize \begin{young}[2.3ex] _^1 & _^3 \\ _^2 \end{young}}};
   \draw[fill] (-1.45,1.7) circle (0.07) node[above]{$\MNpoint_{T_2}$}
    node[right=2ex] {{\scriptsize \begin{young}[2.3ex]]=_^ 1 & _^1 \\ _^2 \end{young}}};
   \draw[fill] (1.45,1.7) circle (0.07) node[above]{$\MNpoint_{T_3}$}
    node[left=2ex] {{\scriptsize \begin{young}[2.3ex] _^1\ynobottom & _^2 \\ _^1\ynotop \end{young}}};
   \draw[very thick, cyan, latex-latex] 
       (-1.4,.13) to [out=-80, in=-100] (1.4,.13);
   \draw[very thick, cyan, latex-latex] 
       (-1.35,0.25) to [out=60, in=-50] (-1.4,1.65);
   \draw[very thick, cyan, latex-latex] 
       (1.35,0.25) to [out=120, in=-130] (1.4,1.65);
\end{tikzpicture}
\caption{The points 
$\MNpoint_{T_0},\MNpoint_{T_1} \in \scell(1^3)$ (shaded),
$\MNpoint_{T_2},\MNpoint_{T_3} \in \scell(21)$ (unshaded), 
and the three paths connecting them.
The paths cross $\vdominovar$, $\hdominovar$, and $\critvar$ as described
in Lemmas~\ref{lem:MNsigns} and~\ref{lem:SYTsigns}.}
\label{fig:threepaths}
\end{figure}
\end{example}

Lemmas \ref{lem:MNsigns} and
\ref{lem:SYTsigns} will be proved 
in Section~\ref{sec:mfamilypaths}.  Modulo these, we can now
prove Theorems~\ref{thm:signs} and~\ref{thm:main}.

\begin{proof}[Proof of Theorem~\ref{thm:signs}]
By the definition of the character orientation,
$\sgn(\MNpoint_{T_0}) = 1$.  
By Lemma~\ref{lem:SYTsigns}
we can connect all the points $\MNpoint_T$, 
$T \in \SYT(\lambda)$ by
a network of paths $\gamma_t$ such that $\sgn(\gamma_t)$ changes
twice along each path.  It follows that $\sgn(\MNpoint_T)
= \sgn(\MNpoint_{T_0}) = 1 = \sgn(T)$,
for all $T \in \SYT(\lambda)$.

For $T \in \MN(\lambda; \mu)$, $\mu \neq 1^n$, we can connect $T$ to
some $T'$ as in Lemma~\ref{lem:MNsigns}.  Then $\sgn(T) = \sgn(T')$ if 
$\shape(T|_b)$ is a horizontal domino,
$\sgn(T) = -\sgn(T')$ if $\shape(T|_b)$ is a vertical domino.
Following the path $\gamma$ connecting $\MNpoint_T$ to 
$\MNpoint_{T'}$, we see that
$\sgn(\gamma_t)$ does not change in the horizontal domino case,
and $\sgn(\gamma_t)$ changes once in the vertical domino case.
The result now follows by a simple induction.
\end{proof}

\begin{proof}[Proof of Theorem~\ref{thm:main}]
Since $h_\mu \in \monics(\mu)$, the topological degree 
of $\Wr : \scell(\mu) \to \monics(\mu)$ is
\[
   \sum_{\boldx \in \Wr^{-1}(h_\mu)} \sgn(\boldx)
= 
   \sum_{T \in \MN(\lambda; \mu)} \sgn(\MNpoint_T)
=
   \sum_{T \in \MN(\lambda; \mu)} \sgn(T)
=
   \sgchar(\mu)
\,,
\]
where the equalities above are by Corollary~\ref{cor:realfibrepoints},
Theorem~\ref{thm:signs}, and Theorem~\ref{thm:MNrule}, respectively.
\end{proof}

\section{Stable curves}
\label{sec:stablecurves}

\subsection{Families over $\mbase$}
\label{sec:mfamily}

In order to connect up points $\MNpoint_T \in \scell(\RR)$, we work
with a different model of the Shapiro--Shapiro picture, in 
which $\Wr: \scell \to \monics$ is replaced by a related
finite flat map $\mmap: \mfamily \to \mbase$, over
the moduli space of genus $0$ 
stable curves with $n+3$ marked points.  This model was described by
Speyer \cite{Sp}.  We give a brief overview of Speyer's construction
and the properties we will need.  Since our goal is not to prove
these properties, some of the exposition will be simplified.
In order to make the connection to the Wronski map explicit, we will
also describe the construction in a more coordinatized way.

Consider the set $A_1 = \{0,1,\infty, a_1, \dots, a_n\}$, 
where each element is regarded purely as a formal symbol.
For a stable curve $C \in \mbase$, we
will label the marked points with the $n+3$ symbols from $A_1$.  
$\openmbase \subset \mbase$ is the open subvariety of smooth curves, 
isomorphic to $\projspace^1$ with distinct marked points.
If $C \in \openmbase$, we will choose coordinates on $C$ such that the 
marked points labelled $0,1,\infty$ are placed at 
$0,1, \infty \in \projspace^1$ respectively.  
Abusing notation somewhat, we also write $a_i$ to
denote the coordinate of the point labelled $a_i$ in $C$, whenever
it appears in a formula.
More generally, if $C \in \mbase$, there is a unique morphism 
$C \to \projspace^1$
such that the points labelled $0,1,\infty$ in $C$ map to 
$0,1,\infty \in \projspace^1$.  In this case $a_i$ will denote the
coordinate of the image of the point labelled $a_i$ in $\projspace^1$.
This allows us to associate a polynomial 
\[
   \pol(C) = \pol(C,z) :=
   \prod_{a_i \neq \infty} (z+a_i)
\]
to every curve $C \in \mbase$.

Let $A_\ell$ denote the set of all ordered $\ell$-tuples of distinct 
elements of $A_1$. For $p,q,r$ distinct points of $\projspace^1$, define 
$\phi_{p,q,r}(s) = \frac{(p-s)(q-r)}{(p-q)(s-r)}$ for $s \in \projspace^1$;
this is the unique transformation such that
\[
\phi_{p,q,r}(p) = 0
\,,
\quad
\phi_{p,q,r}(q) = 1
\,,
\quad
\phi_{p,q,r}(r) = \infty
\,.
\]
Given a curve $C \in \openmbase$ and $(p,q,r) \in A_3$, we interpret
$\phi_{p,q,r} \in \PGL_2$ to be the unique projective linear transformation
of $\projspace^1$ that sends marked points $(p,q,r)$ to $(0, 1, \infty)$
respectively. We have an injection 
$\Gr(d,d+m) \times \openmbase \to \Gr(d,d+m)^{A_3} \times \mbase$,
\[
    (\boldx, C) \mapsto 
   \Big( \big(\phi_{p,q,r}(\boldx)\big)_{(p,q,r) \in A_3}
     \,,\, C\Big)
\,.
\]
Define $\mGr$ to be the closure of the image of this map.  
The projection
\[
\mmap : \mGr \to \mbase
\]
defines a family over $\mbase$. 
We also have a projection map onto $\Gr(d,d+m)$ for each $(p,q,r) \in A_3$;
In particular, let
\[
\mtosmap : \mGr \to \Gr(d,d+m)
\]
onto the Grassmannian factor corresponding to 
$(p,q,r) = (0,1,\infty) \in A_3$.

For $C \in \openmbase$, the fibre $\mmap^{-1}(C)$ of this family
is isomorphic to $\Gr(d,d+m)$.  Specifically, 
$\mtosmap: \mmap^{-1}(C) \to \Gr(d,d+m)$
is an isomorphism (and the same is true for if we replace $\mtosmap$ by
the projection onto any other Grassmannian factor)
\cite[Proposition 3.3]{Sp}.
If on the other hand $C \in \mbase$ is a nodal curve,
the fibre is a flat degeneration of the Grassmannian \cite[Theorem 3.2]{Sp}.

Speyer's construction also allows arbitrary Schubert conditions placed at the marked points of the curve.
For $(p,q,r) \in A_3$, 
$s \in A_1$, and a partition $\lambda$, let
\[
U_\lambda(p,q,r;s)
:= \begin{cases}
\schubert(0) & \quad\text{if $p=s$}  \\
\schubert(1) & \quad\text{if $q=s$}  \\
\schubert(\infty) & \quad\text{if $r=s$}  \\
\Gr(d,d+m) & \quad\text{otherwise} 
\,.
\end{cases}
\]
Define
\[
\mschubert(s) 
:= \left(\prod_{(p,q,r) \in A_3} U_\lambda(p,q,r;s) \times \mbase\right) \cap \mGr
\,.
\]
The restricted map $\mmap: \mschubert(s) \to \mbase$, defines a
family in which the fibre
over a smooth curve $C \in \openmbase$ is identified with the
Schubert variety $\schubert(s)$.  
Specifically,  if $C \in \openmbase$,
$\mtosmap: \mmap^{-1}(C) \to  \schubert(s)$ is an isomorphism \cite[Proposition 3.3]{Sp}.
If $C \in \mbase$ is a nodal curve, the fibre is a degeneration of
the Schubert variety.  Speyer describes these degenerate fibres
explicitly \cite[Theorem 1.2]{Sp},
and in Section~\ref{sec:p1-chains}, we will state this result for
the special case we need.

Speyer also proves that intersections of the varieties $\mschubert(s)$
are well-behaved \cite[Theorem 1.1]{Sp}.

\begin{theorem}
For any partitions $\alpha^s \subset m^d$, $s \in A_1$, 
\[
\mmap : \bigcap_{s \in A_1} \widetilde{X}_{\alpha^s}(s) \to \mbase
\] 
defines a flat, proper, Cohen--Macaulay family over $\mbase$ of relative 
dimension $dm - \sum_{s \in A_1} |\alpha^s|$.  
\end{theorem}

We will be primarily interested in the family defined by the following 
intersection.
\begin{equation}
\label{eqn:mfamilydef}
   \mfamily := \mschubertone(a_1) \cap \dots \cap \mschubertone(a_n)
    \cap \mschubertopposite(\infty)
\,.
\end{equation}
In this case, $\mmap: \mfamily \to \mbase$ is a finite morphism.
Let $\openmfamily := \mmap^{-1}(\openmbase)$ denote the
restriction of this family to $\openmbase$.

\begin{remark}
\label{rmk:superfluousmarkedpoints}
The marked points $0$ and $1$ do not appear in the
definition of $\mfamily$, and it is possible to define a similar
finite family over $\modsc{n+1}$, with only marked 
points $\infty, a_1, \dots, a_n$. 
The reason for including the two additional marked points is to provide 
the following connection to the Wronski map.
\end{remark}

\begin{proposition}
\label{prop:familydiagram}
The diagrams below commute.
\[
\begin{CD}
    \mfamily  @>{\mtosmap}>> \scellclosure \\
    @V{\mmap}VV        @VV{\Wr}V \\
    \mbase  @>>{\pol}> \monicsclosure \\
\end{CD}
\qquad\qquad\qquad\qquad
\begin{CD}
    \openmfamily  @>{\mtosmap}>> \scell \\
    @V{\mmap}VV        @VV{\Wr}V \\
    \openmbase  @>>{\pol}> \monics \\
\end{CD}
\qquad
\]
In the first diagram, $\scellclosure = \schubertopposite(\infty)$
is the closure of $\scell$ in $\Gr(d,d+m)$, 
and $\monicsclosure = \projspace^n$ is the closure of
$\monics$ in $\projspace^{dm}$.
The second diagram is a fibre product: 
$\openmfamily = \openmbase \times_{\monics} \scell$.
\end{proposition}

\begin{proof}
Over a smooth curve $C \in \openmbase$, $\pi$ restricts to an
isomorphism $\Psi^{-1}(C) \to 
\schubertone(a_1) \cap \dots \cap \schubertone(a_n) \cap \schubertopposite(\infty)$ 
\cite[Proposition 3.3]{Sp}.
By Lemma~\ref{lem:schubertwronskian}, $\schubertone(a_1) \cap \dots \cap 
\schubertone(a_n) \cap \schubertopposite(\infty) =
\Wr^{-1}(\pol(C))$;
hence the second diagram is a fibre product.  Taking closures
everywhere in the second diagram gives the first diagram, which 
therefore also commutes.
\end{proof}

\subsection{$\projspace^1$-chains}
\label{sec:p1-chains}

The families $\mschubertopposite(\infty)$ and $\mfamily$ are defined as
subvarieties of $\Gr(d,d+m)^{A_3} \times \mbase$, but
this encoding is highly redundant --- locally,
we only need a small number of factors to get a 
faithful image.
The only nodal curves we will need in this paper are
curves that are $\projspace^1$-chains, and if we restrict to these,
the families $\mschubertopposite(\infty)$ and $\mfamily$ have a simpler
description.

Suppose $C \in \mbase$ has components
$C_1, C_2, \dots, C_{k+1}$.
Choose an isomorphism 
$C_i \to \projspace^1$, and for each
point $p \in C$, let $p^{(i)} \in \projspace^1$
denote the image of $p$ under the contraction map 
$C \to C_i \to \projspace^1$; we call 
$(p^{(1)}, \dots, p^{(k+1)}) \in (\projspace^1)^{k+1}$
the \defn{$C$-coordinates} of $p \in C$.
We thereby associate a polynomial to each component:
\[
\pol^{(i)}(C)
= \pol^{(i)}(C,z)
:= \prod_{a_j^{(i)} \neq \infty} (z+a_j^{(i)})
\,.
\]
By a \defn{$\projspace^1$-chain}, we mean that the nodes and marked
points are arranged as follows:
\begin{enumerate}[(1)]
\item
There is a node $o_i \in C$, joining 
$C_i$ to $C_{i+1}$, $i=1, \dots, k$.
We assume our coordinates are such that
$o_i^{(i)} = \infty$ and $o_i^{(i+1)}= 0$.

\item
The marked points $0, 1, \infty$ are on $C_{1}, C_{b_0}, C_{k+1}$ 
respectively, where $1 \leq b_0 \leq k+1$.
We assume our coordinates are such that
$0^{(1)} = 0$, $1^{(b_0)}=1$, and $\infty^{(k+1)}= \infty$.

\item
For $j=1, \dots, n$, put $b_j := i$ if $a_j \in C_{i}$.
For $i=1, \dots, k+1$, let $\mu_i := \#\{j \mid b_j=i\}$.
We require $\mu_i \geq 1$ for $i \neq b_0$, but $\mu_{b_0} =0$ is
allowed.  We write $\mu = (\mu_1, \dots, \mu_{k+1})$.
(This is equivalent to the stability condition on the curve.)

\item We allow the possibility a \defn{double marked point},
where two of the marked points $a_1, \dots, a_n$ 
are at the same point of $C$.  
In the usual way of thinking about $\mbase$, when two marked points
collide, the curve gains a new 
$\projspace^1$-component (sometimes called a ``bubble'') containing
the two marked points, attached to the original curve at the collision 
point.  But since this resolution is canonical,
it is also fine to think of this as a double marked point on $C$.
We will take this perspective, and we will not count these extra bubbles 
when counting the components of $C$.
The marked points must still be distinct from the nodes
$o_1, \dots, o_k$, and we do not allow triple marked points, 
or any other configuration normally forbidden by the definition
of a stable curve.
\end{enumerate}

\begin{remark}
We can specify a $\projspace^1$-chain by 
specifying values for $k$, $b_0, b_1, \dots, b_n$,
and $a_1^{(b_1)}, \dots, a_n^{(b_n)}$.  Note, however, that this
is also specifying $C$-coordinate isomorphisms
$C_{i} \to \projspace^1$, which are not
part of the actual data of the curve, and are not
canonical if $k \geq 1$. Each of the coordinate maps 
$C_i \to \projspace^1$, $i \neq b_0$
can be rescaled by any non-zero scalar, without changing the
curve.  (The map $C_{b_0} \to \projspace^1$, however, is 
pinned down by conditions 1 and 2.)
In particular, there is more than one way to 
specify the same nodal curve in $\mbase$.  For example, there
is a unique curve with $n+1$ components and $b_j = j+1$ for all
$j$; the values of $a_j^{(b_j)}$ specify coordinates on 
this curve, but do not help specify the curve itself.
\end{remark}

\begin{example}
An example of $\projspace^1$-chain is shown in Figure~\ref{fig:P1chain}.
\begin{figure}
\centering
\begin{tikzpicture}[x=2.5em,y=2.5em]
   \draw[violet] (1,0) circle(1);
   \draw[violet,dashed] (1,0) ellipse (1 and 0.3);
   \draw[violet] (3,0) circle(1);
   \draw[violet,dashed] (3,0) ellipse (1 and 0.3);
   \draw[violet] (5,0) circle(1);
   \draw[violet,dashed] (5,0) ellipse (1 and 0.3);
   \draw[violet] (7,0) circle(1);
   \draw[violet,dashed] (7,0) ellipse (1 and 0.3);
   \draw[fill] (0,0) circle (0.07) node[left]{$0$};
   \draw[fill] (8,0) circle (0.07) node[right]{$\infty$};
   \draw[fill] (5,-.3) circle (0.07) node[below]{$1$};
   \draw[fill] (1,.3) circle (0.07) node[above]{$a_5$};
   \draw[fill] (2.4,-.24) circle (0.07) node[below]{$a_8$};
   \draw[fill] (3,-.3) circle (0.07) node[below]{$a_2$};
   \draw[fill] (3,1) circle (0.07) node[above]{$~~~~~~~~a_3=a_4$};
   \draw[fill] (4.4,.24) circle (0.07) node[right]{$a_1$};
   \draw[fill] (6.4,.8) circle (0.07) node[above]{$a_6$};
   \draw[fill] (6.4,-.8) circle (0.07) node[below]{$a_7$};
\end{tikzpicture}
\caption{A $\projspace^1$-chain in $\modsc{11}$ with four components.}
\label{fig:P1chain}
\end{figure}
In this example we have $b_5 = 1$, $b_2=b_3=b_4=b_8 =2$, $b_0 = b_1 = 3$,
and $b_6 = b_7=4$.  
The dashed circle on each $\projspace^1$ represents points with 
real $C$-coordinates, and the solid outer circle represents points with 
pure imaginary $C$-coordinates.  In this case,
\[
a_1^{(3)} = -\half
\,, ~~~
a_2^{(2)} = 1 
\,, ~~~
a_3^{(2)} = a_4^{(2)} = \imag 
\,, ~~~
a_5^{(1)} = -1
\,, ~~~
a_6^{(4)} = \tfrac{\imag}{2}
\,, ~~~
a_7^{(4)} = -\tfrac{\imag}{2}
\,, ~~~
a_8^{(2)} = \half
\,.
\]
Note that $a_3 = a_4$ is a double marked point.
The associated polynomials are:
\begin{align*}
\pol^{(1)}(z) &= (z-1) \\
\pol^{(2)}(z) &= z(z+1)(z+\half)(z+\imag)^2 \\
\pol^{(3)}(z) &= z^5(z-\half) \\
\pol^{(4)}(z) &= z^6(z+\tfrac{\imag}{2})(z-\tfrac{\imag}{2})
\,,
\end{align*}
and since $b_0= 3$, $\pol(z) = \pol^{(3)}(z)$.
\end{example}

For $j=1, \dots, n$, let 
$\mtosmap_j : \mGr \to \Gr(d,d+m)$ denote the projection onto the Grassmannian
factor corresponding to $(0, a_j, \infty) \in A_3$.  When we restrict our 
base to the open subvariety of $\projspace^1$-chains in $\mbase$, 
\[
    \mschubertopposite(\infty) \xrightarrow{(\mtosmap_1, \dots, \mtosmap_n)}
     \Gr(d,d+m)^n \times \mbase
\]
restricts to an injective map.  Since all calculations in this paper
take place within this restriction, we will henceforth identify
$\mschubertopposite(\infty)$ with its image under this map.

Let $C$ be a $\projspace^1$-chain, and let
$\bigfibre(C)$ denote the fibre
of the map $\mmap: \mschubertopposite(\infty) \to \mbase$ over $C$
(regarded as a subvariety of $\Gr(d,d+m)^n$).
If $C$ has more than one component, then $\bigfibre(C)$ is a reducible scheme,
and its components are indexed by $\Tab(\lambda; \mu)$, 
where $\mu = (\mu_1, \dots, \mu_{k+1})$ is as above.
Write $\bigfibre_T(C)$ for the component indexed by $T \in \Tab(\lambda; \mu)$.
Write $Q^{(b)}_T$ for the closure of 
$X^{\shape(T|_{\leq b})}_{\shape(T|_{\leq b-1})}(0)$.

\begin{theorem}
\label{thm:degeneratefibreiso}
We have an isomorphism 
$\bigfibre_T(C) \to \prod_{b=1}^{k+1} Q^{(b)}_T$, 
\begin{equation}
\label{eqn:C-coords}
    \mfpoint \mapsto 
  (\mfpoint^{(1)}, \dots, \mfpoint^{(k+1)})
\,,
\end{equation}
characterized by
\begin{equation}
\label{eqn:degeneratefibreiso}
\mtosmap_j(\mfpoint) = \phi_{0,a_j^{(b_j)},\infty}(\mfpoint^{(b_j)})
	 \,,
\qquad \text{for }j=1, \dots, n \,.
\end{equation}
\end{theorem}

\begin{proof}
The existence of such an isomorphism is the content 
of \cite[Theorem 1.2]{Sp}, in the case where
the curve is $\projspace^1$-chain, and there is only one Schubert 
condition.  The details of the isomorphism, in general, are described in
\cite[Section 3]{Sp}.
\end{proof}

We call $(\mfpoint^{(1)}, \dots, \mfpoint^{(k+1)}) 
\in \prod_{b=1}^{k+1} Q^{(b)}_T$ the \defn{$C$-coordinates} of the 
corresponding point $\mfpoint \in \bigfibre_T(C)$.

Let $\mfibre(C)$ denote the fibre of the map $\mmap : \mfamily \to \mbase$ over
$C$.  Then $\mfibre(C) \subset \bigfibre(C)$, so every point in $\mfibre(C)$ is in
some component $\bigfibre_T(C)$.
Write $\mfibre_T(C) := \mfibre(C) \cap \bigfibre_T(C)$.

\begin{theorem}
\label{thm:finitefibreiso}
Every point of $\mfibre(C)$ is in $\mfibre_T(C)$ for some unique 
$T \in \Tab(\lambda; \mu)$.
Under the isomorphism \eqref{eqn:C-coords},
$\mfibre_T(C)$ is identified with $\prod_{b=1}^{k+1} \mfibrefactor{b}(C)$, 
where 
\[
  \mfibrefactor{b}(C) = 
    X^{\shape(T|_{\leq b})}_{\shape(T|_{\leq b-1})}(0)
    \ \cap\ \Wr^{-1}(\pol^{(b)}(C))\,.
\]
\end{theorem}

\begin{proof}
First, suppose $C$ has no double marked points.  Then by definition,
\[
    \mfibre(C) = \bigfibre(C)\ \cap\ %
      \bigcap_{j=1}^n \mtosmap_j^{-1} \schubertone(1)
\,.
\]
By \eqref{eqn:degeneratefibreiso}, intersecting $\bigfibre_T(C)$ with
$\mtosmap_j^{-1} \schubertone(1)$  corresponds to intersecting
the factor $\bigfibrefactor{b_j}$ with $\schubertone(a_j^{(b_j)})$.
Thus $\mfibre_T(C)$ is identified with 
\[
   \prod_{b=1}^{k+1} \left(\bigfibrefactor{b}\ \cap\ %
      \bigcap \schubertone(a_j^{(b_j)}) \right)
\,,
\]
where the intersection is taken over all $j$ such that $b_j = b$.
The result now follows from  Lemma~\ref{lem:schubertwronskian}.
In the case where $C$ has a double marked point, we deduce the result by taking limits and observing that 
\[
   \lim_{a' \to a} \schubertone(a) \cap \schubertone(a') 
    = \hdominoschubert(a) \cup \vdominoschubert(a)
\,.
\qedhere
\]
\end{proof}

In the case where the curve has a double marked point, we distinguish
the following two cases.

\begin{corollary}
\label{cor:doublemarkedpoint}
If $C$ has a double marked point, $a_j=a_{j'}$ on component $C_b$, 
and $\mfpoint \in \mfibre_T(C)$, then we have either 
$\mtosmap_j(\mfpoint) \in \hdominoschubert(1) \cap \bigfibrefactor{b}$ or
$\mtosmap_j(\mfpoint) \in \vdominoschubert(1) \cap \bigfibrefactor{b}$ but not both.
\end{corollary}

\begin{proof}
By definition $\mfpoint^{(b)} \in \bigfibrefactor{b}$, and
by Theorem~\ref{thm:finitefibreiso}, $\mfpoint^{(b)} \in
\Wr^{-1}(\pol^{(b)}(C))$.  Since $\pol^{(b)}(C)$ has a double root
at $a_j^{(b)}$, by Lemma~\ref{lem:schubertwronskian}
we have $\mfpoint^{(b)} \in X_\hdomino(a_j^{(b)})$
or $\mfpoint^{(b)} \in X_\vdomino(a_j^{(b)})$ but not both.
Applying \eqref{eqn:degeneratefibreiso} gives the result.
\end{proof}
\begin{remark}
For many of the nodal curves we consider, we will have 
$b_0 = k+1$ and $\mu_{k+1} = 0$, i.e. $1$ and $\infty$ are the only
marked points $C_{k+1}$.
When this happens, $\bigfibrefactor{k+1}$ is a point, and 
$\mfpoint^{(k+1)}$ does not appear on the right side 
of \eqref{eqn:degeneratefibreiso}.
We will therefore sometimes drop the $(k+1)$-term from 
our notation, e.g. writing $\mu = (\mu_1, \dots, \mu_k)$, or
$\mfibre_T(C) = \prod_{b=1}^k \mfibrefactor{b}(C)$, etc.
\end{remark}

\subsection{Real structures}

For each quadruple $(p,q,r,s) \in A_4$ of distinct elements of $A_1$, 
there is a natural
map $\theta_{p,q,r,s} : \mbase \to \modsc{4}$, 
which forgets all marked points except for
$p,q,r,s$, and contracts any unstable components of the curve.
Using the canonical identification of $\modsc{4}$ with $\projspace^1$
(in which $(p,q,r) \mapsto (0, 1, \infty)$),
we rewrite this as $\theta_{p,q,r,s} : \mbase \to \projspace^1$,
\[
    \theta_{p,q,r,s}(C) 
     = \frac{(\hat p -\hat s)(\hat q - \hat r)}
       {(\hat p -\hat q)(\hat r - \hat r)}
\]
where $\hat p, \hat q, \hat r, \hat s$ denote the images of points
$p,q,r,s$ in $\modsc{4}$, in any coordinates.
The product of all such maps gives an embedding
\begin{equation}
\label{eqn:projectiveembedding}
   \mbase \hookrightarrow (\projspace^1)^{A_4}
\,.
\end{equation}

From this construction we obtain the standard real structure 
on $\mbase$, which is inherited from the standard real structure 
on $\projspace^1$: the complex conjugate
of a stable curve $C \in \mbase$ is obtained by conjugating each of the
nodes and marked points.  
Similarly, there is a standard real structure on
on $\mGr$, defined via its embedding in
$\Gr(d,d+m)^{A_3} \times \mbase$.
For either of these spaces, we will denote this complex 
conjugation map by $\xi$.

The standard notion of a real point
of $\mbase$ and $\mfamily$ does not precisely correspond to the 
notion of a real point of $\monics$ or $\scell$.  Informally a point
of $\mbase$ is real iff all nodes and marked points are real, whereas a 
point of $\monics$ is real whenever its roots are a mixture of real points and
complex conjugate pairs.  Thus the real points of $\mbase$ map to the closure of $\monics(1^n) \subset \monicsclosure(\RR)$.
To study curves such that $\pol(C) \in \monics(\mu)$ for $\mu \neq 1^n$, we need to
consider other real structures on $\mbase$, in which specified pairs of 
marked points are required to be complex conjugates of each other.

Let $\symgroup_n$ denote the symmetric group of permutations 
of $\{1, \dots, n\}$.  $\symgroup_n$ acts
on $\mbase$ by permuting the marked points $a_1, \dots, a_n$, and on
$\mGr$ by also permuting the corresponding factors in $\Gr(d,d+m)^{A_3}$.
If $\sigma \in \symgroup_n$ is an involution, we define 
$\xi^\sigma := \sigma \circ \xi$ acting on either
$\mbase$ or $\mGr$ (or any of its $\xi$- and $\symgroup_n$-invariant 
subvarieties, e.g. $\mschubertopposite(\infty)$ or $\mfamily$).
If $\sigma \in \symgroup_n$ is the identity element, then $\xi^\sigma = \xi$;
otherwise, $\xi^\sigma$ is the complex conjugation for a different real
structure on these spaces.  
The real points with respect to
this real structure are the $\xi^\sigma$-fixed points, and
for any $\xi$- and $\symgroup_n$-invariant subvariety $V$ of $\mGr$ or $\mbase$, we denote
the $\xi^\sigma$-fixed points of $V$ by $V(\RR^\sigma)$.

\begin{proposition}
\label{prop:realstructures}
Let $\sigma \in \symgroup_n$ be an involution. We have the following commutative diagram.
\[
\begin{CD}
    \mfamily(\RR^\sigma)  @>{\mtosmap}>> \scellclosure(\RR) \\
    @V{\mmap}VV        @VV{\Wr}V \\
    \mbase(\RR^\sigma) @>>{\pol}> \monicsclosure(\RR) \\
\end{CD}
\]
If $C \in \openmbase$ and $\pol(C) \in \monics(\RR)$, then there is
a unique involution $\sigma \in \symgroup_n$ such that 
$C \in \openmbase(\RR^\sigma)$.
If $\mfpoint \in \mfibre(C)$ and $\mtosmap(\mfpoint) \in \scell(\RR)$,
then $\sigma$ is also the unique
involution such that $\mfpoint \in \mfamily(\RR^\sigma)$.
\end{proposition}

Note that, when $\sigma$ has cycle type $\mu$, the images of $\mfamily(\RR^\sigma)$ and $\mbase(\RR^\sigma)$ are the closures of $\scell(\mu) \subset \scellclosure(\RR)$ and $\monics(\mu) \subset \monicsclosure(\RR)$.

\begin{proof}
We have $\pol \circ \sigma = \pol$, so for $C \in \mbase$,
$\xi^\sigma C = C$ implies that $\pol(C)$ is real.
Since $\sigma$ does not permute $0,1,\infty$, it fixes the 
$(0,1, \infty)$-factor in $\Gr(d,d+m)^{A_3}$, so
$\mtosmap \circ \sigma= \mtosmap$,
and so a similar argument applies for the map $\mtosmap$.

If $C \in \openmbase$ and $\pol(C)$ is real, then 
$\pol(\xi C) = \overline{\pol(C)} = \pol(C)$.  It follows that
$\xi$ is just permuting the marked points of $C$, i.e.
$\xi C = \sigma C$ for some $\sigma \in \symgroup_n$, or equivalently
$C \in \openmbase(\RR^\sigma)$ ($\sigma$ must therefore be
an involution, since $\xi$ is an involution).
Finally, suppose $\mfpoint \in \mfibre(C)$ and $\mtosmap(\mfpoint)$ 
is real.
Then 
\[
\mtosmap(\sigma \mfpoint) = \mtosmap(\mfpoint) = 
\overline{\mtosmap(\mfpoint)} = \mtosmap(\xi \mfpoint)
\]
and 
\[
\mmap(\sigma \mfpoint) = \sigma C = \xi C = \mmap(\xi \mfpoint)
\,,
\]
so
by the last part of Proposition~\ref{prop:familydiagram}, it follows that
$\sigma \mfpoint = \xi \mfpoint$.
\end{proof}

If $C$ is a $\projspace^1$-chain, then we have the following
characterization.
We say that the $C$-coordinates
are \defn{$\xi^\sigma$-compatible}, if 
$a_j^{(i)}$ is the complex conjugate of $a_{\sigma(j)}^{(i)}$ for all 
$i=1,\dots, k+1$, $j=1, \dots, n$ (this requires
$a_j$ and $a_{\sigma(j)}$ to be on the same component of $C$).
Note that this implies that
$C \in \mbase(\RR^\sigma)$ is $\xi^\sigma$-fixed.
Conversely if $C$ is $\xi^\sigma$-fixed
there exists a choice of $\xi^\sigma$-compatible
$C$-coordinates (though not every choice of $C$-coordinates is
$\xi^\sigma$-compatible).

\begin{proposition}
\label{prop:realchains}
Let $\sigma \in \symgroup_n$ be an involution, and let
$C$ be a $\projspace^1$-chain with $\xi^\sigma$-compatible 
$C$-coordinates.  For $\mfpoint \in \mfibre(C)$, we have 
$\xi^\sigma \mfpoint = \mfpoint$ if and only if 
$(\mfpoint^{(1)}, \dots, \mfpoint^{(k+1)})$ are real.
\end{proposition}

\begin{proof}
First suppose $\xi^\sigma \mfpoint = \mfpoint$.  
This means that 
\begin{equation}
\label{eqn:realchains}
\mtosmap_j(\mfpoint) = \xi \mtosmap_{\sigma(j)}(\sigma \mfpoint)
\,.
\end{equation}
for $j=1, \dots, n$.
Let $b \in \{1, \dots, k\}$.
If the component $C_b$ in the chain does not contain any of the 
marked points $a_1, \ldots, a_n$, then
$\bigfibrefactor{b}$ 
consists of a single real point, so $\mfpoint^{(b)}$ is real.  
Otherwise, let $a_j$ be a marked point on $C_b$,
and let $a = a_j^{(b)}$ be its coordinate.  
Then $a_{\sigma(j)} \in C_b$ and
since the coordinates are $\xi^\sigma$-compatible,
$a_{\sigma(j)}^{(b)} = \overline{a}$.
By \eqref{eqn:degeneratefibreiso}, we have
\[
\mtosmap_j(\mfpoint)
=
\phi_{0,a,\infty}(\mfpoint^{(b)})
\]
and
\[
\xi \mtosmap_{\sigma(j)}(\sigma \mfpoint) 
= \xi \phi_{0,\overline{a},\infty}(\sigma \mfpoint^{(b)})
= \xi \phi_{0,\overline{a},\infty}(\mfpoint^{(b)})
= \phi_{0,a,\infty}(\xi \mfpoint^{(b)})
\,,
\]
which implies that $\mfpoint^{(b)}$ is real.  

Conversely, if $\mfpoint^{(b)}$
is real, the calculation above shows that \eqref{eqn:realchains} holds
for all $j$ such that $a_j$ is on $C_b$.
If $\mfpoint^{(b)}$ is real all $b$,
then \eqref{eqn:realchains} holds for all $j$, so 
$\xi^\sigma\mfpoint = \mfpoint$.
\end{proof}

\begin{remark}
There is a well-known CW-complex description of $\modsc{n}(\RR)$ in terms of associahedra \cite{Dev}. It would be interesting to see an analogous description of $\modsc{n}(\RR^\sigma)$ and of the attachments between the twisted structures for each $\sigma$.
\end{remark}

\subsection{Special fibres}
\label{sec:specialfibres-proof}

Let $\mu = (\mu_1, \dots, \mu_k)$, be a composition with 
$\mu_i \in \{1,2\}$.  Recall that 
$\overline\mu_b = n+1 - \sum_{i=1}^b \mu_i$.

Working over $\laurentC$, we
define $C_\mu(u) \in \openmbase(\laurentC)$ to be the curve
whose marked points $a_1, \dots, a_n$ are specified
as follows:
\begin{equation}
\label{eqn:specialcurve}
 a_j = \begin{cases}
   u^j & \quad
    \text{if $j = \overline\mu_i$ for some $i$, $\mu_i = 1$} \\
   \tfrac{\imag}{2}(u^j+u^{j+1}) & \quad
    \text{if $j = \overline\mu_i$ for some $i$, $\mu_i = 2$} \\
   - \tfrac{\imag}{2}(u^{j-1}+u^{j}) & \quad
    \text{if $j-1 = \overline\mu_i$ for some $i$, $\mu_i = 2$} \\
  \end{cases}
\end{equation}
Note that $\pol(C_\mu(u),z) = H_\mu(u,z)$.

\begin{proposition}
\label{prop:limitcurve}
The limit curve $\lim_{u \to 0}C_\mu(u)$ is the $\projspace^1$-chain
$C_\mu(0)$, with $k+1$ components, $b_0 = k+1$, specified by the following 
coordinate data:
\begin{packeditemize}
\item if $j=\overline\mu_i$, $\mu_i = 1$,
 then $b_j = i$ and $a_j^{(b_j)} = 1$;
\item if $j=\overline\mu_i$, $\mu_i = 2$,
 then $b_j =b_{j+1} = i$ and 
$(a_j^{(b_j)},
a_{j+1}^{(b_{j+1})}) =
(\tfrac{\imag}{2},
-\tfrac{\imag}{2})$.
\end{packeditemize}
\end{proposition}

\begin{proof}
Using the embedding~\eqref{eqn:projectiveembedding},
it suffices to show that for all $(p,q,r,s)$,
we have 
$\lim_{u \to 0} \theta_{p,q,r,s}(C_\mu(u))
= \theta_{p,q,r,s}(C_\mu(0))$.
There are several cases, but this is straightforward. Note that $a_j$'s from distinct $\mu_i$'s have distinct leading orders as $u \to 0$, which greatly simplifies the calculation.
\end{proof}

\begin{example}
For $\mu=(2,1,2,2,1)$, the curve $C_\mu(u)$ has marked points $0,1,\infty$
and
\[
\begin{aligned}
a_7 &= \tfrac{\imag}{2}(u^7+u^8)\,,  \\
a_8 &= -\tfrac{\imag}{2}(u^7+u^8) \,, 
\end{aligned} 
\quad~
a_6 = u^6\,, 
\quad~
\begin{aligned}
a_4 &= \tfrac{\imag}{2}(u^4+u^5)\,,  \\
a_5 &= -\tfrac{\imag}{2}(u^4+u^5) \,, 
\end{aligned} 
\quad~
\begin{aligned}
a_2 &= \tfrac{\imag}{2}(u^2+u^3)\,,  \\
a_3 &= -\tfrac{\imag}{2}(u^2+u^3) \,, 
\end{aligned}
\quad~
a_1 = u^1\,.
\]
The polynomial $\pol(C_\mu(u),z) = H_\mu(u,z)$ is given in
Example~\ref{ex:Hpolynomial}.
The limit curve $C_\mu(0)$ is shown in Figure~\ref{fig:limitcurve}.
\begin{figure}
\centering
\begin{tikzpicture}[x=2.5em,y=2.5em]
   \draw[violet] (1,0) circle(1);
   \draw[violet,dashed] (1,0) ellipse (1 and 0.3);
   \draw[violet] (3,0) circle(1);
   \draw[violet,dashed] (3,0) ellipse (1 and 0.3);
   \draw[violet] (5,0) circle(1);
   \draw[violet,dashed] (5,0) ellipse (1 and 0.3);
   \draw[violet] (7,0) circle(1);
   \draw[violet,dashed] (7,0) ellipse (1 and 0.3);
   \draw[violet] (9,0) circle(1);
   \draw[violet,dashed] (9,0) ellipse (1 and 0.3);
   \draw[violet] (11,0) circle(1);
   \draw[violet,dashed] (11,0) ellipse (1 and 0.3);
   \draw[fill] (0,0) circle (0.07) node[left]{$0$};
   \draw[fill] (12,0) circle (0.07) node[right]{$\infty$};
   \draw[fill] (11,-.3) circle (0.07) node[below]{$1$};
   \draw[fill] (9,-.3) circle (0.07) node[below]{$a_1$};
   \draw[fill] (6.4,.8) circle (0.07) node[above]{$a_2$};
   \draw[fill] (6.4,-.8) circle (0.07) node[below]{$a_3$};
   \draw[fill] (4.4,.8) circle (0.07) node[above]{$a_4$};
   \draw[fill] (4.4,-.8) circle (0.07) node[below]{$a_5$};
   \draw[fill] (3,-.3) circle (0.07) node[below]{$a_6$};
   \draw[fill] (.4,.8) circle (0.07) node[above]{$a_7$};
   \draw[fill] (.4,-.8) circle (0.07) node[below]{$a_8$};
\end{tikzpicture}
\caption{The nodal curve $C_\mu(0)$ for $\mu = (2,1,2,2,1)$.}
\label{fig:limitcurve}
\end{figure}
\end{example}

Let $\sigma \in \symgroup_n$ be the involution 
\begin{equation}
\label{eqn:involution}
\sigma(j) = 
\begin{cases}
j &\quad\text{if $j= \overline{\mu}_i$, $\mu_i=1$} \\
j+1 &\quad\text{if $j= \overline{\mu}_i$, $\mu_i=2$} \\
j-1 &\quad\text{if $j-1= \overline{\mu}_i$, $\mu_i=2$.} \\
\end{cases}
\end{equation}
This is the unique involution such that the curve 
$C_\mu(u)$ is $\xi^\sigma$-fixed.  Moreover,
the $C_\mu(0)$-coordinates in Proposition~\ref{prop:limitcurve}
are $\xi^\sigma$-compatible.

We can now prove Lemma~\ref{lem:specialfibres}.  First, we
establish the analogous result for the limit fibre $\mfibre(C_\mu(0))$.

\begin{lemma}
\label{lem:degeneratespecialfibres}
For $T \in \Tab(\lambda;\mu)$, 
$\mfibre_T(C_\mu(0))$ consists of $2^{\#_\twoskew(T)}$ reduced points. 
If $\#_\twoskew(T) = 0$, the unique point in $\mfibre_T(C_\mu(0))$ is
$\xi^\sigma$-fixed.  If $\#_\twoskew(T) > 0$, then none of the
points in $\mfibre_T(C_\mu(0))$ are $\xi^\sigma$-fixed.
\end{lemma}

\begin{proof}
By Theorem \ref{thm:finitefibreiso}, this is identified with
$\prod_{b =1}^n \mfibrefactor{b}(C_\mu(0))$, where
\[
    \mfibrefactor{b}(C_\mu(0)) = 
   X^{\shape(T|_{\leq b})}_{\shape(T|_{\leq b-1})}(0)
    \cap \Wr^{-1}(\pol^{(b)}(C))\,.
\]
Now,
\[
   \pol^{(b)}(C) =
\begin{cases}
  z^{\mu_1 + \dots +\mu_{b-1}} (z +1) 
   &\quad\text{if $\mu_b = 1$} \\
  z^{\mu_1 + \dots +\mu_{b-1}} (z - \imag)(z+\imag)
   &\quad\text{if $\mu_b = 2$}. \\
\end{cases}
\]
By Lemmas~\ref{lem:calculation1} and~\ref{lem:calculation2}, there is
a unique point in $\mfibrefactor{b}(C_\mu(0))$ when $\shape(T|_b)$ is either 
a single box, or two adjacent boxes.
If $\shape(T|_b)$ has two boxes that
are non-adjacent, of distance $L > 1$ apart, then the discriminant in
Lemma~\ref{lem:calculation2} is
$- (1-L^{-2}) < 0$, so both solutions are non-real.  Thus we see that
$\prod_{b=1}^k \mfibrefactor{b}(C_\mu(0))$ consists of $2^{\#_\twoskew(T)}$
points, which are all non-real, unless $\#_\twoskew(T) = 0$ (in which
case $\prod_{b=1}^k \mfibrefactor{b}(C_\mu(0))$ consists of a single real
point).  

By Proposition~\ref{prop:realchains}, it follows that
$\mfibre_T(C_\mu(0))$ consists of $2^{\#_\twoskew(T)}$ points; 
none are $\xi^\sigma$-fixed, unless $\#_\twoskew(T) = 0$.
\end{proof}

Note that by Lemma~\ref{lem:degeneratespecialfibres}, 
$\mfibre(C_\mu(0)) = \bigcup_{T \in \Tab(\lambda;\mu)} \mfibre_T(C_\mu(0))$ 
consists of $\numsyt\lambda$ distinct (reduced) points.  
Abusing notation slightly, we define
\[
   \mfibre_T(C_\mu(u)) :=
\{ \mfpoint(u) \in \mfibre(C_\mu(u)) \mid
 \lim_{u \to 0} \mfpoint(u) \in \mfibre_T(C_\mu(0))
\}
\,.
\]

\begin{proof}[Proof of Lemma~\ref{lem:specialfibres}]
The curves $C_\mu(u)$ are actually defined over $\CC(u)$,
and hence the points of $\mfibre(C_\mu(u))$ are defined over some algebraic
extension of $\CC(u)$.  This, together with the fact the limit points
over $C_\mu(0)$ are distinct implies that the $\numsyt\lambda$
points of $\mfibre(C_\mu(u))$ are distinct and their coordinates are defined
by power series with a positive radius of convergence.

If $\mfpoint(u) \in \mfamily(\laurentC)$ is $\xi^\sigma$-fixed,
then $\lim_{u \to 0}\mfpoint(u)$ is $\xi^\sigma$-fixed.  Conversely,
if $\mfpoint(u)$ is not $\xi^\sigma$-fixed, then either 
$\lim_{u \to 0}\mfpoint(u)$ is also not $\xi^\sigma$-fixed, or
$\lim_{u \to 0}\mfpoint(u) = \lim_{u \to 0} \xi^\sigma \mfpoint(u)$.
The latter does not occur for points in $\mfibre(C_\mu(u))$, since the points 
of $\mfibre(C_\mu(0))$ are distinct.
Therefore $\mfpoint(u) \in \mfibre(C_\mu(u))$ is $\xi^\sigma$-fixed if and only
if $\lim_{u \to 0} \in \mfibre(C)$ is $\xi^\sigma$-fixed.

Now, define 
\[
  W_T := 
   \{\mtosmap(\mfpoint(u)) \mid 
     \mfpoint(u) \in \mfibre_T(C_\mu(u))\}
\,.
\]
Since $\mtosmap : \mfibre(C) \to \Wr^{-1}(H_\mu)$ is an isomorphism, 
$W_T$ consists of $2^{\#_\twoskew(T)}$ points, defined by power
series with a positive radius of convergence.
To see that the \emph{normalized} \Plucker coordinates are power series, 
note that $\lim_{u \to 0} \boldx(u) \in \scell$ for 
$\boldx(u) \in \Wr^{-1}(H_\mu)$,
which would be false if some normalized \Plucker coordinate involved
a negative power of $u$.
Moreover, this isomorphism establishes property (a).
To see that property (b) holds, let
$\boldx(u) \in W_T$, and write
$\boldx(u) = \mtosmap(\mfpoint(u))$, $\mfpoint(u) \in \mfibre_T(C_\mu(u))$.  Then
\[
\lim_{u \to 0} 
\left(\begin{smallmatrix} 1 & 0 \\[.4ex] 0 & u^{\overline\mu_b} \end{smallmatrix}\right)
\boldx(u)
= \mfpoint(0)^{(b)} 
\,,
\]
which is in $\scellb{\shape(T|_{\leq b})}$
by Theorem~\ref{thm:finitefibreiso}.  
By Proposition~\ref{prop:realstructures},
none of the points of $W_T$ are real except when 
$\#_\twoskew(T) = 0$, so property (c) holds.
\end{proof}

\subsection{Paths in $\mbase$}
\label{sec:mbasepaths}

Fix a composition $\mu = (\mu_1, \dots, \mu_k)$, with $\mu_i \in \{1,2\}$,
and an index $b$ such that $\mu_b = 2$.  Let 
$\mu' = (\mu_1, \dots, \mu_{b-1}, 1,1, \mu_{b+1}, \dots, \mu_k)$.
We now define curves
$G_t(u) \in \mbase$, for each $t \in [0,1]$, $u \in [0, \varepsilon]$,
where $\varepsilon$ is a (sufficiently small) positive real number.
These curves have the property that $G_0(u) = C_{\mu'}(u)$
and $G_1(t) = C_{\mu}(u)$, so for fixed $u$, $t \mapsto G_t(u)$ 
is a path from $C_{\mu'}(u)$ to $C_\mu(u)$ in $\mbase$.

First, suppose $u> 0$.  In this case, we can specify $G_t(u)$ by 
specifying the marked points $a_1, \dots, a_n$.
Let $c := \overline\mu_b$.
For $j \notin \{c,c+1\}$, $a_j$ is independent of $t$, and is
defined by \eqref{eqn:specialcurve}.
The marked points $a_c$ and $a_{c+1}$ depend on $t$, and are defined
as follows.
\begin{align*}
a_c &= \begin{cases}
    (1-t)u^{c} + tu^{c+1} & \quad
    \text{if $0 \leq t \leq \half$} \\
    e^{\pi \imag (t-\half)} (\half u^{c} + \half u^{c+1}) & \quad
    \text{if $\half \leq t \leq 1$}
 \end{cases}
\\[1ex]
a_{c+1} &= \begin{cases}
    tu^{c} + (1-t)u^{c+1} & \quad
    \text{if $0 \leq t \leq \half$} \\
    e^{-\pi \imag (t-\half)} (\half u^{c} + \half u^{c+1}) & \quad
    \text{if $\half \leq t \leq 1$}.
\end{cases}
\end{align*}
For $t \neq \frac{1}{2}$, $a_1, \dots, a_n$ are distinct and distinct from
$\{0,1,\infty\}$ so this uniquely specifies a curve in $\openmbase$.
For $t = \frac{1}{2}$, we have a double marked point, $a_c= a_{c-1}$ 
but all other marked points are distinct, so this is still a 
$\projspace^1$-chain (though not in $\openmbase$).

Note that the marked points $a_c$ and $a_{c+1}$ begin 
at $u^{c}$ and $u^{c+1}$ respectively (when $t=0$).  They come 
together along the real axis and collide 
(when $t = \half$) to produce
a double marked point at $\half(u^c+u^{c+1})$.  Then
they move apart as a conjugate pair along a circle in the complex plane, to end up
at $\pm\tfrac{\imag}{2}(u^c+u^{c+1})$ (when $t=1$).  See 
Figure~\ref{fig:mbasepath}.
\begin{figure}
\centering
\begin{tikzpicture}[x=2.5em,y=2.5em]
   \draw[violet] (0,0) circle(4);
   \draw[violet,dashed] (0,0) ellipse (4 and .5);
   \draw[fill] (-4,0) circle (0.07) node[left]{$0$};
   \draw[fill] (4,0) circle (0.07) node[right]{$\infty$};
   \draw[fill] (0,-.5) circle (0.07) node[below]{$1$};
   \draw[fill] (-.65,-.493) circle (0.04);
   \draw[fill] (-.5,-.496) circle (0.04);
   \draw[fill] (-.35,-.498) circle (0.04);
   \draw[fill] (-3.65,-.205) circle (0.04);
   \draw[fill] (-3.5,-.242) circle (0.04);
   \draw[fill] (-3.35,-.273) circle (0.04);
   \draw[very thick, cyan, -latex] (-105:4 and .5) arc (-105:-120:4 and .5);
   \draw[very thick, cyan, -latex] (-139:4 and .5) arc (-139:-120:4 and .5);
   \draw[very thick, cyan, -latex] (-4,0)++(-7.3:2.02 and 3) arc (-7.3:61:2.02 and 3) 
         node[left,black]{$\tfrac{\imag}{2}(u^c+u^{c+1})$};
   \draw[very thick, cyan, -latex] (-4,0)++(-9.3:2.02 and 3) arc (-9.3:-61:2.02 and 3)
         node[left,black]{$-\tfrac{\imag}{2}(u^c+u^{c+1})$};
   \draw[fill] (-3,-.331) circle (0.07) node[below]{$u^{c+1}$};
   \draw[fill] (-1,-.484) circle (0.07) node[below]{$u^c$};
\end{tikzpicture}
\caption{The path $G_t(u)$ in $\mbase$, for $u > 0$.}
\label{fig:mbasepath}
\end{figure}

For $u=0$ we define the curve $G_t(0) := \lim_{u \to 0} G_t(u)$, for
every $t \in [0,1]$.  For $t=0,1$ we have $G_0(0) = C_{\mu'}(0)$
and $G_1(0) = C_\mu(0)$, as described in Proposition~\ref{prop:limitcurve}.
These curves have $k+2$ and $k+1$ components respectively.
For $t \in (0,1)$, $G_t(0)$ is 
a $\projspace^1$-chain with $k+1$ components, $b_0 = k+1$,
specified by the following coordinate data:
\begin{packeditemize}
\item for $j \notin \{c,c+1\}$,
$b_j$ and $a_j^{(b_j)}$ are the same as for $C_\mu(0)$;
\item $b_c = b_{c+1} = b$, and
\begin{align*}
a_c^{(b)} &= \begin{cases}
    1-t & \quad
    \text{if $0 < t \leq \half$} \\
    \half e^{\pi \imag (t-\half)}  & \quad
    \text{if $\half \leq t \leq 1$}
 \end{cases}
\\[1ex]
a_{c+1}^{(b)} &= \begin{cases}
    t & \quad
    \text{if $0 < t \leq \half$} \\
    \half e^{-\pi \imag (t-\half)}  & \quad
    \text{if $\half \leq t \leq 1$}.
\end{cases}
\end{align*}
\end{packeditemize}
See Figure~\ref{fig:limitmbasepath}.
Note that as $t \to 0$, $a_{c+1}$ approaches a node; hence at $t=0$, 
a new component forms such that $a_{c}$ and $a_{c+1}$ are on different 
components.  
\begin{figure}
\centering
\begin{tikzpicture}[x=2.5em,y=2.5em]
   \draw[violet] (1,0) circle(1);
   \draw[violet,dashed] (1,0) ellipse (1 and 0.2);
   \draw[violet] (12,0) circle(1);
   \draw[violet,dashed] (12,0) ellipse (1 and 0.2);
   \draw[fill] (0,0) circle (0.07) node[left]{$0$};
   \draw[fill] (12,-.2) circle (0.07) node[below]{$1$};
   \draw[fill] (13,0) circle (0.07) node[right]{$\infty$};
   \draw[violet] (4,0) circle(1);
   \draw[violet,dashed] (4,0) ellipse (1 and 0.2);
   \draw[violet] (6.5,0) circle(1.5);
   \draw[violet,dashed] (6.5,0) ellipse (1.5 and 0.3);
   \draw[violet] (9,0) circle(1);
   \draw[violet,dashed] (9,0) ellipse (1 and 0.2);
   \draw[very thick, cyan, -latex] (6.5,0)++(-180:1.5 and .3) arc (-180:-120:1.5 and .3);
   \draw[very thick, cyan, -latex] (6.5,0)++(-90:1.5 and .3) arc (-90:-120:1.5 and .3);
   \draw[very thick, cyan, -latex] (5,0)++(-8:.77 and 1.5) arc (-8:50:.77 and 1.5);
   \draw[very thick, cyan, -latex] (5,0)++(-12:.77 and 1.5) arc (-11:-50:.77 and 1.5);
   \draw[fill] (5,0) circle (0.07) node[left]{$a_{c+1}$};
   \draw[fill] (6.5,-.3) circle (0.07) node[below]{$a_c$};
   \draw (2.5,0) node {$\cdots$};
   \draw (10.5,0) node{$\cdots$};
\end{tikzpicture}
\caption{The limiting path $G_t(0)$ in $\mbase$.}
\label{fig:limitmbasepath}
\end{figure}

\begin{proposition}
The map 
\begin{gather*}
G: [0,1] \times [0, \varepsilon] \to \mbase\\
(t,u) \mapsto G_t(u)
\end{gather*}
is continuous.
\end{proposition}

\begin{proof}
Using the embedding
\eqref{eqn:projectiveembedding},
we must show that for
each $(p,q,r,s)$, the map $(t,u) \mapsto \theta_{p,q,r,s}(G_t(u))$
is continuous.  Since $G_t(u)$ is a $\projspace^1$-chain for all
$(t,u)$, it suffices to show this for 
$(p,q,r,s) = (0,a_{j_1}, \infty, a_{j_2})$, $j_1 < j_2$.  
This is straightforward.
For example, in the case $(p,q,r,s) = (0,a_c,\infty,a_{c+1})$,
we find that $\theta_{p,q,r,s}(G_t(u)) = \frac{a_{c+1}}{a_c}$,
which is a ratio of two continuous functions, and the denominator
is never zero.
\end{proof}

All of the curves $G_t(u)$ are real, but not all with respect to
the same real structure.
Define $\sigma \in \symgroup_n$ as in~\eqref{eqn:involution}, and 
define $\sigma' \in \symgroup_n$ analogously, with $\mu'$ in place of $\mu$.
For $(t,u) \in [0, \half] \times [0, \varepsilon]$ we have 
$G_t(u) \in \mbase(\RR^{\sigma'})$,
and for $(t,u) \in [\half,1] \times [0, \varepsilon]$ we have 
$G_t(u) \in \mbase(\RR^{\sigma})$.
If $t = \half$, the curve is real with respect to both real
structures, which is possible because $G_{1/2}(u)$ has a double marked
point, and is therefore not in $\openmbase$.

\subsection{Paths in $\mfamily$}
\label{sec:mfamilypaths}

We now lift the family of paths $G_t(u)$ in $\mbase$ to a family
of paths in $\mfamily$.  Projecting to $\scell$ will give the paths
we need for Lemmas~\ref{lem:SYTsigns} and~\ref{lem:MNsigns},
thereby allowing us to prove these statements.

Let $\mu, \mu'$ be as in the previous section. Suppose $T' \in \MN(\lambda;\mu')$.  
Let $T \in \Tab(\lambda; \mu)$
be the tableau obtained by decrementing all entries $b+1, \dots, k+1$.
Then $T$ may or may not be a Murnaghan--Nakayama tableau: we have
$T \in \MN(\lambda; \mu)$ if and only if $b$ and $b+1$
are adjacent in $T'$.  We consider these two cases separately.

First, suppose $T \in \MN(\lambda; \mu)$.  Then 
$T$ and $T'$ are as in the
statement of Lemma~\ref{lem:MNsigns}. In this case, $G_t(u)$ lifts 
isomorphically to a family of curves in $\mfamily$, as shown 
in Figure~\ref{fig:onelift}.
\begin{figure}
    \centering
    \begin{tikzpicture}[x=1em,y=1em]
    \node[inner sep=0] at (0,0) 
    {\includegraphics[width=28em]{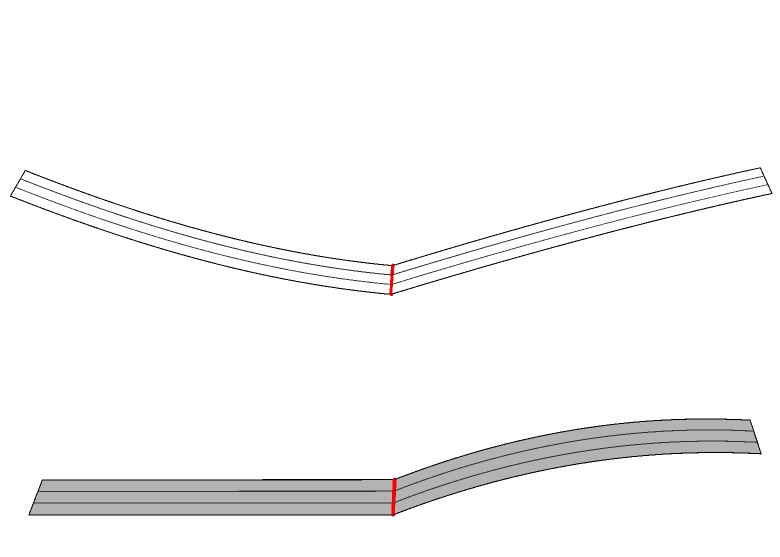}};
    \node at (-14.9,-7.7) {$G_t(u)$};
    \node at (-14.9,3.7) {$\Gamma_t(u)$};
	\node at (-12.6,-6.2) {\small $t=0$};
	\node at (-0.2,-5.6) {\small $t=\half$};
    \node at (13,-4.2) {\small $t=1$};
    \node at (-13,-9.4) {\small $C_{\mu'}(u)$};
    \node at (13.3,-7.2){\small $C_{\mu}(u)$};
	\node at (-14,1.3) {\small $\mfibre_{T'}(C_{\mu'}(u))$};
    \node at (13.3,1.7){\small $\mfibre_{T}(C_{\mu}(u))$};
    \draw[thick,-latex] (-6,-1.7) -- (-6,-5.7) node [midway,right] {$\Psi$};
    \end{tikzpicture}
    \caption{The lift $\Gamma$ of $G_t(u)$ constructed in Lemma \ref{lem:MNlift}. The endpoints of $G_t(u)$ are $C_{\mu'}(u)$ (at $t=0$) and $C_\mu(u)$ (at $t=1$), which lift to $\mfibre_{T'}(C_{\mu'}(u))$ and $\mfibre_{T}(C_\mu(u))$ respectively. 
The fibre over $t=\tfrac{1}{2}$ maps to $\hdominovar$ or $\vdominovar$ 
in $\scell(\RR)$, depending on $\shape(T|_b)$.}
    \label{fig:onelift}
\end{figure}

\begin{lemma}[Path lifting, $\vdomino/\hdomino$ case]
\label{lem:MNlift}
For sufficiently small $\varepsilon > 0$, there exists a continuous 
map
\begin{gather*}
\varGamma : [0,1] \times [0, \varepsilon] \to \mfamily \\
\qquad (t,u) \mapsto 
\varGamma_t(u)
\,,
\end{gather*}
with the following properties.
\begin{enumerate}[(a)]
\item \textbf{(Path lifting.)}
For all $(t,u)$, $\mmap(\varGamma_t(u)) = G_t(u)$.  
\item \textbf{(Connecting $T$ to $T'$.)}
$\varGamma_0(u) \in \mfibre_{T'}(C_{\mu'}(u))$
and $\varGamma_1(u) \in \mfibre_T(C_{\mu}(u))$.
\item \textbf{(No ramification.)}
For all $(t,u)$, 
$\varGamma_t(u)$ is a reduced point of
the fibre $\mfibre(G_t(u))$.
\item \textbf{(Compatibility with real structure.)}
$\varGamma_t(u) \in \mfamily(\RR^{\sigma'})$ for $t \in [0, \half]$,
and $\varGamma_t(u) \in \mfamily(\RR^\sigma)$ for $t \in [\half,1]$.
\item \textbf{(Crossing $\vdominovar$ or $\hdominovar$.)}
For $u >0$, 
$\mtosmap(\varGamma_{1/2}(u)) \in \hdominovar$ if $\shape(T|_b)$
is a horizontal domino,
and $\mtosmap(\varGamma_{1/2}(u)) \in \vdominovar$ if $\shape(T|_b)$
is a vertical domino.
\end{enumerate}
\end{lemma}

\begin{proof}
By Lemma~\ref{lem:degeneratespecialfibres}, there is a unique reduced
point $\varGamma_0(0) \in \mfibre_{T'}(G_0(0))$ and by a similar argument,
there is a unique reduced point $\varGamma_t(0) \in \mfibre_T(G_0(0))$
for all $t \in (0,1]$.  This defines a continuous path 
$\varGamma_t(0)$, $t\in[0,1]$. 
As in the proof of Lemma~\ref{lem:degeneratespecialfibres} we see
that $\varGamma_t(0)$ is $\xi^{\sigma'}$-fixed for
$0 \leq t \leq \half$, and $\xi^\sigma$-fixed for $\half \leq t \leq 1$.
When $t =\half$, the curve $G_{1/2}(0)$ has a double marked point
$a_c = a_{c+1}$, $c= \overline\mu_b$. 
By Corollary~\ref{cor:doublemarkedpoint},
$\mtosmap_c(\varGamma_{1/2}(0)) \in X_\hdomino(1) \cap \bigfibrefactor{b}$ or 
$\mtosmap_c(\varGamma_{1/2}(0)) \in X_\vdomino(1) \cap \bigfibrefactor{b}$.
But by Lemma~\ref{lem:generalschubertintersection},
if $\shape(T|_b)$ is a vertical domino then
$X_\hdomino(1) \cap \bigfibrefactor{b}$ is empty, and if $\shape(T|_b)$ is
a vertical domino, then $X_\vdomino(1) \cap \bigfibrefactor{b}$ is empty.
Thus we must have $\mtosmap_c(\varGamma_{1/2}(0)) \in X_\hdomino(1)$ if
$\shape(T|_b)$ is a horizontal domino, and 
$\mtosmap_c(\varGamma_{1/2}(0)) \in X_\vdomino(1)$ if
$\shape(T|_b)$ is a vertical domino.

Since all points of these fibres are reduced
and the map $\mmap : \mfamily \to \mbase$ is finite,
this extends uniquely to a continuous 
family $\varGamma: [0,1] \times [0, \varepsilon] \to \mfamily$, 
satisfying (a) and (c), for some sufficiently small $\varepsilon$.  
By construction, (b) is also satisfied.
For fixed $t \in [0,\half]$, $G_t(u)$ is $\xi^{\sigma'}$-fixed for all 
$u$.  Thus, $\varGamma_t(u)$ can only cease to be $\xi^{\sigma'}$-fixed 
at a double point of the fibre $\mfibre(G_t(u))$.  Since
property (c) ensures that there are no such points for
$u \in [0, \varepsilon]$, $\varGamma_t(u)$ is $\xi^{\sigma'}$-fixed
for all such $u$.  A similar argument holds for $t \in [\half, 1]$,
which establishes (d).
When $t =\half$, $a_c = a_{c+1}$ is a double marked point of the curve
$G_{1/2}(u)$; 
by continuity, $\mtosmap_c(\varGamma_{1/2}(u)) \in X_\hdomino(1)$ if 
$\shape(T|_b)$ is a horizontal domino, and 
$\mtosmap_c(\varGamma_{1/2}(u)) \in X_\vdomino(1)$ if 
$\shape(T|_b)$ is a vertical domino.
But on $\openmfamily$, $\mtosmap_c$ and $\mtosmap$ are related by
a transformation $\phi \in \borelgroup \subset \PGL_2(\CC)$, so the previous
statement implies (e).
\end{proof}

\begin{proof}[Proof of Lemma~\ref{lem:MNsigns}]
Let $\varGamma_t(u)$ be as in Lemma~\ref{lem:MNlift}.
Consider the path $\gamma : [0,1] \to \scell(\CC)$, defined by
$\gamma_t := \mtosmap(\varGamma_t(\varepsilon))$, $t \in [0,1]$.
First note that $\gamma_t$ is in fact a path in $\scell(\RR)$, by
property (d) and Proposition~\ref{prop:realstructures}.  
Next note that if follows from property (b) and the definition of 
$W_T$ (see Proof of Lemma~\ref{lem:specialfibres}),
$\mtosmap(\varGamma_0(u)) \in W_T$; therefore
$\gamma_0 = \MNpoint_{T'}$.  Similarly $\gamma_1 = \MNpoint_T$.

Let $g_t := \Wr(\gamma_t)$ the image of the path $\gamma_t$ in
$\monics(\RR)$. Note that we also have $g_t = \pol(G_t(\varepsilon))$ 
By property (c) and
Proposition~\ref{prop:familydiagram}, $\gamma_t$ is a reduced point 
of $\Wr^{-1}(g_t)$.  In particular,
this means $\gamma_t \notin \critvar(\RR)$, for all $t \in (0,1)$.

By property (e), 
$\gamma_{1/2} \in \hdominovar(\RR)$ 
if $\shape(T|_b)$
is a horizontal domino, and 
$\gamma_{1/2} \in \vdominovar(\RR)$ 
if $\shape(T|_b)$ is a vertical domino.
Since $t= \half$ is the only value of $t$ for which $g_t$ has a repeated
root, there are no other crossings of either of these varieties.
To see that the crossing at $t= \half$ is a simple crossing, 
note that since $\gamma_{1/2} \notin \critvar(\RR)$, 
$\Wr : \scell(\RR) \to \monics(\RR)$ is a diffeomorphism in a neighbourhood
of $\gamma_{1/2}$.  Thus $\gamma_t$ has a simple crossing of 
$\hdominovar(\RR)$ or $\vdominovar(\RR)$ at $t = \half$
if and only if $g_t$ has a
a simple crossing of the discriminant variety $\discrimvar(\RR)$.
Since $g_t = \pol(G_t(\varepsilon))$ is given completely explicitly,
it is straightforward to check that this is a simple crossing.
\end{proof}

Now, suppose $T \notin \MN(\lambda;\mu)$. This means that 
$b$ and $b+1$ are non-adjacent in $T'$.  Therefore switching the
positions of these entries results in another 
tableau $T''$.  Note that both $T', T'' \in \MN(\lambda;\mu')$,
and are related as in the statement of Lemma~\ref{lem:SYTsigns}.
Moreover, this relation is symmetrical, and both are related to $T$ 
in the same way. 

Note that since $T$ is not a Murnaghan--Nakayama tableau, 
by Lemma \ref{lem:degeneratespecialfibres}
the fibre $\mfibre_T(G_1(0)) = \mfibre_T(C_\mu(0))$ has no
$\xi^\sigma$-fixed points.  We will therefore not be able to 
lift all $G_t(u)$ to $\mfamily$, in a way that satisfies properties
(a)--(d) of Lemma~\ref{lem:MNlift}, because we already know (d) must be
false at $(t,u) = (1,0)$.  Instead, we restrict the domain
from $[0,1] \times [0,\varepsilon]$ to the subset over which (d)
will hold.  We then obtain two different lifts of $G_t(u)$
over this domain, which are associated to the two tableaux 
$T'$ and $T''$. See Figure \ref{fig:twolifts}.
\begin{figure}
    \centering
    \begin{tikzpicture}[x=1em,y=1em]
    \node[inner sep=0] at (0,0) 
    {\includegraphics[width=28em]{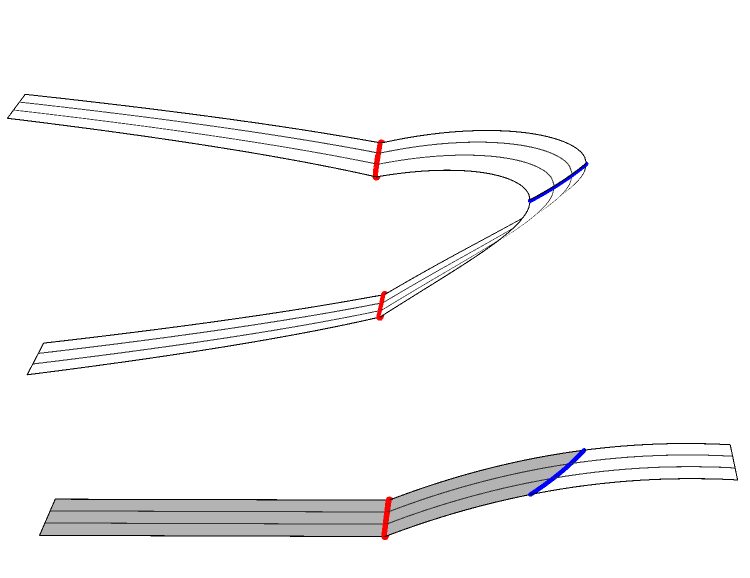}};
    \node at (-15.1,-8.5) {$G_t(u)$};
    \node at (-15.1,-2.8) {$\Gamma''_t(u)$};
    \node at (-15.1,7) {$\Gamma'_t(u)$};
	\node at (-11.6,-7) {\small $t=0$};
	\node at (.5,-6.6) {\small $t=\half$};
	\node at (7.7,-5.2) {\small $\tmax$};
	\node at (13,-5) {\small $t=1$};
    \node at (-12,-10) {\small $C_{\mu'}(u)$};
    \node at (13.3,-8){\small $C_{\mu}(u)$};
	\draw[thick,-latex] (-5,-3) -- (-5,-7) node [midway,right] {$\Psi$};
    \end{tikzpicture}
    \caption{The two lifts $\Gamma', \Gamma''$ of $G_t(u)$ constructed in Lemma \ref{lem:SYTlift}. The fibres over $t=\tfrac{1}{2}$ map to $\hdominovar$ and $\vdominovar$ in $\scell(\RR)$, while the fibre over $t = \tmax(u)$ maps to $\critvar(\RR)$.}
    \label{fig:twolifts}
\end{figure}

\begin{lemma}[Path lifting, $\twoskew$ case]
\label{lem:SYTlift}
For sufficiently small $\varepsilon > 0$,
there exists a subset $K \subset [0,1] \times [0, \varepsilon]$
and continuous maps 
\[
\begin{gathered}
\varGamma' : K \to \mfamily \\
(t,u) \mapsto \varGamma'_t(u)
\end{gathered}
\qquad\qquad\qquad\qquad
\begin{gathered}
\varGamma'' : K \to \mfamily \\
(t,u) \mapsto \varGamma''_t(u)\,,
\end{gathered}
\]
with the following properties.
\begin{enumerate}[(a)]
\item \textbf{(Path lifting.)}
For all $(t,u) \in K$,
$\mmap(\varGamma'_t(u)) = \mmap(\varGamma''_t(u)) = G_t(u)$.
\item \textbf{(Shape of $K$.)}
$K$ is of the form
\[
   K = \{(t,u) \mid u \in [0,\varepsilon], t \in [0, \tmax(u)]\}
\,.
\]
where $\tmax : [0, \varepsilon] \to (\half, 1)$ is a function, in particular $\tmax > \half$.
\item \textbf{(Starting from $T'$ and $T''$.)}
$\varGamma'_0(u) \in \mfibre_{T'}(C_{\mu'}(u))$
and
$\varGamma''_0(u) \in \mfibre_{T''}(C_{\mu'}(u))$.
\item \textbf{(No ramification until $\tmax$.)}
For $(t,u) \in K$ with $t < \tmax(u)$,
$\varGamma'_t(u)$ and $\varGamma''_t(u)$ are reduced points of
the fibre $\mfibre(G_t(u))$.
\item \textbf{(Paths join at $\tmax$.)}
For $t = \tmax(u)$, $\varGamma'_t(u) = \varGamma''_t(u)$ is a double 
point of the fibre $\mfibre(G_t(u))$.
\item \textbf{(Compatibility with real structure.)}
$\varGamma'_t(u) \in \mfamily(\RR^{\sigma'})$ for $t \in [0, \half]$,
and $\varGamma'_t(u) \in \mfamily(\RR^\sigma)$ for $t \in [\half, \tmax(u)]$.  
The
same holds for $\varGamma''$.
\item \textbf{(Crossing $\vdominovar$ and $\hdominovar$.)}
Either $\mtosmap(\varGamma'_{1/2}(u)) \in \hdominovar$ and
$\mtosmap(\varGamma''_{1/2}(u)) \in \vdominovar$ for all $u > 0$, 
or vice-versa.
\end{enumerate}
\end{lemma}

\begin{proof}
First, we compute $\mfibre_T(G_t(0))$, for $t \in (0,1]$.  Proceeding as
in Lemma~\ref{lem:degeneratespecialfibres}, we find that $\mfibre_T(G_t(0))$
consists of two distinct $\xi^{\sigma'}$-fixed points if $t \in (0,\half]$.
The limit fibre $\lim_{t \to 0} \mfibre_T(G_t(0)) =
\mfibre_{T'}(C_{\mu'}(0)) \cup \mfibre_{T''}(C_{\mu'}(0))$; again we have two distinct
$\xi^{\sigma'}$-fixed points.
For $t \in [\half,1]$, 
the discriminant in Lemma~\ref{lem:calculation2} for computing
$\mfibrefactor{b}(G_t(0))$ 
is equal to 
$L^{-2} - \sin^2 \big(\pi(t - \half)\big)$;
hence $\mfibre_T(G_t(0))$ consists of:
\begin{packeditemize}
\item two distinct $\xi^{\sigma}$-fixed points 
if $t < \half + \tfrac{1}{\pi}\arcsin(L^{-1})$;
\item a double $\xi^{\sigma}$-fixed point 
if $t = \half + \tfrac{1}{\pi}\arcsin(L^{-1})$;
\item two distinct $\xi^{\sigma}$-conjugate points 
if $t > \half + \tfrac{1}{\pi}\arcsin(L^{-1})$.
\end{packeditemize}
where $L$ is the distance between the entries $b$ and $b+1$ in $T'$. 

Let $J \subset \mbase$ denote the image of the map
$[0,1] \to \mbase$, $t \mapsto G_t(0)$.
Then $\mmap^{-1}(J) \cap \bigfibre_T$ is the union of
all $\mfibre_T(G_t(0))$, $t \in (0,1]$, together with 
$\lim_{t \to 0} \mfibre_T(G_t(0)) = 
\mfibre_{T'}(C_{\mu'}(0)) \cup \mfibre_{T''}(C_{\mu'}(0))$.
Since $\mmap^{-1}(J) \cap \bigfibre_T$ is 
a closed subset of $\mmap^{-1}(J)$, we can find an open
subset  $\calU \subset \mfamily$, containing 
$\mmap^{-1}(J) \cap \bigfibre_T$ such that
$\calU$ is a component on $\mmap(\mmap^{-1}(\calU))$.

Assuming $\varepsilon > 0$ is sufficiently small, 
$\mfibre(G_t(u)) \cap \calU$ is then
a finite scheme of length $2$, for all 
$(t,u) \in [0,1] \times [0, \varepsilon]$.
Hence this is either two distinct points, or a 
(non-reduced) double point.  
By the same argument as in Lemma~\ref{lem:MNlift},
$\mfibre(G_t(u)) \cap \calU$ consists of two distinct points in 
$\mfamily(\RR^{\sigma'})$ for $t \in [0,\half]$, $u>0$.
Moreover, for $t = \half$, $u > 0$, the two points of $\mfibre(G_t(u)) \cap \calU$  
are also in $\mfamily(\RR^{\sigma'})$, and for $t = 1$ they are not in 
$\mfamily(\RR^{\sigma'})$.  This implies that for all $u$ there exists
a $t \in (\half, 1)$ such that $\mfibre(G_t(u)) \cap \calU$ is a double
point.  We put
\[
\tmax(u) := \min\{t \in [0,1] \mid \text{$\mfibre(G_t(u)) \cap \calU$ 
     is a double point}\}
\,,
\]
and take (b) to be the definition of $K$.  Note that $K$ is simply
connected, and therefore $\mmap : \mmap^{-1}(K) \cap \calU \to K$ is a 
topologicially trivial two-to-one covering map away from $t = \tmax(u)$.

We can therefore define $\varGamma' : K \to \mfamily$ to be the unique 
map such that $\varGamma'_0(0)$ is the unique point in 
$\mfibre_{T'}(C_{\mu'}(0))$, and $\varGamma'_t(u) \in P(G_t(u)) \cap \calU$
for all $(t,u) \in K$.  Similarly we define
$\varGamma'' : K \to \mfamily$, with $T''$ in place of $T'$.
Properties (a), (c), (d), (e) are immediate, and the proof of (f)
is identical to the proof of (d) in Lemma~\ref{lem:MNlift}.

When $t =\half$, 
the curve $G_{1/2}(0)$ has a double marked point
$a_c = a_{c+1}$, $c = \overline\mu_b$. 
By Corollary~\ref{cor:doublemarkedpoint}, either
$\mtosmap_c(\varGamma'_{1/2}(0)) \in X_\hdomino(1) \cap \bigfibrefactor{b}$ or 
$\mtosmap_c(\varGamma'_{1/2}(0)) \in X_\vdomino(1) \cap \bigfibrefactor{b}$,
and the same is true for $\varGamma''$.
By Lemma~\ref{lem:generalschubertintersection}, both $X_\hdomino(1) \cap \bigfibrefactor{b}$ and
$X_\vdomino(1) \cap \bigfibrefactor{b}$ consist of a single point.
Recall that $\mfibre_T(G_{1/2}(0))$ is identified with 
$\prod_{i=1}^k \mfibrefactor{i}(G_{1/2}(0))$.  By 
Lemmas~\ref{lem:calculation1} and~\ref{lem:calculation2},
every term in this product
consists of a single point, except $\mfibrefactor{b}(G_{1/2}(0))$ which has two points.
$\varGamma'_{1/2}(0)$ and $\varGamma''_{1/2}(0)$ are distinct points of
$\mfibre_T(G_{1/2}(0))$, and the only coordinate
on which they can differ is
$\varGamma'_{1/2}(0)^{(b)} \neq \varGamma''_{1/2}(0)^{(b)}$.
Equivalently $\mtosmap_c(\varGamma'_{1/2}(0)) \neq 
\mtosmap_c(\varGamma''_{1/2}(0))$.
It follows that we have either 
$\mtosmap_{\overline\mu_b}(\varGamma'_{1/2}(0)) 
\in X_\hdomino(1) \cap \bigfibrefactor{b}$ 
and $\mtosmap_{\overline\mu_b}(\varGamma''_{1/2}(0)) 
\in X_\vdomino(1) \cap \bigfibrefactor{b}$, or the other way around.
Arguing now as in the proof of Lemma~\ref{lem:MNlift}, we deduce (g).
\end{proof}

\begin{remark}
With a more thorough analysis of $\mfibre_T(G_t(0))$, $t \in [0, \half]$,
property (g) in Lemma~\ref{lem:MNlift} 
can be replaced by the following
stronger statement.  If $b$ is left of $b+1$ in $T'$ then 
$\mtosmap(\varGamma'_{1/2}(u)) \in \hdominovar$ and 
$\mtosmap(\varGamma''_{1/2}(u)) \in \vdominovar$.  If $b$ is right
of $b+1$, the identification is the other way around.  We omit the proof,
since we do not actually need this fact here. Note that in each case, the two boxes $b, b+1$ rectify to the corresponding domino shape (compare with \cite[Theorem 6.4]{Pur-Gr} and \cite[Theorem 4.4]{Sp}).
\end{remark}

\begin{proof}[Proof of Lemma~\ref{lem:SYTsigns}]
Let $K$, $\tmax$, $\varGamma'_t(u)$ and $\varGamma''_t(u)$ be as in 
Lemma~\ref{lem:SYTlift}.  
We define the path $\gamma: [0,1] \to \scell(\CC)$ by joining together
the paths $\mtosmap(\varGamma'_t(\varepsilon))$ and 
$\mtosmap(\varGamma''_t(\varepsilon))$:
\[
 \gamma_t :=
\begin{cases}
\mtosmap(\varGamma'_{2t}(\varepsilon))
&\quad\text{for $t \in [0,\tfrac{1}{4}]$} \\
\mtosmap(\varGamma'_{1/2+(\tmax(\varepsilon)-1/2)(4t-1)}(\varepsilon))
&\quad\text{for $t \in [\tfrac{1}{4},\half]$} \\
\mtosmap(\varGamma''_{1/2+(\tmax(\varepsilon)-1/2)(3-4t)}(\varepsilon))
&\quad\text{for $t \in [\half, \tfrac{3}{4}]$} \\
\mtosmap(\varGamma''_{2-2t}(\varepsilon))
&\quad\text{for $t \in [\tfrac{3}{4},1]$}. \\
\end{cases}
\]
Let $g_t = \Wr(\gamma_t)$; note that $g_t = g_{1-t}$.

Many pieces of the argument are essentially the same as in 
the proof of Lemma~\ref{lem:MNsigns}.
By property (f), $\gamma_t \in \scell(\RR)$ for all $t \in [0,1]$.
By property (c),
$\gamma_0 = \mtosmap(\varGamma'_0(\varepsilon))= \MNpoint_{T'}$ and
$\gamma_1 = \mtosmap(\varGamma''_0(\varepsilon))= \MNpoint_{T''}$.
By property (g), one of 
$\gamma_{1/4} = \mtosmap(\varGamma'_{1/2}(\varepsilon))$ and 
$\gamma_{3/4} = \mtosmap(\varGamma''_{1/2}(\varepsilon))$ 
is in $\hdominovar(\RR)$ and the other is in $\vdominovar(\RR)$, and 
these are the only points $\gamma_t$ in either of these varieties.
These are simple crossings of $\hdominovar(\RR)$
and $\vdominovar(\RR)$, by the same reasoning as in the proof
of Lemma~\ref{lem:MNsigns}.  

We now consider crossings of $\critvar(\RR)$.
By property (e),
$\gamma_{1/2} = \mtosmap(\varGamma'_{1/2}(\varepsilon)) = 
\mtosmap(\varGamma'_{1/2}(\varepsilon))$ is a double point
of $\Wr^{-1}(g_{1/2})$; thus $\gamma_{1/2} \in \critvar(\RR)$.
This
is certainly a crossing, because 
$\gamma_{1/2-\epsilon}$ and $\gamma_{1/2+\epsilon}$ are two different
points of the fibre over $g_{1/2-\epsilon}= g_{1/2+\epsilon}$,
and therefore in different components of some neighbourhood of 
$\Wr^{-1}(g_{1/2}) \setminus \critvar(\RR)$.  
By Lemma~\ref{lem:ramificationdivisor}, $\gamma_{1/2}$ is a smooth
point of $\critvar(\RR)$, and therefore this is a simple crossing.
By property (d), $\gamma_t \notin \critvar$ for $t \neq \half$.
\end{proof}

\section{Upper and lower bounds}
\label{sec:bounds}

\subsection{Bounds from Theorem~\ref{thm:main}}

For $g \in \monics(\RR)$, let $N_g$ denote the number of real points 
in the fibre $\Wr^{-1}(g)$, counted with algebraic multiplicity.
One consequence of Theorem~\ref{thm:main} is that we
obtain bounds on $N_g$.  When the lower bound is non-zero, this
gives a proof of the existence of real solutions to the equation
$\Wr(f_1, \dots, f_d) = g$. 

\begin{restatetheorem}{Corollary~\ref{cor:bounds}}
If $g \in \overline{\monics(\mu)}$, then 
\[
|\sgchar(\mu)| \leq N_g \leq \numsyt\lambda
\,.
\]
\end{restatetheorem}

\begin{proof}
If $g \in \monics(\mu)$ and $g$ is a regular value of the Wronski map, 
the the multiplicities of the points in $\Wr^{-1}(g)$ are all $1$, and
the topological degree of the map 
$\Wr:\scell(\mu) \to \monics(\mu)$ is a signed count of the points
in $\Wr^{-1}(g)$.  Thus in this case, the lower bound 
$N_g \geq |\sgchar(\mu)|$ holds.
Since $N_g$ is an upper semi-continuous function of $g$, 
the result remains true if we pass to the closure.

For the upper bound, note that $\numsyt\lambda$ is the number of complex 
points in any fibre of the map $\Wr : \scell(\CC) \to \monics(\CC)$,
and the real points in the fibre are subset of the complex points.
\end{proof}

If $g \in \overline{\monics(\mu)}$ for more than one
$\mu$ (which occurs when $g$ has repeated real roots), then we get 
more than one lower bound.  
In general, there does not appear to be any
simple rule as to which $\mu$ will give the best lower bound.
Of course, if all the roots of $g$ are all 
real, the best lower bound comes from $\mu=1^n$, and
$N_g = \numsyt\lambda$, regardless of whether 
$g$ has repeated roots.
In particular, the upper bound 
in Corollary~\ref{cor:bounds} is tight for all $\mu$, since it
is achieved when all roots of $g$ are real, and such $g$ exist
in $\overline{\monics(\mu)}$.  Moreover, we can find such $g$
such that $\Wr^{-1}(g)$ is reduced, which implies that the upper bound 
can also be achieved by polynomials properly in $\monics(\mu)$,
by starting with a polynomial with only real roots in the closure,
and perturbing.  

In some cases, we can also see that the lower bound is tight.

\begin{theorem}
\label{thm:tight}
Let $\mu =(\mu_1, \dots, \mu_k)$ be a composition of $n$, 
with $\mu_i \in \{1,2\}$.  
If all tableaux in $\MN(\lambda; \mu)$ have the same sign, then
the lower bound in Corollary~\ref{cor:bounds} is tight.  That is, there
exists a polynomial $g \in \monics(\mu)$ such that 
$N_g = |\sgchar(\mu)|$.
\end{theorem}

\begin{proof}
Consider $h_\mu \in \monics(\mu)$.  
By Corollary~\ref{cor:realfibrepoints},
$N_{h_\mu} = \#\MN(\lambda; \mu)$. But if all tableaux in $\MN(\lambda; \mu)$
have the same sign, then $\#\MN(\lambda; \mu) = |\sgchar(\mu)|$.
\end{proof}

Note that for a given partition, there are many compositions
with that associated partition.  If any one of these compositions 
satisfies the hypothesis of Theorem~\ref{thm:tight}, then we
conclude the lower bound is tight.

\begin{example}
\label{ex:tightcases}
Here a few noteworthy cases where all tableaux in $\MN(\lambda; \mu)$
have the same sign, and hence Theorem~\ref{thm:tight} tells us that
the lower bound is tight.
\begin{enumerate}[(i)]
\item If $\lambda \vdash n$ is any partition, and $\mu = (1^n)$.  This 
is the all real roots
case, where the lower bound is exact.

\item 
If $|\lambda|$ is even, and $\mu = (2^{n_2})$.

\item 
If $|\lambda|$ is odd, and $\mu = (1,2^{n_2})$.  Cases
(ii) and (iii) are where we have the maximal number of complex roots.

\item 
If $\lambda = (\lambda_1, \lambda_2) \vdash n$ is any $2$-part partition,
and $\mu = (1, 2^l, 1^{n-2l-1})$, $0 \leq l < n/2$.
This, together with (ii), shows the lower bound is always tight if
$\lambda$ has two parts.

\item 
If $\lambda = m^d$ is a rectangle with $dm$ even, 
 and $\mu = (1,2^{dm/2-1},1)$. 

\item 
If $\lambda = (d, d-1, \dots, 2,1)$ is a staircase, 
 and $\mu = (\mu_1, \dots, \mu_k)$ is any composition 
such that $\mu_k =2$.
This, together with (i), shows that the lower bound is always tight if
$\lambda$ is a staircase.

\item 
If $\lambda = (\lambda_1, \dots, \lambda_d) \vdash n$ 
is any partition such that $\lambda_i-\lambda_{i+1}$ is odd for
all $i=1, \dots, d-1$, and $\mu = (1^{n_1}, 2^{n_2})$, where $n_1+2n_2 = n$.
This generalizes (vi): the lower bound is tight for all such 
partitions $\lambda$.
\end{enumerate}
In case (vi), $\sgchar(\mu) = \#\MN(\lambda;\mu) = 0$, and
Theorem~\ref{thm:tight} is asserting that the Wronski 
map $\Wr: \scell(\mu) \to \monics(\mu)$ is not surjective.
This also happens in (v), in the case where $m$ and $d$ are both even,
as was first shown in~\cite{EG-pole}.
\end{example}

\subsection{Eremenko and Gabrielov's lower bound}
\label{sec:EG-deg}

Eremenko and Gabrielov computed degrees of real Wronski maps 
in~\cite{EG-deg}.  Their result is stated for the projective real Wronski
map $\Wr : \Gr(d, \RR^{d+m}) \to \projspace^{dm}(\RR)$,
but as is pointed out in \cite{SS}, the
computation can be done on any Schubert cell 
and we state the result as such here.
We include a concise proof, based on Lemma~\ref{lem:SYTsigns}.

\begin{definition}
Let $T \in \SYT(\lambda)$.  An \defn{inversion} in $T$ is a pair of cells
$(i,j),\ (i',j') \in \lambda$ such that $i<i'$ and $T(i,j) > T(i',j')$.
Let $\inv(T)$ denote the number of inversions in $T$.
\end{definition}

\begin{theorem}[Eremenko--Gabrielov]
\label{thm:EG-deg}
With respect to the ambient orientation, the topological degree of
the Wronski map $\Wr : \scell(\RR) \to \monics(\RR)$ is
\[
    I_\lambda := \sum_{T \in \SYT(\lambda)} (-1)^{\inv(T)}
\,.
\]
\end{theorem}

\begin{proof}
We claim that for every tableau $T \in \SYT(\lambda)$, 
$\ambsgn(\MNpoint_T) = (-1)^{\inv(T)}$.
If $T = T_0$, this is certainly true.  
Suppose $T, T' \in \SYT(\lambda)$ are two tableaux related as in 
Lemma~\ref{lem:SYTsigns}.  There is path joining
$\MNpoint_T$ to $\MNpoint_{T'}$, 
along which the ambient sign changes exactly once,
at the simple crossing of $\critvar(\RR)$.
Therefore $\ambsgn(\MNpoint_T) = -\ambsgn(\MNpoint_{T'})$.
We also have $\inv(T) = \inv(T') \pm 1$, and the claim follows.
Therefore, the topological degree of the Wronski map is
\[
   \sum_{\boldx \in \Wr^{-1}(h_{1^n})} \ambsgn(\boldx)
    = \sum_{T \in \SYT(\lambda)} \ambsgn(\MNpoint_T)
    = I_\lambda
\,.
\qedhere
\]
\end{proof}

This also yields a lower bound for the number of real points in
the fibre of the Wronski map.

\begin{corollary}[Eremenko--Gabrielov]
\label{cor:EG-bound}
For every $g \in \monics(\RR)$, we have
\[
   |I_\lambda| \leq N_g
\,.
\]
\end{corollary}

In the case of a rectangle, $\lambda = m^d$, the Schubert cell
$\scell(\RR)$ is an open dense subset of $\Gr(d,\RR^{d+m})$.  If
$m+d$ is odd, then both $\Gr(d,\RR^{d+m})$ and $\projspace^{dm}(\RR)$
are non-orientable.  Eremenko and Gabrielov showed that the real
projective Wronski
map lifts to a map between the oriented double covers of these
spaces, and hence has a well-defined degree up to sign, which is
$\pm I_\lambda$. In this case, $I_\lambda \neq 0$ \cite{Whi}.
However, if $m+d$ is even, $m, d \geq 2$, then $I_\lambda = 0$.
If $m$ and $d$ are both even, this can be attributed to the fact
that $\Gr(d, \RR^{d+m})$ is orientable but $\projspace^{dm}(\RR)$ is not, 
and the degree of the real projective Wronski map cannot be defined.
Example~\ref{ex:tightcases} also shows that the real Wronski map
is not surjective in this case.  But if $m$ and $d$ are both odd,
then both spaces are orientable, and there is no obvious geometric
reason why the topological degree is $0$.

By comparing the lower bounds in Corollaries~\ref{cor:bounds} 
and~\ref{cor:EG-bound}, one can see that in general, neither lower
bound is tight.  We include two examples, which first appeared in
\cite{MT-bound}.

\begin{example}
Let $\lambda = 3^5$.  Then $|\sgchar(\mu)| \geq 6$ for every partition
$\mu$ of the form $2^{n_2}1^{n_1}$. (The minimum value of $|\sgchar(\mu)|$
occurs for $\mu = 2^41^7$.)  Thus, $N_g \geq 6$ for
every $g \in \monicsb{15}(\RR)$.  
As noted above, $I_\lambda = 0$ in this case, which shows that 
the Eremenko--Gabrielov lower bound is not tight.
\end{example}

\begin{example}
Let $\lambda = 3^6$.  Then $I_\lambda = 12$, which means $N_g \geq 12$
for all $g \in \monicsb{18}(\RR)$.  On the other hand,
if $\mu = 2^61^6$ or $\mu=2^71^4$, then $\sgchar(\mu) = 0$, so the
bound from Corollary~\ref{cor:bounds} is not tight.
\end{example}

\subsection{Mukhin and Tarasov's lower bound}
\label{sec:MT-bound}

As mentioned in the introduction, the lower bound in
Corollary~\ref{cor:bounds} is 
a special case of a more general inequality.
Let  $a_1, \dots, a_k$ be distinct real numbers, and let
$b_1, \overline{b_1}, \dots, b_l, \overline{b_l}$ be distinct complex
numbers.
Consider a $0$-dimensional intersection of the form
\begin{equation}
\label{eqn:complexintersection}
   X_{\alpha^1}(a_1) \cap \dots \cap X_{\alpha^k}(a_k)
   \ \cap\  X_{\beta^1}(b_1) \cap \dots \cap X_{\beta^l}(b_l)
   \ \cap\  X_{\beta^1}(\overline{b_1})
   \cap \dots \cap
   X_{\beta^l}(\overline{b_l})
\end{equation}
inside the Grassmannian $\Gr(d,\CC_{d+m-1}[z])$.
(Here, $\alpha^1, \dots, \alpha^k, \beta^1, \dots, \beta^l$ are 
partitions.)
In \cite{MT-bound},
Mukhin and Tarasov gave a lower bound for the number of real points
of such an intersection, which depends only on the discrete data
$(k,l, \alpha^1, \dots, \alpha^k, \beta^1, \dots, \beta^l)$.

In the case where,
$\alpha^1 = \lambda^\vee$, 
$\alpha^2 = \dots = \alpha^k = \beta^1 = \dots = \beta^l = \one$,
and $a_1 = \infty$, the
intersection~\eqref{eqn:complexintersection} is precisely 
$\Wr^{-1}(g)$, where
\[
  g(z) = \prod_{i=2}^k (z+a_i) \cdot 
   \prod_{j=1}^l(z+b_j)(z+\overline{b_j})
\,.
\]
The lower bound in Corollary~\ref{cor:bounds} 
coincides with the Mukhin--Tarasov lower bound in this special case.  
(This equivalence is not completely obvious the way things are 
stated in \cite{MT-bound}, but it is not hard to show using standard 
results in symmetric function theory.)  
The two proofs, however, are completely
different.  Mukhin and Tarasov obtain their inequality using the fact
that the number of real eigenvalues of an operator that is
self-adjoint with respect to an indefinite Hermitian form is
at least the absolute value of the signature of the form.  
The machinery in \cite{MTV2}
identifies the points of intersection \eqref{eqn:complexintersection}
with one-dimensional eigenspaces of an algebra commuting
self-adjoint operators; the point is real if and only if
the associated eigenvalues are real.  Hence, they obtain
a lower bound by computing the signature of the associated
form.  Since the signature of the form depends only on the discrete
data, it is an invariant for the problem.
We do not know why this invariant coincides with the topological
degree of the restricted Wronski map with the character orientation.
It is fairly natural to conjecture that the sign of a real intersection 
point is equal to the signature of the Hermitian form 
restricted to the associated eigenspace.

In the general case, by Lemma \ref{lem:schubertwronskian},
intersection~\eqref{eqn:complexintersection} is contained in the
fibre $\Wr^{-1}(g)$ 
\[
g(z) = 
\prod_{i=1}^k(z+a_i)^{|\alpha^i|} \cdot
\prod_{j=1}^l\big((z+b_j)(z+\overline{b_j})\big)^{|\beta^j|}
\,,
\]
and we might try to use this to count the points of
\eqref{eqn:complexintersection} with signs.
Unfortunately, these signs will not always be defined,
either because $g \notin \monics(\mu)$ for any $\mu$, or because
$g$ is a critical value of the Wronski map.
It is nevertheless possible to give a topological 
proof of the Mukhin--Tarasov lower bound in its full generality,
using Theorem~\ref{thm:main} in combination with results from \cite{Pur-Gr, Pur-ribbon}. 
Although the argument is too long to include here in full detail, we
give a brief sketch of how this can be done.

Let $\lambda = m^d$, so $\scell(\RR)$ is an open dense 
subset of $\Gr(d,\RR^{d+m})$.  The real points of the 
intersection~\eqref{eqn:complexintersection} will be $\scell(\RR)$,
and as noted above, specifically contained in the fibre $\Wr^{-1}(g)$.
Perturb $g \in \monics(\RR)$ to a nearby polynomial $g'$ which is a regular value of the Wronski map with distinct roots, in such a way that the number of real roots does not change. For any point $\boldx$ in the intersection~\eqref{eqn:complexintersection},
we can consider all points $\boldx' \in \Wr^{-1}(g)$ that perturbations
of $\boldx$.   By properties of $g'$ each such $\boldx'$ has a well-defined
sign.  We define the \defn{weight} of $\boldx$ to be the average of the signs
of all $\boldx' \in \Wr^{-1}(g')$ that are perturbations of $\boldx$.
One can show the sum of the weights of the points in the intersection~\eqref{eqn:complexintersection} is an invariant (i.e. it depends only on the discrete data), by relating to the topological degree of a restriction of the Wronski map. Moreover, since the weight of $\boldx$ is a rational number in the interval $[-1,1]$, this invariant gives a lower bound on the number of points in the intersection.

It remains to compute the invariant.  It suffices to compute
it for a single tuple of parameters $(a_1, \dots, a_k, b_1, \dots b_l)$,
and a specific perturbation $g$.
We take $a_i = u^i$, and $b_j = \imag u^{k+j}$, where (at the end of the day)
$u$ will be evaluated at some small parameter $\varepsilon > 0$,
and we perturb the roots of $g$ in such a way that all roots remain
either real or pure imaginary.
Using \cite[Theorem 3.15]{Pur-ribbon} and degeneration arguments 
similar to Section~\ref{sec:stablecurves}, we can label the points
of the intersection by certain equivalence classes of tableaux, and
we can say which of these points are real.  The
tableaux themselves label the perturbed points, and the equivalence
classes tell us which points of $\Wr^{-1}(g')$ are perturbations of
of the same point in $\Wr^{-1}(g)$.  As in Section~\ref{sec:orientation},
we can compute the signs of the signs of the points in $\Wr^{-1}(g')$ from 
the associated tableaux, and hence compute the weights of the points
in~\eqref{eqn:complexintersection}.  (It turns out that, for this fibre, all perturbations
of the same point have the same sign, so the weights are in fact $\pm 1$.)
This gives a combinatorial formula for the sum of the weights of the
points in the intersection.  We note that when there are no complex points
$(l=0)$, this argument reproduces the proof of the Littlewood-Richardson
rule in \cite{Pur-Gr}.  Finally, using known results from symmetric
function theory, one can identify this combinatorial formula as equivalent
to the formula given by Mukhin and Tarasov.

We note that proof of \cite[Theorem 3.15]{Pur-ribbon} referenced above uses 
Theorem~\ref{thm:MTV} in an essential way.  As such, to prove the stronger 
theorem independently of Mukhin--Tarasov--Varchenko's work, one needs to 
establish the weaker Theorem~\ref{thm:main} first.

\section{Concluding remarks}
\label{sec:conclusion}

\subsection{Generalization to Richardson varieties}
\label{sec:richgeneralization}

With a few small changes, Theorem~\ref{thm:main} generalizes to 
the open Richardson variety.

Let $\lambda/\lambda'$ be a skew shape such that $|\lambda/\lambda'| = n$.
The \defn{open Richardson variety} is the intersection of Schubert
cells
$\openrich := \schubertopposite^\circ(\infty) \cap X_{\lambda'}^\circ(0)$.
This is a smooth affine variety.
Let $\richmonics := \{g \in \monics \mid g(0) \neq 0\}$.
The Wronski map induces a finite proper map 
\begin{gather*}
    \Wrich : \openrich \to \richmonics \\
   \boldx \mapsto z^{-|\lambda'|}\,\Wr(\boldx, z)
\,.
\end{gather*}
For a partition $\mu = 2^{n_2}1^{n_1} \vdash n$, 
we define $\richmonics(\mu) := \monics(\mu) \cap \richmonics(\RR)$,
and $\openrich(\mu) := \Wrich^{-1}(\richmonics(\mu))$.  As before,
$\Wrich$ restricts to a proper map 
$\Wrich : \openrich(\mu) \to \richmonics(\mu)$
for each $\mu$.

To proceed further, we use the following technical fact 
about $\openrich$ (see \cite{LP-openrich}).

\begin{lemma}
\label{lem:UFD}
The divisor class group of $\openrich(\RR)$ is trivial.
\end{lemma}

We can now define a character 
orientation function on $\openrich(\RR)$, with the following 
modifications to the discussion in Section~\ref{sec:orientation}. 
We work over $\tspace^1 = \affinespace^1 \setminus\{0\}$
instead of $\affinespace^1$.
For a partition $\kappa$, let $Y^{\lambda/\lambda'}_\kappa
\subset \openrich \times \tspace^1$
be the total space of
the family of 
$X_\kappa(a) \cap \openrich$, $a \in \tspace^1$.
Then define $Z^{\lambda/\lambda'}_\kappa \subset \openrich$ 
to the image of the
projection onto the first factor.  For $|\kappa| = 2$, 
$Z^{\lambda/\lambda'}_\kappa \subset \openrich$ is a hypersurface. 
By Lemma~\ref{lem:UFD}, it is the zero locus of a single real polynomial
function 
$\Phi^{\lambda/\lambda'}_\kappa$.  From here, we can proceed exactly
as in Section~\ref{sec:orientation}, using 
$\Phi^{\lambda/\lambda'}_\vdomino$ to define
the character orientation of $\openrich(\mu)$.

One can then compute the topological degree 
of $\Wrich : \openrich(\mu) \to \richmonics(\mu)$ with respect to
the character orientation.
The entire proof of Theorem~\ref{thm:main} works almost
verbatim if we simply replace $\lambda$ by $\lambda/\lambda'$ throughout.
The biggest modification is in Section~\ref{sec:mfamily}, where
we need to replace the family \eqref{eqn:mfamilydef}
by
\[
   \richmfamily := \mschubertone(a_1) \cap \dots \cap \mschubertone(a_n)
    \cap \mschubertopposite(\infty) \cap \mschubertlambdaprime(0)
\,.
\]
In this way, we obtain the following theorem.

\begin{theorem}
\label{thm:openrichdegree}
With the character orientation,
the topological degree of the map
$\Wrich : \openrich(\mu) \to \richmonics(\mu)$ is equal to
$\skewsgchar(\mu)$. 
\end{theorem}

Here $\skewsgchar$ denotes the skew symmetric
group character (see \cite{GW}).
We therefore also obtain analogues of
Corollaries~\ref{cor:bounds} and~\ref{cor:cover}.  

\begin{corollary}
For $g \in \richmonics(\RR)$, 
let $N_g$ be the number of real points in the 
fibre $\Wrich^{-1}(g)$, counted with algebraic multiplicity.
If $g \in \overline{\richmonics(\mu)}$, then
\[
   |\skewsgchar(\mu)| \leq N_g \leq \numsyt{\lambda/\lambda'}
\,.
\]
\end{corollary}

\begin{corollary}
\label{cor:richcover}
$\Wrich : \openrich(1^n) \to \richmonics(1^n)$ is a topologically trivial 
covering map of degree $\numsyt{\lambda/\lambda'}$.  
\end{corollary}

Corollary~\ref{cor:richcover} can also be 
deduced without Theorem~\ref{thm:openrichdegree}, as 
we explain next, in Section~\ref{sec:transversality}.

\begin{remark}
It is important that $\vdominovar \subset \openrich$ 
is a principal divisor
over $\RR$.  To see why, recall that our goal is to construct an 
orientation of $X(\RR) \setminus V(\RR)$, where $X$ is a smooth affine 
variety
over $\RR$ and $V \subset X$ is a hypersurface, such that the orientation 
reverses along $V(\RR)$.  
If $X(\RR)$ is orientable, this implies the real line bundle 
associated to $V(\RR)$ must also be orientable, which, in our case,
is guaranteed by Lemma~\ref{lem:UFD}.
As an example of what could go wrong, 
consider the circle $X = \Spec \FF[x,y]/ \langle x^2 +y^2-1\rangle$,
and let $V$ be the point $(1,0)$.  
If $\FF = \RR$, $V(\RR)$ is not a principal divisor on $X(\RR)$, 
and clearly we cannot orient 
$X(\RR) \setminus V(\RR)$ in such a way that the orientation reverses 
along $V(\RR)$.  Nevertheless $V(\RR)$ is locally principal, and
its complexification $V(\CC)$ is a principal divisor of $X(\CC)$,
so neither of these criteria is sufficient.
\end{remark}

\subsection{Transversality}
\label{sec:transversality}

Since some of the known applications of the Shapiro--Shapiro conjecture 
rely on some form of transversality or reducedness property, such as 
Corollary~\ref{cor:cover}, we give a brief account of those 
that are known to follow from Theorem~\ref{thm:MTV}.  The strongest 
known transversality theorem for Shapiro-type Schubert intersections 
is Mukhin, Tarasov and Varchenko's theorem in \cite{MTV2}, which 
does not appear to follow easily from Theorem~\ref{thm:MTV}.
However, we do get some important special cases.  These
are deduced via the following lemma (see \cite[Theorem 13.2]{Sot-RSEG}).

\begin{lemma}
\label{lem:realunramified}
Let $\finitemap : X \to Y$ be a finite morphism of smooth 
varieties defined over $\RR$.  
If $\calU \subset Y(\RR)$ is an analytic open subset 
such that every point of $\finitemap^{-1}(\calU)$ is real, 
then $\finitemap$ is unramified over $\calU$.
\end{lemma}

For example, taking $\psi$ to be the Wronski map 
$\Wr : \scell \to \monics$, and
$\calU = \monics(1^n)$ yields Corollary~\ref{cor:cover}.  

We can also take $\psi$ to be the map 
$\Wrich : \openrich \to \richmonics$, and $\calU = \richmonics(1^n)$,
which gives Corollary~\ref{cor:richcover}.
Restated in terms of Schubert intersections, this says the following.

\begin{corollary}
\label{cor:transverse}
Let $\alpha, \beta$ be partitions 
and let $a_1, a_2, \dots, a_{n+2} \in \projspace^1(\RR)$ be distinct,
where $|\alpha| + |\beta| + n = dm$.
Then the intersection
\[
   \schubertone(a_1) \cap \dots \cap \schubertone(a_n)
    \cap
    X_\alpha(a_{n+1}) \cap X_\beta(a_{n+2}) 
\]
in $\Gr(d,d+m)$ is transverse.
\end{corollary}

For the orthogonal Grassmannian $\OG$, there is a finite map to 
projective space which is an 
analogue of the Wronski map.
In~\cite{Pur-OG}, the second author showed that Theorem~\ref{thm:MTV} 
implies an analogous reality theorem for $\OG$, and thus 
Lemma~\ref{lem:realunramified} can also be applied to $\OG$.  
In particular, Theorem~\ref{thm:MTV} 
implies the analogue of Corollary~\ref{cor:transverse} for $\OG$.
It is interesting to note that we do not have an obvious analogue of 
Theorem~\ref{thm:main} for $\OG$, as there is only a single
codimension $2$ Schubert variety in $\OG$.

\subsection{Open questions}

Quite a few generalizations of Theorem~\ref{thm:MTV} have been conjectured,
supported by large quantities of computational evidence.  We will not
give a complete overview of these here, but instead refer the reader
to the discussion in \cite[Ch. 13 \& 14]{Sot-RSEG}.
Although it is not immediately obvious how to generalize 
Theorem~\ref{thm:main} to these other settings, we hope that it
will lead to new ideas toward solving these problems.

Since a large part of the proof of Theorem~\ref{thm:main}
involved the moduli space of stable curves,
it is natural to wonder if there is an a statement analogous to
Theorem~\ref{thm:main} involving the finite family
$\mmap : \mfamily \to \modsc{n+1}$,
instead of the Wronski map.
The semialgebraic sets $\scell(\mu)$ and $\monics(\mu)$ would be replaced
by actual real loci 
$\mfamily(\RR^\sigma)$ and $\modsc{n+1}(\RR^\sigma)$ respectively.
One would then expect the character orientation function
to be replaced by a section of the anticanonical bundle,
which naturally defines an orientation.  This would be very nice, but
there are some difficulties with this idea, not the least of 
which is the fact that $\modsc{n+1}(\RR^\sigma)$ is not always orientable. 

Another possible approach to proving Theorem~\ref{thm:main} would be to
determine the character orientation function $\vdominofcn$ explicitly.
With this, it should be possible to compute the signs of the points 
$\MNpoint_T$ directly, i.e. without using Lemmas~\ref{lem:MNsigns} 
and~\ref{lem:SYTsigns}.  Conceivably, this could yield a more concise
proof.  It would also be interesting to see if the ideas in \cite{MT-bound} 
can be used to obtain an alternate proof of Theorem~\ref{thm:main}.

Finally, the obvious question: is there a similar interpretation of
$\sgchar(\mu)$ if $\mu$ is an arbitrary partition of $n$?  We propose
that the general answer should take the following form.  
If $\sigma \in \symgroup_n$ is a permutation of cycle type $\mu$, 
and $C \in \modsc{n+1}^\sigma$ is a $\sigma$-fixed curve, then 
\[
  \sgchar(\mu) = \sum_{\mfpoint \in \mmap^{-1}(C)^\sigma} \sgn(\mfpoint)
\]
where $\mmap^{-1}(C)^\sigma$ denotes the $\sigma$-fixed subscheme
of the fibre of the map $\mmap : \mfamily \to \modsc{n+1}$, and
$\sgn(\mfpoint) \in \{\pm 1\}$ 
is some geometrically meaningful sign assigned
to $\mfpoint$.  If $C$ is a $\projspace^1$-chain without double marked
points, it is not hard to show
that $\#\mmap^{-1}(C)^\sigma = \#\MN(\lambda;\mu')$ for some composition 
$\mu'$ with associated partition $\mu$, so in this case, there is 
some way to assign signs so that this equation is true.  However, not all
$\sigma$-fixed curves are of this form, and without
some sort of geometric interpretation for the signs, this is 
not particularly deep.  The 
crux of the question is to explain the geometric meaning of 
the signs in the Murnaghan--Nakayama rule.


\bigskip

\footnotesize%
\noindent
\textsc{J. Levinson, Mathematics Department, 
       Simon Fraser University,
       Burnaby, BC, V5A 1S6, Canada.} \texttt{jake\_levinson@sfu.ca}.

\bigskip

\noindent
\textsc{K. Purbhoo, Combinatorics and Optimization Department, 
       University of Waterloo,  
       Waterloo, ON, N2L 3G1, Canada.} \texttt{kpurbhoo@uwaterloo.ca}.


\begin{thebibliography}{00}

\bibitem{BV}
 S. Billey and R. Vakil, \emph{Intersections of Schubert varieties and other permutation array schemes}, Algorithms in algebraic geometry (A. Dickenstein, F.-O. Schreyer, and A. J. Sommese, eds.), IMA Vol. Math. Appl., \textbf{146}, Springer, New York, 2008, pp. 21--54.

\bibitem{CMPB} M. Chan, A. L\'{o}pez Mart{\i}n, N. Pflueger and M. Teixidor i Bigas, \emph{Genera of Brill-Noether curves and staircase paths in Young tableaux}, Trans. Amer. Math. Soc., \textbf{370} (2018), 3405--3439.

\bibitem{Dev} S. Devadoss, \emph{Combinatorial equivalence of real moduli spaces}, Notices Amer. Math. Soc., \textbf{51} (2004) no. 6, 620--628.

\bibitem{EG-deg} A. Eremenko and A. Gabrielov,
\emph{Degrees of real Wronski maps},
Disc. Comp. Geom., \textbf{28} (2002), 331--347.

\bibitem{EG-pole} A. Eremenko and A. Gabrielov,
\emph{Counterexamples to pole placement by static output feedback.}
Linear Algebra Appl., \textbf{351/352} (2002), 211--218.

\bibitem{EH} D. Eisenbud and J. Harris,
\emph{Divisors on general curves and cuspidal rational curves},
Invent. Math., \textbf{74} (1983), 371--418.

\bibitem{FRT} J. M. Frame, G. d. B. Robinson  and R. M. Thrall,
\emph{The hook graphs of the symmetric group},
Canad. J. Math., \textbf{6} (1954), 316--325.

\bibitem{GW}
A. Garsia and M. Wachs,
\emph{Combinatorial aspects of skew representations of the symmetric group}, 
J. Combin. Theory Ser. A, \textbf{50} (1989) no. 1, 47--81.

\bibitem{GL}
M.~Gillespie and J.~Levinson,
\emph{K-theory and monodromy of Schubert curves via generalized jeu de taquin},
J. Alg. Comb., \textbf{45} (2017) no. 1, 191--243.

\bibitem{HKRW}
I. Halacheva, J. Kamnitzer, L. Rybnikov, A. Weekes,
\emph{Crystals and monodromy of Bethe vectors}, preprint (2017), arXiv:1708.05105.

\bibitem{HHMS}
J. Hauenstein, N. Hein, A. Mart\'in del Campo, and F. Sottile,
\emph{Beyond the Shapiro Conjecture and Eremenko-Gabrielov lower bounds},
website (2010).
\texttt{\small 
www.math.tamu.edu/{\footnotesize$\sim$}frank.sottile/research/pages/lower\_Shapiro/}

\bibitem{HHS}
N. Hein, C. Hillar, and F. Sottile,
\emph{Lower Bounds in Real Schubert Calculus},
S\~ao Paulo J. Math. Sci., \textbf{7} (2013) no. 1, 33--58.

\bibitem{Kl}
S.~L. Kleiman, \emph{The transversality of a general translate}, Compositio
Math., \textbf{28} (1974), 287--297.

\bibitem{KS}
V. Kharlamov and F. Sottile,
\emph{Maximally inflected real rational curves},
Mosc. Math. J., \textbf{3} (2003), no. 3, 947--987, 1199--1200.

\bibitem{LS}
A. Leykin and F. Sottile,
\emph{Galois groups of Schubert problems via homotopy continuation}, Math. Comput., \textbf{28} (2009) no. 267, 1749-1765.

\bibitem{LMV}
K. Lu, E. Mukhin and A. Varchenko,
\emph{Self-dual Grassmannian, Wronski map, and representations of $\mathfrak{gl}_N$, $\mathfrak{sp}_{2r}$, $\mathfrak{so}_{2r+1}$},
Pure Appl. Math. Q., \textbf{13} (2017), no. 2, 291--335. 

\bibitem{MT-bound}
E. Mukhin and V. Tarasov,
\emph{Lower bounds for numbers of real solutions 
in problems of Schubert calculus},
Acta Math.,  \textbf{217}, (2016), no. 1,  177--193.

\bibitem{MTV1}
E. Mukhin, V. Tarasov and A. Varchenko,
\emph{The B. and M. Shapiro conjecture in real algebraic geometry
and the Bethe Ansatz}, Ann. Math., \textbf{170} (2009), no. 2, 863--881.

\bibitem{MTV2}
E. Mukhin, V. Tarasov and A. Varchenko,
\emph{Schubert calculus and representations of general linear group},
J. Amer. Math. Soc.,  \textbf{22}  (2009),  no. 4, 909--940.

\bibitem{MTV3}
E. Mukhin, V. Tarasov and A. Varchenko,
\emph{On reality property of Wronski maps},
Confluentes Math., \textbf{1} (2009), no. 2, 225--247.

\bibitem{Lev}
J. Levinson,
\emph{One-dimensional Schubert problems with respect to osculating flags},
Canad. J. Math., \textbf{69} (2016), 143--177.

\bibitem{LP-openrich}
J. Levinson and K. Purbhoo,
\emph{Class groups of open Richardson varieties in the
Grassmannian are trivial}, J. Commut. Algebra, to appear.

\bibitem{Os} B. Osserman, 
\emph{A limit linear series moduli scheme}, 
Ann. Inst. Four., \textbf{56} (2006) no. 4, 1165--1205.

\bibitem{Pur-Gr}
K. Purbhoo,
\emph{Jeu de taquin and a monodromy problem for Wronksians
of polynomials}, Adv. Math.,  \textbf{224} (2010) no. 3, 827--862.

\bibitem{Pur-OG}
K. Purbhoo,
\emph{Reality and transversality for Schubert calculus in
$\mathrm{OG}(n,2n{+}1)$},
Math. Res. Lett., \textbf{17} (2010) no. 6, 1041--1046.

\bibitem{Pur-ribbon}
K. Purbhoo,
\emph{Wronskians, cyclic group actions, and ribbon tableaux},
Trans. Amer. Math. Soc., \textbf{365} (2013), 1977--2030.

\bibitem{Pur-LG}
K. Purbhoo,
\emph{A marvellous embedding of the Lagrangian Grassmannian},
J. Combin. Theory Ser. A, \textbf{155} (2018), 1--26. 

\bibitem{RS}
J. Rosenthal and F. Sottile,
\emph{Some remarks on real and complex output feedback},
Sys. \& Control Lett., \textbf{33} (1998), 73--80.

\bibitem{Ry}
L. Rybnikov, \emph{Cactus group and monodromy of Bethe vectors}. Int. Math. Res. Not., \textbf{2018} (2016) no. 1, 202--235.

\bibitem{Sp}
D. Speyer,
\emph{Schubert problems with respect to osculating flags
of stable rational curves},
Alg. Geom., \textbf{1} (2014) no. 1, 14--45.

\bibitem{Sot-F}
F. Sottile,
\emph{Frontiers of reality in Schubert calculus},
Bull. Amer. Math. Soc., \textbf{47}  (2010),  no. 1, 31--71.

\bibitem{Sot-RSEG}
F. Sottile,
\emph{Real Solutions to Equations From Geometry,}
University Lecture Series, \textbf{57}, AMS, 2011.

\bibitem{SS}
F. Sottile and E. Soprunova,
\emph{Lower Bounds for Real Solutions to Sparse Polynomial Systems},
Adv. Math., \textbf{204}, (2006), no. 1, 116--151. 

\bibitem{SW}
F. Sottile and J. White, \emph{Double transitivity of Galois groups in Schubert calculus of Grassmannians}, Alg. Geom., \textbf{2} (2015), no. 4, 422--445.

\bibitem{Vak} R. Vakil,
\emph{Schubert induction},
Ann. Math., \textbf{164} (2006), 489--512. 

\bibitem{Whi} D. White, 
\emph{Sign-balanced posets}, 
J. Combin. Theory, \textbf{95} (2001), 1--38.

\bibitem{Whi2} N. White, \emph{The monodromy of real Bethe vectors for the Gaudin model}, preprint (2015), arXiv:1511.04740.

\end{thebibliography}
\end{document}